\documentclass[a4paper,12pt]{article}
\usepackage{latexsym}
\usepackage{amsmath}
\usepackage{amssymb}
\usepackage{enumerate}
\usepackage{mathrsfs}
\usepackage{color}

\usepackage{theorem}
\newtheorem{theorem}{Theorem}[section]
\newtheorem{proposition}[theorem]{Proposition}
\newtheorem{lemma}[theorem]{Lemma}
\newtheorem{corollary}[theorem]{Corollary}
\newtheorem{claim}[theorem]{Claim}

\theorembodyfont{\rmfamily}
\newtheorem{proof}{\textmd{\textit{Proof.}}}

\newtheorem{remark}[theorem]{Remark}
\newtheorem{example}[theorem]{Example}
\newtheorem{definition}[theorem]{Definition}

\makeatletter

\@addtoreset{equation}{section}
\makeatother

\newcommand{\qedd}{\hfill \Box}
\newcommand{\ve}{\varepsilon}
\newcommand{\del}{\partial}
\newcommand{\lra}{\longrightarrow}
\renewcommand{\det}{\ensuremath{\mathrm{det}}}
\newcommand{\grad}{\nabla\!\!_-}

\newcommand{\lr}[2]{\langle{#1}{,#2}\rangle}

\newcommand{\N}{\ensuremath{\mathbb{N}}}
\newcommand{\R}{\ensuremath{\mathbb{R}}}

\newcommand{\cF}{\ensuremath{\mathcal{F}}}
\newcommand{\cL}{\ensuremath{\mathcal{L}}} 

\newcommand{\cP}{\ensuremath{\mathcal{P}}}
\newcommand{\cT}{\ensuremath{\mathcal{T}}}
\newcommand{\DC}{\ensuremath{\mathcal{DC}}}
\newcommand{\cM}{\ensuremath{\mathcal{M}}}
\newcommand{\bs}{\ensuremath{\mathbf{s}}}
\newcommand{\bv}{\ensuremath{\mathbf{v}}}
\newcommand{\bJ}{\ensuremath{\mathbf{J}}}
\newcommand{\CD}{\ensuremath{\mathsf{CD}}}

\def\trace{\mathop{\mathrm{tr}}\nolimits}
\def\diam{\mathop{\mathrm{diam}}\nolimits}
\def\vol{\mathop{\mathrm{vol}}\nolimits}
\def\area{\mathop{\mathrm{area}}\nolimits}
\def\div{\mathop{\mathrm{div}}\nolimits}
\def\Hess{\mathop{\mathrm{Hess}}\nolimits}
\def\supp{\mathop{\mathrm{supp}}\nolimits}
\def\Ent{\mathop{\mathrm{Ent}}\nolimits}
\def\Ric{\mathop{\mathrm{Ric}}\nolimits}

\def\loc{\mathop{\mathrm{loc}}\nolimits}
\def\ac{\mathop{\mathrm{ac}}\nolimits}
\def\ev{\mathop{\mathrm{ev}}\nolimits}
\def\diag{\mathop{\mathrm{diag}}\nolimits}

\title{Displacement convexity\\ of generalized relative entropies.~II}
\author{
Shin-ichi Ohta\thanks{Department of Mathematics, Kyoto University, Kyoto 606-8502,
Japan ({\sf sohta@math.kyoto-u.ac.jp});
Supported in part by the Grant-in-Aid for Young Scientists (B) 23740048.}
and 
Asuka Takatsu\thanks{Graduate School of Mathematics, Nagoya University, Nagoya 464-8602,
Japan ({\sf takatsu@math.nagoya-u.ac.jp}) ;
Supported in part by the Grant-in-Aid for Young Scientists (B) 24740042.}}
\date{\today}
\begin{document}

\maketitle

\begin{abstract}
We introduce a class of generalized relative entropies
(inspired by the Bregman divergence in information theory)
on the Wasserstein space over a weighted Riemannian or Finsler manifold.
We prove that the convexity of all the entropies in this class is equivalent to
the combination of the nonnegative weighted Ricci curvature and
the convexity of another weight function used in the definition of the generalized relative entropies.
This convexity condition corresponds to Lott and Villani's version of the curvature-dimension condition.
As applications, we obtain appropriate variants of the Talagrand, HWI
and logarithmic Sobolev inequalities, as well as the concentration of measures.
We also investigate the gradient flow of our generalized relative entropy.
%\bigskip
%
%\noindent MSC (2000): 53C21, 58J35 (Primary); 28A33, 49Q20 (Secondary) \\
%Keywords: Bregman divergence, Ricci curvature, optimal transport,
%curvature-dimension condition, gradient flow
\end{abstract}

\tableofcontents

\section{Introduction}\label{sc:intro}%%%%%%%%%

This is a continuation of our work~\cite{mCD} on the displacement convexity
of generalized entropies and its applications.
We consider more general entropies than \cite{mCD} and generalize most results in appropriate ways.
Some of our observation shall shed new light on \cite{mCD}.

It has been known since the celebrated work of McCann~\cite{Mc1} 
that the convexity of an energy (entropy) functional along geodesics in the Wasserstein space
plays a vital role in the study of the existence and the uniqueness of a ground state
(a minimizer of the energy).
Here the (quadratic) \emph{Wasserstein space} over a complete separable metric space $(X,d)$
is the space $\cP^2(X)$ of Borel probability measures on $X$ having finite second moments, endowed with
the \emph{Wasserstein distance function} $W_2$ derived from the Monge--Kantorovich mass transport problem
(see Subsection~\ref{ssc:wass}).
We say that a functional $S$ on $\cP^2(X)$ is \emph{displacement $K$-convex} for $K \in \R$
($\Hess S \ge K$ for short) if any pair $\mu_0, \mu_1 \in \cP^2(X)$ can be joined by
a minimal geodesic $(\mu_t)_{t \in [0,1]}$ in $(\cP^2(X),W_2)$ such that
\[ S(\mu_t) \le (1-t) S(\mu_0) +t S(\mu_1) -\frac{K}{2}(1-t)t W_2(\mu_0,\mu_1)^2 \]
holds for all $t \in [0,1]$.
As usual, the displacement $0$-convexity may be simply called the \emph{displacement convexity}.
The word `displacement' is inserted for avoiding a confusion with the convexity
along the linear interpolation $S((1-t)\mu_0 +t\mu_1) \le (1-t)S(\mu_0)+tS(\mu_1)$.
Since we deal with only the displacement convexity,
we may sometimes omit `displacement'.

As any geodesic in the Wasserstein space is written as the transport along geodesics
in the underlying metric space, the displacement convexity of an energy functional
can be derived from the convexity of its generating function.
For instance, let $S_u^\Psi$ be the \emph{free energy functional} on $\cP^2(\R^n)$
consisting of the internal energy and the potential energy as
\[ S_u^{\Psi}(\mu) :=\int_{\R^n} u \left( \frac{d\mu}{d\cL^n} \right) \,d\cL^n +\int_{\R^n} \Psi \,d\mu \]
for absolutely continuous probability measures $\mu$ on $\R^n$
with respect to the Lebesgue measure $\cL^n$,
where the energy density $u$ is a function on $\R$ and the potential $\Psi$ is a function on $\R^n$.
Then $S_u^\Psi$ is strictly displacement convex
if $u$ is convex (and satisfies certain additional conditions to be precise, see Definition~\ref{df:DC_N})
and $\Psi$ is strictly convex,
and the unique ground state $\nu:=\sigma\cL^n$ satisfies $u'(\sigma)=-\Psi +\lambda$ with
a normalizing constant $\lambda$.
We mention that the uniqueness is measured at the level of the energy functional, that is,
we have $S_u^{\Psi}(\mu) -S_u^{\Psi}(\nu) \ge 0$
and equality holds if and only if $\mu=\nu$.
Moreover, the displacement convexity of the free energy $S_u^\Psi$ is a crucial tool also in the investigation of
the asymptotic behavior of the solution to the associated evolution equation
\[ \frac{\del \rho}{\del t} =\div \!\big(\rho \nabla [u'(\rho)] +\rho \nabla \Psi \big) \]
by regarding it as the gradient flow of $S_u^{\Psi}$ in the Wasserstein space
(see~\cite{JKO}, \cite{AGS}, \cite{CMV1} and \cite{CMV2} among others).
In particular, the heat flow is regarded as the gradient flow of the \emph{relative entropy}
(with respect to the Lebesgue measure)
\[ \Ent_{\cL^n}(\mu) :=\int_{\R^n} \frac{d\mu}{d\cL^n} \ln \left(\frac{d\mu}{d\cL^n}\right) \,d\cL^n, \]
which is also called the \emph{Kullback--Leibler divergence} in information theory.
\medskip

On curved spaces such as Riemannian manifolds, the displacement convexity of energy functionals
is related to the curvature of the underlying space,
that is a crucial difference from the convexity along linear interpolations $(1-t)\mu_0+t\mu_1$.
On a Riemannian manifold equipped with the Riemannian volume measure $\vol_g$,
the relative entropy is similarly defined by
\[ \Ent_{\vol_g}(\mu):=\int_M \frac{d\mu}{d\!\vol_g} \ln \left(\frac{d\mu}{d\!\vol_g}\right) \,d\!\vol_g. \]
It has been shown by von Renesse and Sturm~\cite{vRS} (inspired by \cite{CMS} and \cite{OV})
that for any $K \in \R$ the following are mutually equivalent:
\begin{itemize}
\item The relative entropy $\Ent_{\vol_g}$ is displacement $K$-convex on $(\cP^2(M),W_2)$.

\item The Ricci curvature is bounded from below by $K$ as
$\Ric_g(\bv,\bv) \ge K\langle \bv,\bv \rangle$ for all $\bv \in TM$.

\item The heat flow is \emph{$K$-contractive} in the sense that
$W_2(\mu_t,\tilde{\mu}_t) \le e^{-Kt} W_2(\mu_0,\tilde{\mu}_0)$ holds
for all $t \ge 0$ and for any weak solutions $(\rho_t)_{t \ge 0}$, $(\tilde{\rho}_t)_{t \ge 0}$
to the heat equation $\del \rho/\del t=\Delta \rho$ such that
$\mu_t:=\rho_t \vol_g, \tilde{\mu}_t:=\tilde{\rho}_t \vol_g \in \cP^2(M)$.
\end{itemize}
See also \cite{AGS}, \cite{AGS3}, \cite{Ogra}, \cite{Sa} and \cite[Chapter~23]{Vi2} for the connection
between the $K$-convexity of the functional and the $K$-contraction property of its gradient flow.

The displacement $K$-convexity of $\Ent_{\vol_g}$ ($\Hess \Ent_{\vol_g} \ge K$) is called
the \emph{curvature-dimension condition} $\CD(K,\infty)$ after Bakry and \'Emery's pioneering work~\cite{BE}.
One remarkable point of $\CD(K,\infty)$ is that it can be formulated on general metric measure spaces
without any differentiable (manifold) structure.
Such metric measure spaces \emph{with Ricci curvature bounded below} are independently investigated
by Sturm~\cite{StI} and Lott and Villani~\cite{LV2},
and known to enjoy several properties common to Riemannian manifolds of $\Ric_g \ge K$.
For example, as was indicated by Otto and Villani~\cite{OV},
$\CD(K,\infty)$ with $K>0$ implies various functional inequalities such as the \emph{Talagrand inequality},
the \emph{HWI inequality}, the \emph{logarithmic Sobolev inequality}
and the \emph{global Poincar\'e inequality} (\cite[Section~6]{LV2}).

The curvature-dimension condition $\CD(K,\infty)$ is generalized to $\CD(K,N)$
for each $K \in \R$ and $N \in (1,\infty]$.
On an $n$-dimensional complete connected Riemannian manifold $(M,g)$ of $n \ge 2$
equipped with a weighted measure $\omega=e^{-f}\vol_g$ with $f \in C^{\infty}(M)$,
the condition $\CD(K,N)$ is known to be equivalent to the lower bound of the $N$-Ricci curvature
$\Ric_N(\bv,\bv) \ge K\langle \bv,\bv \rangle$ (\cite{Stcon}, \cite{StII}, \cite{LV1}, \cite{LV2},
see Definition~\ref{df:Ric_N} for the definition of $\Ric_N$).
In particular, an unweighted Riemannian manifold $(M,\vol_g)$ satisfies $\CD(K,N)$ if and only if
its Ricci curvature is bounded below by $K$ and its dimension is bounded above by $N$.
We remark that Sturm's and Lott and Villani's definitions of the curvature-dimension condition
are slightly different, though they are equivalent on non-branching spaces
such as Riemannian or Finsler manifolds.
In both cases it is a certain convexity condition of a class of entropies, and Lott and Villani's class
is larger than Sturm's one.

On non-branching metric measure spaces,
the condition $\CD(0,N)$ for $N \in [n,\infty)$ is equivalent to the displacement convexity of
the \emph{R\'enyi entropy}
\[ S_N(\mu):=-\int_M \left( \frac{d\mu}{d\omega} \right)^{(N-1)/N} \,d\omega. \]
For $K \neq 0$, however, $\CD(K,N)$ is not simply the displacement $K$-convexity of $S_N$.
In fact, it was shown in \cite{Stcon} (see also \cite[Theorem~4.1, Remark~4.3(2)]{mCD} and \cite{BS}) that,
on a weighted Riemannian manifold $(M,\omega)$, $\Hess S_N \ge K$ can hold only for $K \le 0$
and is equivalent to $\Ric_N \ge 0$ regardless of the value of $K \le 0$.
It was also observed in \cite[Theorem~1.7]{Stcon} for unweighted Riemannian manifolds
that there are some functionals whose displacement $K$-convexity characterizes
the combination of $\Ric \ge K$ and $\dim \le N$, whereas it is unclear if there are any applications
of these entropies.
\medskip

In our previous work~\cite{mCD}, we introduced the \emph{$m$-relative entropy}
$H_m$ for the parameter $m \in [(n-1)/n,1) \cup (1,\infty)$ inspired by the Bregman divergence
in information theory/geometry (see \cite{Am}, \cite{AN})
as well as the Tsallis entropy in statistical mechanics (see \cite{Ts1}, \cite{Ts2}).
We fix a reference measure $\nu=\exp_m(-\Psi) \omega$ on a weighted Riemannian manifold $(M,\omega)$
involving the \emph{$m$-exponential function}
\[ \exp_m(t):=\max\{ 1+(m-1)t,0 \}^{1/(m-1)}, \]
then the $m$-relative entropy of an absolutely continuous measure $\mu \in\cP^2(M)$
with respect to $\nu$ is given by (up to an additive constant)
\[ H_m(\mu):=\frac{1}{m(m-1)} \int_M \left\{ \bigg( \frac{d\mu}{d\omega} \bigg)^m
 -m \frac{d\mu}{d\omega} \bigg( \frac{d\nu}{d\omega} \bigg)^{m-1} \right\} \,d\omega. \]
This generalizes the relative and the R\'enyi entropies in the sense that
$\lim_{m \to 1}H_m(\mu)=\Ent_{\nu}(\mu)-1$ and that $H_m(\mu)=N\{m^{-1}S_N(\mu)+1\}$
with $N=1/(1-m)$ if $\Psi \equiv 0$ (i.e., $\nu=\omega$).

Then the displacement $K$-convexity of $H_m$ is equivalent to the combination of $\Ric_N \ge 0$
(of $(M,\omega)$) and $\Hess\Psi \ge K$ (\cite[Theorem~4.1]{mCD}).
We stress that $N$ becomes \emph{negative} for $m>1$, then $\Ric_N$ is defined
in the same form as the case of $N \in (n,\infty)$ (see Definition~\ref{df:Ric_N}).
Similarly to $\CD(K,\infty)$, we can derive from $\Hess H_m \ge K>0$ the associated functional inequalities
(see also \cite{AGK}, \cite{CGH}, \cite{Ta1} for related works)
and the concentration of measures (in terms of $\exp_m$).
Furthermore, the gradient flow of $H_m$ produces weak solutions to the fast diffusion equation ($m<1$)
or the porous medium equation ($m>1$) with drift of the form
\[ \frac{\del \rho}{\del t} =\div_{\omega}\left( \frac{1}{m}\nabla(\rho^m) +\rho\nabla\Psi \right), \]
where $\div_{\omega}$ is the divergence of $(M,\omega)$ (see also \cite{Ot}, \cite[Theorem~23.19]{Vi2}).
We remark that Sturm~\cite{Stcon} studied a more general class of entropies on unweighted
Riemannian manifolds, where $\Ric_N =\Ric$ for all $N$.
Compared to it, \cite{mCD} gave a detailed investigation of a concrete class of entropies,
on more general weighted Riemannian manifolds (by choosing appropriate parameters $N$).
\medskip

In this article, we introduce the more general class of entropies,
called the \emph{$\varphi$-relative entropies} $H_{\varphi}$, again inspired by information theory/geometry.
Here $\varphi:(0,\infty) \lra (0,\infty)$ is a non-decreasing, positive, continuous function.
Roughly speaking, our new class corresponds to Lott and Villani's class of entropies
in their definition of the curvature-dimension condition, while the $m$-relative entropies in \cite{mCD}
correspond to Sturm's class.
The definition of $H_{\varphi}$ (see Definition~\ref{df:Hm} for details)
involves $\nu=\exp_{\varphi}(-\Psi)$ with the \emph{$\varphi$-exponential function}
$\exp_{\varphi}$ which is the inverse function of the \emph{$\varphi$-logarithmic function}
$\ln_{\varphi}(t):=\int_1^t \varphi(s)^{-1} \,ds$.
We recover $\exp_m$ and $H_m$ from $\varphi(s)=s^{2-m}$.

Our first main theorem (Theorem~\ref{th:mCD}) asserts that $\Hess H_m \ge K$
is equivalent to $\Hess H_{\varphi} \ge K$ for all $\varphi$'s in a certain class.
This actually corresponds to the equivalence between Sturm's and Lott and Villani's
curvature-dimension conditions on weighted Riemannian manifolds.
This reveals that $H_m$ is an extremal element among $H_{\varphi}$'s
in the appropriate class, see \cite{Tphi} for a related work.
Similarly to $H_m$,
we can derive from $\Hess H_{\varphi} \ge K>0$ the variants of the Talagrand, HWI, logarithmic Sobolev,
and global Poincar\'e inequalities (Theorem~\ref{thm:inequ}) as well as the concentration
of measures in terms of $\exp_m$ for some $m=m(\varphi)$ (Theorem~\ref{thm:conc}).
Moreover, the gradient flow of $H_{\varphi}$ in $(\cP^2(M),W_2)$ produces weak solutions
to the \emph{$\varphi$-heat equation} (Theorems~\ref{th:gf}, \ref{th:gf+})
\[ \frac{\del \rho}{\del t}=\div_{\omega} \left( \frac{\rho\nabla\rho}{\varphi(\rho)}+\rho\nabla\Psi \right). \]

The article is organized as follows:
We first review the basic notions of weighted Riemannian geometry, Wasserstein geometry
and information geometry in Section~\ref{sc:pre}.
Then, after preparing necessary notions in Sections~\ref{sc:DC}, \ref{sc:ad},
we define $H_{\varphi}$ and study its displacement convexity in Section~\ref{sc:ent}.
Section~\ref{sc:func} is devoted to the functional inequalities and
Section~\ref{sc:conc} is concerned with the concentration of measures.
The gradient flow of $H_\varphi$ is studied in Sections~\ref{sc:gf}, \ref{sc:gf+}
in the compact and noncompact cases, respectively.
We extend most results to Finsler manifolds in Section~\ref{sc:Fins}.
Finally in Appendix, we compare our concentration of measures derived from the generalized
Talagrand inequality with the Herbst-type argument deriving the concentration
from the $u_{\varphi}$-entropy inequality, which is a generalization of
the logarithmic Sobolev inequality different from ours.

%%%%%%%%%%%%%%%%%%%%%%%%%%
\section{Preliminaries}\label{sc:pre}%%%%%%%%%%

\subsection{Weighted Riemannian manifolds}\label{ssc:weight}%%%%%

Throughout the article except Section~\ref{sc:Fins},
$(M,g)$ will be an $n$-dimensional complete connected Riemannian manifold without boundary.
As we are interested in the role of the curvature, we will always assume $n \ge 2$.
Denote by $d_g$ and $\vol_g$ the Riemannian distance function and the Riemannian volume measure of $(M,g)$.
We fix an arbitrary measure
\[ \omega=e^{-f}\vol_g, \quad\ f \in C^{\infty}(M), \]
as our base measure.
To control the behavior of $\omega$, we modify the Ricci curvature $\Ric_g$ of $(M,g)$ as follows.

\begin{definition}[Weighted Ricci curvature]\label{df:Ric_N}
Given $N \in (-\infty,0) \cup [n,\infty]$, we define the \emph{$N$-Ricci curvature tensor}
of $(M,\omega)$ by
\[ \Ric_{N}:=
\begin{cases}
\displaystyle\Ric_g+\Hess_g f  & \text{if\ } N=\infty, \\
\displaystyle\Ric_g+\Hess_g f-\frac{Df \otimes Df}{N-n}  & \text{if\ } N \in (-\infty,0) \cup (n,\infty), \\
\displaystyle\Ric_g+\Hess_g f -\infty \cdot(Df \otimes Df) & \text{if\ } N =n, \\
\end{cases} \]
where by convention $\infty\cdot0=0$.
\end{definition}

We set $\Ric_N(\bv):=\Ric_N(\bv,\bv)$ and will say that $\Ric_N \ge K$ holds for some $K \in \R$
if $\Ric_N(\bv) \ge K\langle \bv,\bv \rangle$ for every $\bv \in TM$.

\begin{remark}\label{rm:Ric_N}
The tensor $\Ric_N$ was usually considered only for $N \in [n,\infty]$,
and then the monotonicity
$\Ric_N(\bv) \le \Ric_{N'}(\bv)$ for $N<N'$
clearly holds.
Note that $\Ric_{\infty}$ is the famous \emph{Bakry--\'Emery tensor}
and $\Ric_N$ for $N \in (n,\infty)$ was introduced by Qian (see \cite{BE}, \cite{Qi} and \cite{Lo} as well).
Extending the range of $N$ to $(-\infty,0) \cup [n,\infty]$ violates the above monotonicity in $N$,
however, observe that $\Ric_N$ is non-decreasing in the parameter
\[ m:=1-\frac{1}{N} \in \bigg[ 1-\frac{1}{n},\infty \bigg),
 \quad \text{where}\ m:=1\ \text{if}\ N=\infty. \]
This observation will be helpful for understanding the validity of Theorem~\ref{th:mCD} below.
\end{remark}

Note that, if $(M,\omega)$ satisfies $\Ric_N \ge K$ for some $K \in \R$ and $N \in [n, \infty)$,
then it behaves like a Riemannian manifold with dimension bounded above by $N$
and Ricci curvature bounded below by $K$ (see \cite{Qi}, \cite{Lo},
as well as \cite{StI}, \cite{StII}, \cite{LV1}, \cite{LV2}, \cite[Part~III]{Vi2} related to the curvature-dimension condition).
For example, the following area growth inequality of Bishop type
(numerically extended to non-integer $N$'s) holds.
Denote by $\area_{\omega}[S(x_0,r)]$ the area of the sphere
$S(x_0,r):=\{ x \in M \,|\, d_g(x_0,x)=r \}$
with respect to $\omega$.

\begin{theorem}{\rm (\cite{Qi}, \cite[Theorem~2.3]{StII})}\label{thm:BG}
If $(M,\omega)$ satisfies $\Ric_N \ge 0$ for some $N \in[n,\infty)$, then  
\[ \area_{\omega}[S(x_0,R)] \le \area_{\omega}[S(x_0,r)] \cdot \left( \frac{R}{r} \right)^{N-1} \]
holds for any $0<r<R$ and $x_0 \in M$.
\end{theorem}

For $N=\infty$, we have the following global estimate.

\begin{theorem} {\rm (\cite[Theorem~18.12]{Vi2})} \label{thm:mBG}
Under the nonnegativity of $\Ric_{\infty}$ of $(M,\omega)$, 
\[
\int_M \exp \left( -\lambda d_g(x_0,x)^2 \right) d\omega(x) <\infty
\]
holds for any $\lambda>0$ and $x_0\in M$.
\end{theorem}

Though Theorems~\ref{thm:BG}, \ref{thm:mBG} are generalized to $\Ric_N \ge K$ for $K \neq 0$,
we will need only the above special cases.

%%%%%%%%%%%%%%%%%%%%%%%%%%%%%%%%%%%
\subsection{Wasserstein geometry}\label{ssc:wass}%%%%%%%%%%%%%

Let us recall some basic notions and facts in optimal transport theory and Wasserstein geometry.
See \cite{AGS}, \cite{Vi1} and \cite{Vi2} for details and more information.

Let $(X,d)$ be a metric space.
A rectifiable curve $\gamma:[0,1] \lra X$ is called a {\it geodesic} if it is locally minimizing and has a constant speed.
We say that $\gamma$ is {\it minimal} if it is globally minimizing, namely
$d(\gamma(s),\gamma(t))=|s-t|d(\gamma(0),\gamma(1))$ holds for all $s,t \in [0,1]$.
A subset $Y$ of $X$ is said to be \emph{totally convex} if, for any $x,y \in Y$,
any minimal geodesic in $X$ from $x$ to $y$ is contained in $Y$.

For a complete Riemannian manifold $(M,g)$, let $\cP(M)$ be the set of all Borel probability measures on $M$.
Given $\mu\in\cP(M)$ and a measurable map $\cT:M\lra M$, the \emph{push forward measure}
$\cT_{\sharp}\mu$ of $\mu$ through $\cT$ is defined by
$\cT_{\sharp}\mu[B]:=\mu[\cT^{-1}(B)]$ for all Borel sets $B \subset M$.
For each $p \in [1,\infty)$, denote by $\cP^p(M) \subset \cP(M)$ the subset consisting of
measures $\mu$ of finite $p$-th moments, that is,
$\int_M d_g(x_0,x)^p \,d\mu(x) <\infty$
for some (and hence all) $x_0 \in M$.

For $\mu,\nu \in \cP(M)$, a probability measure $\pi \in \cP(M \times M)$ is called a \emph{coupling}
of $\mu$ and $\nu$ if its projections are $\mu$ and $\nu$, namely
$\pi[B \times M]=\mu[B]$ and $\pi[M \times B]=\nu[B]$ hold for any Borel set $B \subset M$.
We define the \emph{$L^p$-Wasserstein distance} between $\mu,\nu \in \cP^p(M)$ by
\[ W_p(\mu,\nu):=\inf\left\{ \left( \int_{M \times M} d_g(x,y)^p \,d\pi(x,y) \right)^{1/p} \Biggm|
 \text{$\pi$: a coupling of $\mu$ and $\nu$} \right\}. \]
A coupling $\pi$ is said to be \emph{optimal} if it attains the infimum above.
The function $W_p$ is indeed a distance function on $\cP^p(M)$.
The metric space $(\cP^p(M),W_p)$ is complete, separable and called
the \emph{$L^p$-Wasserstein space} over $M$.
The Wasserstein space inherits several properties of $M$.
For instance, if $M$ is compact, then $(\cP^p(M),W_p)$ is also compact
and the topology induced from $W_p$ coincides with the weak topology.
We will mainly consider the quadratic case $p=2$, and then we omit `$L^2$-' and simply call
$W_2$ and $(\cP^2(M),W_2)$ the Wasserstein distance function and the Wasserstein space.

In view of optimal transport theory, $W_2(\mu_0,\mu_1)^2$ is regarded as the least cost
of transporting $\mu_0$ to $\mu_1$, where the cost of transporting a unit mass
from $x$ to $y$ is $d_g(x,y)^2$.
A minimal geodesic $(\mu_t)_{t \in [0,1]}$ with respect to $W_2$
is then also called the \emph{optimal transport} from $\mu_0$ to $\mu_1$,
and it can be described by using a family of minimal geodesics in the underlying space $M$.
We denote by $\Gamma(M)$ the set of all minimal geodesics $\gamma:[0,1] \lra M$
endowed with the uniform topology induced from the distance function 
$d_{\Gamma(M)}(\gamma,\eta):=\sup_{t \in [0,1]}d_g(\gamma(t),\eta(t))$.
For $t \in [0,1]$, the \emph{evaluation map} $\ev_t:\Gamma(M) \lra M$
is defined by $\ev_t(\gamma):=\gamma(t)$, which is clearly $1$-Lipschitz.

\begin{proposition}{\rm (\cite[Proposition~2.10]{LV2}, \cite[Corollary~7.22]{Vi2})}\label{pr:LV}
Given any minimal geodesic $(\mu_t)_{t \in [0,1]} \subset \cP^2(M)$,
there exists $\Pi \in \cP(\Gamma(M))$ such that $(\ev_t)_{\sharp}\Pi=\mu_t$ for all $t \in [0,1]$
and that $(\ev_0 \times \ev_1)_{\sharp}\Pi$ is an optimal coupling of $\mu_0$ and $\mu_1$.

In particular, for any totally convex set $X$ of $(M,d_g)$, $\cP^2(X)$ is also totally convex in $(\cP^2(M),W_2)$.
\end{proposition}

If one of $\mu_0$ and $\mu_1$ is absolutely continuous with respect to $\vol_g$,
then a more precise description of a minimal geodesic $(\mu_t)_{t \in [0,1]}$
is obtained via the gradient vector field of a \emph{locally semi-convex} function $\phi$
(i.e., every point $x \in M$ admits a neighborhood on which $\phi$ is $K$-convex
in the weak sense for some $K \in \R$, see Definition~\ref{df:K-conv}).
For a measure $\nu$ on $M$, we denote by $\cP_{\ac}(M,\nu) \subset \cP(M)$ the subset
of absolutely continuous measures with respect to $\nu$.
We also set $\cP^2_{\ac}(M,\nu):=\cP^2(M) \cap \cP_{\ac}(M,\nu)$.

\begin{theorem}{\rm (\cite[Theorem~1]{FG})}\label{th:FG}
Given any $\mu_0 \in \cP^2_{\ac}(M,\vol_g)$ and $\mu_1 \in \cP^2(M)$,
there exists a locally semi-convex function $\phi:\Omega \lra \R$ on an open set $\Omega \subset M$
with $\mu_0[\Omega]=1$ such that the map $\cT_t(x):=\exp_x(t\nabla\phi(x))$, $t \in [0,1]$,
provides a unique minimal geodesic from $\mu_0$ to $\mu_1$.
Precisely, $(\cT_0 \times \cT_1)_{\sharp}\mu_0$ is a unique optimal coupling of $\mu_0$ and $\mu_1$,
and $\mu_t:=(\cT_t)_{\sharp}\mu_0$ is a unique minimal geodesic from $\mu_0$ to $\mu_1$ with respect to $W_2$.
\end{theorem}

If $M$ is compact, then the above theorem is due to McCann's celebrated work~\cite{Mc2}
and we can take as the potential function $-\phi$ a \emph{$c$-concave function} for the cost $c(x,y)=d_g(x,y)^2/2$.
(We do not give the definition of the $c$-concave function, what we need is
only the fact that $c$-concave functions are locally semi-convex.)
A locally semi-convex function is locally Lipschitz and twice differentiable almost everywhere by the Alexandrov--Bangert theorem.
Thus $\cT_t$ is differentiable $\mu_0$-a.e.\ and
the following \emph{Jacobian} (or \emph{Monge--Amper\`e}) \emph{equation} holds.

\begin{theorem}{\rm (\cite[Theorems~8.7, 11.1]{Vi2})}\label{th:MA}
Under the same assumptions as Theorem~$\ref{th:FG}$ above,
we have $\mu_t \in \cP^2_{\ac}(M,\vol_g)$ for all $t \in [0,1)$.
Moreover, by putting
\[ \rho_t\omega:=\mu_t=(\cT_t)_{\sharp}\mu, \qquad
\bJ_t^{\omega}(x):=e^{f(x)-f(\cT_t(x))} \det\big( D\cT_t(x) \big), \]
we have $\rho_t(\cT_t(x)) \bJ^{\omega}_t(x) =\rho_0(x)$ and $\bJ^{\omega}_t(x) >0$
for all $t \in [0,1)$ at $\mu_0$-a.e.\ $x \in \Omega$.
In the case of $\nu \in \cP_{\ac}^2(M,\vol_g)$, the above assertions hold also at $t=1$.
\end{theorem}

Note that $\bJ^{\omega}_t$ should be understood as the Jacobian with respect to $\omega$,
and its behavior is naturally controlled by the weighted Ricci curvature.
This is a fundamental geometric intuition behind the curvature-dimension condition
(see Section~\ref{sc:ent}).

%%%%%%%%%%%%%%%%%%%%%%%%%%%%
\subsection{Information geometry}\label{ssc:inf}%%%%%%

We briefly summarize some notions in information geometry associated with
a non-decreasing, positive, continuous function $\varphi:(0,\infty) \lra (0,\infty)$.
We refer to~\cite{Nad} and \cite{Na} for further discussion.

We define the \emph{$\varphi$-logarithmic function} on $(0,\infty)$ by
\[ \ln_\varphi(t):=\int_1^t \frac1{\varphi(s)} \,ds, \]
which is clearly strictly increasing.
We will denote by $l_{\varphi}$ and $L_{\varphi}$ the infimum and the supremum of $\ln_ {\varphi}$, that is,
\[ l_{\varphi} :=\inf_{t>0} \ln_{\varphi}(t) =\lim_{t\downarrow 0} \ln_{\varphi}(t) \in [-\infty,0),\quad
 L_{\varphi} :=\sup_{t>0} \ln_{\varphi}(t) = \lim_{t\uparrow \infty} \ln_{\varphi}(t) \in(0,\infty]. \]
The inverse function of $\ln_{\varphi}$ is called the \emph{$\varphi$-exponential function}.
We extend it to the function on $\R$ as
\[ \exp_{\varphi}(\tau):=
\begin{cases}
0   &\text{if\ }  \tau  \leq l_{\varphi}, \\
\ln_{\varphi}^{-1}(\tau) &  \text{if\ }  \tau \in \left(l_{\varphi},L_{\varphi} \right), \\
\infty &  \text{if\ }  \tau \geq L_\varphi.
\end{cases} \]
We also introduce the strictly convex function
\[ u_{\varphi}(r):=\int_0^{r} \ln_{\varphi} (t) \,dt, \qquad r \in [0,\infty), \]
provided that it is well-defined (i.e., $\ln_{\varphi}$ is integrable on $(0,1)$).

\begin{lemma}\label{lm:u_phi}
The function  $u_{\varphi}$ is well-defined if 
\begin{equation}\label{del}
\inf \left\{ \delta \in \R \biggm| \frac{s^{1+\delta}}{\varphi(s)} \text{ is bounded on $(0,1)$} \right\}<1.
\end{equation}
\end{lemma}

\begin{proof}
As $\varphi$ is positive and non-decreasing, it suffices to see $u_{\varphi}(1)>-\infty$.
We deduce from the hypothesis that $s/\varphi(s)$ is integrable on $(0,1)$.
This shows the claim since
\[ u_{\varphi}(1)
 =-\int_0^1 \int_t^1 \frac{1}{\varphi(s)} \,dsdt
 =-\int_0^1 \frac{s}{\varphi(s)} \,ds >-\infty. \]
$\qedd$
\end{proof}

Entropy is a function measuring the uncertainty of an event, and the divergence
in information theory is a quantity expressing the difference between a pair of probability measures.
In this spirit, the {\it $\varphi$-entropy} for $\rho \omega \in \cP_{\ac}(M,\omega)$ is defined by 
\[ E_\varphi(\rho\omega):=-\int_{M} u_{\varphi}(\rho) \,d\omega \]
(provided that it is well-defined).
Then we define the \emph{Bregman divergence} between $\rho\omega, \sigma\omega \in \cP_{\ac}(M,\omega)$ by
\begin{equation}\label{eq:Breg}
D_{\varphi}(\rho \omega|\sigma \omega)
 :=\int_M \left\{ u_{\varphi}(\rho)-u_{\varphi}(\sigma)-u'_{\varphi}(\sigma)(\rho-\sigma) \right\} \,d\omega.
\end{equation}
The strict convexity of $u_{\varphi}$ guarantees $D_{\varphi}(\rho \omega|\sigma \omega)>0$
unless $\rho=\sigma$ $\omega$-a.e..
Furthermore,  the square root of the divergence $D_\varphi$ satisfies a generalized Pythagorean theorem
and hence it can be regarded as a kind of distance function, though it is not symmetric
(i.e., $D_{\varphi}(\rho \omega|\sigma \omega) \neq D_{\varphi}(\sigma \omega|\rho \omega)$ in general).
We shall explain some more details of these facts.

Let us consider an $\R^k$-valued  random variable $X=(X_i)_{i=1}^k$ on $M$,
and a convex open set $\Xi \subset \R^k$ on which a partition function
$\Lambda_\varphi:\Xi \lra \R$ can be chosen so that 
the total mass of $\rho_\xi$ is normalized to unity for any $\xi \in \Xi$,
where $\rho_\xi$ is defined by
\[ \nu_{\xi} =\rho_{\xi} \omega
 :=\exp_{\varphi}(\Lambda_{\varphi}(\xi)-\lr{\xi}{X}) \omega \in \cP_{\ac}(M,\omega). \]
Then $\nu_\xi$ is a conditional maximizer of $E_\varphi$ (see~\cite[Theorem~7.2]{Na}).
We set $\cM_{\varphi}$ as the set of conditional maximizers for $E_\varphi$,
\[ \cM_{\varphi}:= \{ \nu_\xi=\exp_{\varphi}(\Lambda_{\varphi}(\xi)-\lr{\xi}{X}) \omega
 \,|\, \xi \in \Xi \ \}  \subset \cP_{\ac}(M,\omega). \]
Through the coordinate $\xi \mapsto \nu_{\xi}$, we introduce the Riemannian metric
$g_{\varphi}$ of $\cM_{\varphi}$ (associated with the Fisher metric when $\varphi(s)=s$) as
\[ g_{\varphi} \left(\frac{\partial}{\partial \xi^i}, \frac{\partial}{\partial \xi^j} \right)
 :=\int_{\Omega} \frac{\partial \rho_{\xi}}{\partial \xi^i}
 \frac{\partial (\ln_\varphi( \rho_{\xi}))}{\partial \xi^j} \,d\omega,
 \quad i,j=1,\ldots,k. \]
We consider another global coordinate $(\eta_{\alpha})_{\alpha=1}^k$ of $\cM_{\varphi}$ given by
$\eta_{\alpha}(\mu) :=\int_{M}  X_{\alpha} \,d\mu$
provided that $\eta_{\alpha}(\mu) \in \R$ for all $\mu \in \cM_{\varphi}$ and $\alpha=1,\ldots,k$.
Then it holds
\[ g_{\varphi} \left(\frac{\partial}{\partial \xi^i}, \frac{\partial}{\partial \eta_{\alpha}} \right)
 =\delta^i_{\alpha}. \]
This shows that geodesics generated by two coordinate systems $\xi$ and $\eta$ are orthogonal to each other.
Then the square root of the Bregman divergence $D_{\varphi}$ satisfies a generalized Pythagorean theorem
for geodesic right triangles in the following sense.

\begin{proposition}{\rm (\cite[Proposition~3]{OW})}\label{pro:pyth}
Fix $\mu\in \cP_{\ac}(M,\omega)$ satisfying $\eta_{\alpha}(\mu) \in \R $ for any $\alpha=1,\ldots, k$.
We assume that, for any $\xi \in \Xi$,  $\supp(\nu_\xi)=M$ holds and
$D_{\varphi}(\mu|\cdot)$ and $\Lambda_{\varphi}$ are differentiable on $\cM_\varphi$.
If some $\xi(\mu) \in \Xi$ satisfies 
$\inf_{ \xi \in \Xi} D_{\varphi}(\mu|\nu_\xi)=D_{\varphi}(\mu|\nu_{\xi(\mu)})$,
in other words, if $\nu_{\xi(\mu)}$ is the foot of the $\eta$-geodesic from $\mu$
meeting $\cM_{\varphi}$ orthogonally, then we have
\begin{align*} 
\eta_{\alpha}(\mu) &=\eta_{\alpha}(\nu_{\xi(\mu)}) =\int_{M} X_\alpha \,d\nu_{\xi(\mu)}
 \quad \text{for all\ } \alpha =1,\ldots, k, \\
D_{\varphi}(\mu|\nu_{\xi})
&=D_{\varphi}(\mu|\nu_{\xi(\mu)})+D_{\varphi}(\nu_{\xi(\mu)}|\nu_{\xi})
\quad \text{for all\ } \xi \in \Xi.
\end{align*}
\end{proposition}

We define three more quantities measuring the order of $\varphi$ for later use:
\begin{align}
&\theta_{\varphi} :=\sup\left\{ \frac{s}{\varphi(s)} \cdot \limsup_{t \downarrow 0} \frac{\varphi(s+t)-\varphi(s)}{t}
 \biggm| s>0 \right\} \in [0,\infty], \label{eq:theta}\\
&\delta_{\varphi} :=\inf\left\{ \frac{s}{\varphi(s)} \cdot \limsup_{t \downarrow 0} \frac{\varphi(s+t)-\varphi(s)}{t}
 \biggm| s>0  \right\} \in [0,\infty), \label{eq:delta}\\
&N_\varphi:= 
\begin{cases}
(\theta_\varphi-1)^{-1}  &\text{if\ } \theta_\varphi \neq 1,\\
\infty &\text{if\ } \theta_\varphi =1.
\end{cases} \label{eq:N}
\end{align}
The following lemma will be useful.

\begin{lemma} \label{lem:mono}
The function $s^{\delta_{\varphi}}/\varphi(s)$ is non-increasing in $s \in (0,\infty)$.
Moreover, if $\theta_\varphi$ is finite, then the function $s^{\theta_{\varphi}}/\varphi(s)$
is non-decreasing in $s \in (0,\infty)$.
\end{lemma}

\begin{proof}
Assume $\theta_{\varphi}<\infty$ (which will not play any role in the discussion on $\delta_{\varphi}$),
and fix $s>0$ and small $\ve>0$.
By the definitions of $\theta_{\varphi}$ and $\delta_{\varphi}$,
there exists $r_{\ve}(s)>0$ such that
\[ \delta_{\varphi} \le
 \frac{s}{\varphi(s)} \cdot \sup_{t \in (0,r_\ve(s)]}\frac{\varphi(s+t) -\varphi(s)}{t}
 \le \theta_{\varphi} +\frac{\ve}{2}. \]
Consider the functions
\[ h_+(\tau) :=\theta_{\varphi} +\frac\ve2 -\frac{(1+\tau)^{\theta_{\varphi}+\ve}-1}{\tau},
 \qquad
 h_-(\tau) :=\delta_{\varphi} -\frac{(1+\tau)^{\delta_{\varphi}-\ve}-1}{\tau} \]
for $\tau>0$.
Since $h_+$ and $h_-$ are continuous and satisfy
\[ \lim_{\tau \downarrow 0} h_+(\tau) =\theta_{\varphi} +\frac\ve2 -(\theta_{\varphi}+\ve) <0,
 \qquad
 \lim_{\tau \downarrow 0} h_-(\tau) =\delta_{\varphi} -(\delta_{\varphi}-\ve) >0, \]
there exists $\tau_{\ve}>0$ such that we have $h_+(\tau)<0$ and $h_-(\tau)>0$ for all $\tau \in (0,\tau_\ve)$.

Given any $t \in (0,\min\{r_\ve(s),s\tau_\ve\})$, we have
\begin{align*}
\frac{s^{\theta_{\varphi}+\ve}}{\varphi(s)} -\frac{(s+t)^{\theta_{\varphi}+\ve}}{\varphi(s+t)}
 &=\frac{ ts^{\theta_\varphi+\ve-1}}{\varphi(s+t)}
 \bigg\{ \frac{s}{\varphi(s)} \frac{\varphi(s+t)-\varphi(s)}{t}
 -\frac{ \left(1+ ts^{-1} \right)^{\theta_\varphi+\ve} -1}{ts^{-1}} \bigg\} \\
 & \le \frac{ts^{\theta_\varphi+\ve-1}}{\varphi(s+t)} h_+ (ts^{-1}) <0.
\end{align*}
As $s>0$ was arbitrary, this shows that $s^{\theta_{\varphi}+\ve}/\varphi(s)$ is strictly increasing in $s>0$.
We similarly obtain
\[ \frac{s^{\delta_\varphi-\ve}}{\varphi(s)} -\frac{(s+t)^{\delta_\varphi-\ve}}{\varphi(s+t)}
 \ge \frac{ts^{\delta_\varphi-\ve-1}}{\varphi(s+t)} h_- (ts^{-1}) >0, \]
so that ${s^{\delta_\varphi-\ve}}/{\varphi(s)}$ is strictly decreasing.
Letting $\ve \downarrow 0$, we complete the proof.
$\qedd$
\end{proof}

\begin{remark}\label{rm:sim}
The function $\varphi$ will be sometimes normalized so as to satisfy $\varphi(1)=1$.
This costs no generality as we easily see the following relations for any $a>0$:
\begin{align*}
&\ln_{a\varphi}(t)=a^{-1} \ln_{\varphi}(t), \quad
\exp_{a\varphi}(\tau)=\exp_{\varphi}(a \tau), \quad
u_{a\varphi}(r)=a^{-1} u_{\varphi}(r), \\
&l_{a\varphi}(\tau)=a^{-1} l_{\varphi}, \quad
L_{a\varphi}=a^{-1} L_{\varphi}, \quad
\theta_{a\varphi}=\theta_{\varphi}, \quad
\delta_{a\varphi}=\delta_{\varphi}, \quad
N_{a\varphi}=N_{\varphi}.
\end{align*}
\end{remark}

%%%%%%%%%%%%%%%%%%%%%%%%%%%%%%%%%%%%%%%%%%%%%%%
\subsection{Information geometry continued: The case of $\varphi_m(s)=s^{2-m}$}\label{ssc:phi_m}%%%%%

In \cite{mCD}, we considered the power function $\varphi_m(s):=s^{2-m}$ for $m \in (0,2]$
and the corresponding \emph{$m$-logarithmic} and \emph{$m$-exponential functions}.
(We have actually considered $m \in [(n-1)/n,\infty)$ in \cite{mCD},
but $\varphi_m$ is non-decreasing only when $m \le 2$.)
We summarize several facts in this especially important case.
For brevity, we set
\[ \ell_m:=\ln_{\varphi_m},\quad e_m:=\exp_{\varphi_m},\quad
 l_m:=l_{\varphi_m},\quad L_m:=L_{\varphi_m},\quad 
 \theta_m:=\theta_{\varphi_m},\quad N_m:=N_{\varphi_m}. \]

(A)
In the case of $\varphi_1(s)=s$, $\ell_1$ and $e_1$ coincide with
the usual logarithmic and exponential functions, respectively.
Thus we find $l_1=-\infty$ and $L_1=\infty$.
We can easily observe $\theta_1=1$ and $N_1=\infty$ as well.
For $\rho\omega,\sigma\omega\in\cP_{\ac}(M,\omega)$, we deduce from
$u_{\varphi_1}(r)=r\ln r -r$ that
\[ E_{\varphi_1}(\rho\omega)=-\int_M \rho \ln \rho \,d\omega+1, \qquad
D_{\varphi_1}(\rho\omega|\sigma \omega)= \int_M \rho \ln \frac{\rho}{\sigma} \,d\omega. \]
Namely $E_{\varphi_1}$ is the \emph{Boltzmann entropy} up to adding $1$,
and $D_{\varphi_1}$ is the \emph{Kullback--Leibler divergence}.
By choosing $\sigma \equiv 1$ formally in the definition of $D_{\varphi_1}(\rho\omega|\sigma \omega)$,
the \emph{relative entropy} of $\mu \in \cP^2(M)$ with respect to $\omega$ is defined by
\begin{equation}\label{eq:Ent}
\Ent_{\omega}(\mu) :=\left\{
\begin{array}{ll}
\displaystyle\lim_{\ve \downarrow 0} \int_{\{\rho \ge \ve\}} \rho\ln\rho \,d\omega
 & \text{if}\ \mu=\rho\omega \in \cP^2_{\ac}(M,\omega), \\
 \infty & \text{otherwise}.
\end{array} \right.
\end{equation}
In other words, the relative entropy is `$(-1) \times$ the Boltzmann entropy.'

(B)
For ${\varphi_m}(s)=s^{2-m}$ with $m \in (0,1) \cup (1,2]$,
$\ell_m$ and $e_m$ are given by power functions as
\[ \ell_m(t) =\frac{t^{m-1}-1}{m-1},\qquad
e_m(\tau) =[1+(m-1)\tau]_+^{1/(m-1)}, \]
where we set $[t]_+:=\max\{t,0\}$ and by convention $0^a:=\infty$ for $a<0$.
Observe
\begin{equation}\label{eq:lLm}
l_m =\begin{cases}
 -\infty& \text{if\ }m<1, \\
 \displaystyle -\frac{1}{m-1} & \text{if\ }m>1,
 \end{cases} \qquad
 L_m =\begin{cases}
 \displaystyle \frac{1}{1-m} &\text{if\ }m<1, \\
 \infty & \text{if\ }m>1,
\end{cases}
\end{equation}
$\theta_m=2-m$ and $N_m=(1-m)^{-1}$.
As $u_{\varphi_m}(r)=(r^m-mr)/\{m(m-1)\}$,
the ${\varphi_m}$-entropy for $\rho\omega\in\cP_{\ac}(M,\omega)$ is given by 
\[ E_{{\varphi_m}}(\rho\omega)=-\int_M \frac{\rho^m}{m(m-1)} \,d\omega +\frac1{m-1}. \]
Up to additive and multiplicative constants, this coincides with the
\emph{R\'enyi$($-Tsallis$)$ entropy}
\begin{equation}\label{eq:S_N}
S_N(\rho\omega):=-\int_M \rho^{(N-1)/N} \,d\omega
\end{equation}
with $N=N_m$, which is applied to complex (strongly correlated) systems.
The Bregman divergence between $\rho\omega,\sigma\omega\in\cP_{\ac}(M,\omega)$ is given by 
\[ D_{{\varphi_m}}(\rho\omega|\sigma\omega)
=\frac1{m(m-1)} \int_{M} \left[(\rho^m-\sigma^m)-m\sigma^{m-1} (\rho-\sigma) \right] d\omega. \]
This coincides with the \emph{$\beta$-divergence}, whose strength is its robustness.
For instance, we refer to~\cite{mtke} for the roles and the differences of statistical divergences
including the Bregman divergences.
Note that as $m \to 1$ we have
\[ \ell_m(t) \to \ell_1(t),\quad e_m(\tau) \to e_1(t),\quad
E_{\varphi_m}(\rho\omega) \to E_{\varphi_1}(\rho\omega),\quad
D_{\varphi_m}(\rho\omega|\sigma\omega) \to D_{\varphi_1}(\rho\omega|\sigma\omega). \]

The function $\varphi_m=s^{2-m}$ is an extremal element among those $\varphi$'s satisfying
$\theta_{\varphi}=2-m$ in several respects, as one can see in the next useful lemma for instance.

\begin{lemma}\label{hikaku}
Assume $\theta_{\varphi}<2$ and put $m=2-\theta_{\varphi}$.
Then for any $t>0$ and $r \in \R$ we have 
\begin{gather}\label{eq:ln}
\frac{1}{\varphi(1)}\ell_m(t) \leq \ln_\varphi(t)
\leq \frac{t^{\theta_\varphi}}{\varphi(t)}\ell_m(t),\\ \label{eq:exp}
\exp_\varphi(r) \leq e_m \big( \varphi(1)r \big).
\end{gather}
\end{lemma}

\begin{proof}
It follows from Lemma~\ref{lem:mono} that, for any $t>0$,
\[ \frac{1}{\varphi(1)} \int_1^t s^{-\theta_\varphi} \,ds
 \le \int_1^t \frac{1}{\varphi(s)} \,ds
 \le \frac{t^{\theta_\varphi}}{\varphi(t)} \int_1^t s^{-\theta_\varphi} \,ds. \]
This is exactly \eqref{eq:ln} since $\theta_{\varphi}=2-m$.

As for \eqref{eq:exp},
the assertion for $r \le l_\varphi$ is trivial since $\exp_{\varphi}(r)=0$ by definition.
If $r \ge L _\varphi$, then we deduce from \eqref{eq:ln} that
$\varphi(1)r \ge \varphi(1)L_{\varphi} \ge L_m$,
which shows $e_m(\varphi(1)r)=\infty$.
We therefore assume $l_\varphi<r<L_\varphi$ and set $t:=\exp_\varphi(r)>0$.
Then we obtain again from \eqref{eq:ln} that
\[ \exp_\varphi(r)=t=e_m \big( \ell_m(t) \big) \le e_m \big( \varphi(1) \ln_\varphi(t) \big)
 =e_m \big( \varphi(1)r \big). \]
$\qedd$
\end{proof}

Taking the limits as $t \downarrow 0$ or $t \uparrow \infty$ in \eqref{eq:ln},
we obtain from \eqref{eq:lLm} the following.

\begin{lemma}\label{lm:case}
Suppose $\theta_\varphi<2$.
\begin{enumerate}[{\rm (i)}]
\item If $\theta_{\varphi} <1$, then $l_{\varphi}>-\infty$
$($equivalently, if $l_{\varphi}=-\infty$, then $\theta_{\varphi} \ge 1)$.
\item If $\theta_{\varphi} \le 1$, then $L_{\varphi}=\infty$
$($equivalently, if $L_{\varphi} <\infty$, then $\theta_{\varphi}>1)$.
\end{enumerate}
\end{lemma}

We similarly find the corresponding estimates concerning $\delta_{\varphi}$.
Note that $\delta_{\varphi_m}=\theta_m=2-m$.

\begin{lemma}\label{lm:cased}
Assume $\delta_{\varphi}<2$ and put $m=2-\delta_{\varphi}$.
Then for any $t>0$ and $r \in \R$ we have 
\[ \frac{t^{\delta_{\varphi}}}{\varphi(t)} \ell_m(t) \le \ln_\varphi(t) \le \frac{1}{\varphi(1)} \ell_m(t),
 \qquad \exp_\varphi(r) \ge e_m \big( \varphi(1)r \big). \]
In particular,
\begin{enumerate}[{\rm (i)}]
\item If $\delta_{\varphi} >1$, then $L_{\varphi}<\infty$
$($equivalently, if $L_{\varphi}=\infty$, then $\delta_{\varphi} \le 1)$.
\item If $\delta_{\varphi} \ge 1$, then $l_{\varphi}=-\infty$
$($equivalently, if $l_{\varphi}>-\infty$, then $\delta_{\varphi}<1)$.
\end{enumerate}
\end{lemma}

%%%%%%%%%%%%%%%%%%%%%%%%%%%%%%%%%%%
\section{Displacement convexity classes $\DC_N$}\label{sc:DC}%%%%%%

In this section, we introduce the important classes of convex functions.
These classes were first considered by McCann~\cite{Mc1} for $N \ge 1$
(see also \cite[\S 5.1]{LV2}, \cite[Section~5.2]{Vi1}, \cite[Chapter~16]{Vi2}),
we adopt the same definition also for $N<0$.

\begin{definition}[Displacement convexity classes]\label{df:DC_N}
For $N \in (-\infty,0) \cup [1,\infty)$, we define $\DC_N$ as the set of all continuous convex functions
$u:[0,\infty) \lra \R$ such that $u(0)=0$ and that the function 
\[ \psi_N(r):=r^N u(r^{-N}) \]
is convex on $(0,\infty)$.
In a similar way, $\DC_\infty$ is defined as  the set of all continuous convex functions
$u:[0,\infty) \lra \R$ such that $u(0)=0$ and that the function 
\[ \psi_{\infty}(r):=e^r u(e^{-r}) \] 
is convex on $\R$.
\end{definition}

The following is well-known for $N \ge 1$, we give a proof for completeness.

\begin{lemma}\label{lm:psi_N}
If $u \in \DC_N$ for $N \in [1,\infty]$ $($resp.\ $N \in (-\infty,0))$,
then the function $\psi_N$ is non-increasing $($resp.\ non-decreasing$)$.
\end{lemma}

\begin{proof}
For $N \in [1,\infty)$ and $0<s<t$, the convexity of $u$ and $u(0)=0$ yield
\[ \psi_N(t) =t^N u(t^{-N})
 \le t^N \bigg\{ \bigg( 1-\frac{s^N}{t^N} \bigg) u(0)+\frac{s^N}{t^N}u(s^{-N}) \bigg\}
 =\psi_N(s). \]
We similarly obtain for $N=\infty$ and $s,t \in \R$ with $s<t$ that
\[ \psi_{\infty}(t) =e^t u(e^{-t})
 \le e^t \bigg\{ \bigg( 1-\frac{e^s}{e^t} \bigg) u(0)+\frac{e^s}{e^t}u(e^{-s}) \bigg\}
 =\psi_{\infty}(s). \]
Finally, for $N<0$ and $0<s<t$, it holds
\[ \psi_N(s) =s^N u(s^{-N})
 \le s^N \bigg\{ \bigg( 1-\frac{s^{-N}}{t^{-N}} \bigg) u(0)+\frac{s^{-N}}{t^{-N}}u(t^{-N}) \bigg\}
 =\psi_N(t). \]
$\qedd$
\end{proof}

It is also known that $\DC_{N'} \subset \DC_N$ for $1 \le N<N' \le \infty$.
This monotonicity in $N$ is violated by extending to $N<0$,
but the monotonicity in $m=(N-1)/N \in [0,\infty)$ holds instead.
Compare this with the monotonicity of $\Ric_N$ in $m$ (Remark~\ref{rm:Ric_N}).

\begin{lemma}\label{lm:DC_N}
For each $N,N' \in (-\infty,0) \cup [1,\infty]$ with $m<m'$,
we have $\DC_{N'} \subset \DC_N$, where we set $m=(N-1)/N$, $m'=(N'-1)/N'$
and $m=1$ if $N=\infty$ $($resp.\ $m'=1$ if $N'=\infty)$.
\end{lemma}

\begin{proof}
We first consider the case of $0\le m<m'<1$ (equivalently, $1 \le N<N'<\infty$).
For any $u \in \DC_{N'}$ and $r>0$, we observe
\[ \psi_N(r) =r^N u(r^{-N}) =(r^{N/N'})^{N'} u\big( (r^{N/N'})^{-N'}\big)
 =\psi_{N'}(r^{N/N'}). \]
This is convex in $r$ since the function $r \mapsto r^{N/N'}$ is concave
and $\psi_{N'}$ is convex and non-increasing.
Thus $u \in \DC_N$ and hence $\DC_{N'} \subset \DC_N$.

The other cases are similar.
For $1<m<m'$, we have $N<N'<0$ so that
$r \mapsto r^{N/N'}$ is convex and $\psi_{N'}$ is non-decreasing.
When $m'=1>m$, $\psi_N(r)=\psi_{\infty}(N\log r)$ holds and note that
$r \mapsto N\log r$ is concave and $\psi_{\infty}$ is non-increasing.
For $m=1<m'$, we find $\psi_{\infty}(r)=\psi_{N'}(e^{r/N'})$
and that $r \mapsto e^{r/N'}$ is convex and $\psi_{N'}$ is non-decreasing.
$\qedd$
\end{proof}

We shall write down a condition for $u_{\varphi} \in \DC_N$ on $\varphi$.
As $u_{\varphi}$ is continuous, convex and satisfies $u_{\varphi}(0)=0$
by definition once it is well-defined, it is sufficient to check \eqref{del} and the convexity of $\psi_N$.

\begin{proposition}\label{prop:dochi}
Assume that $\varphi$ satisfies  the condition~\eqref{del}.
Then the function $\psi_N$ for $N \in (-\infty,0) \cup (1,\infty]$ is convex if and only if 
\begin{equation}\label{eq:cvx}
\int_0^t \frac{s}{\varphi(s)} \,ds \le \frac{N}{N-1} \frac{t^2}{\varphi(t)}
\end{equation}
holds for all $t>0$, where $N/(N-1)=1$ if $N=\infty$.
\end{proposition}

\begin{proof}
We first of all recall that $u_{\varphi}$ is well-defined thanks to \eqref{del} (Lemma~\ref{lm:u_phi}).
For $N \in (-\infty,0) \cup (1,\infty)$ and $r>0$, we calculate
\begin{align*}
\psi'_N(r) &=Nr^{N-1}u_{\varphi}(r^{-N}) -Nr^{-1}u'_{\varphi}(r^{-N}), \\
\psi''_N(r) &= N(N-1)r^{N-2}u_{\varphi}(r^{-N}) +(N-N^2)r^{-2}u'_{\varphi}(r^{-N})
 +N^2r^{-N-2}u''_{\varphi}(r^{-N}) \\
&=N(N-1)r^{N-2} 
 \left\{ u_{\varphi}(r^{-N})-  r^{-N}\ln_{\varphi}(r^{-N})  +\frac{N}{N-1}\frac{r^{-2N}}{\varphi( r^{-N})} \right\}.
\end{align*}
Note that $N(N-1)>0$.
For any $t>0$, we have
\begin{align*}
&u_{\varphi}(t)-t\ln_{\varphi}(t)+\frac{N}{N-1}\frac{t^2}{\varphi(t)}
 = \int_0^t \{ \ln_{\varphi}(\tau)-\ln_{\varphi}(t) \} \,d\tau +\frac{N}{N-1}\frac{t^2}{\varphi(t)} \\
&= -\int_0^t \int_{\tau}^t \frac{1}{\varphi(s)} \,dsd\tau +\frac{N}{N-1}\frac{t^2}{\varphi(t)}
 = -\int_0^t \frac{s}{\varphi(s)} \,ds +\frac{N}{N-1}\frac{t^2}{\varphi(t)}.
\end{align*}
Therefore $\psi''_N \ge 0$ if and only if \eqref{eq:cvx} holds.
For $N=\infty$, we similarly obtain
\[ \psi''_{\infty}(r)
 =e^r\left\{ u_{\varphi}(e^{-r}) -e^{-r}\ln_{\varphi}(e^{-r})+\frac{e^{-2r}}{\varphi(e^{-r})} \right\} \]
for $r \in \R$, and
\[ u_{\varphi}(t)-t\ln_{\varphi}(t)+\frac{t^2}{\varphi(t)}
 = -\int_0^t \frac{s}{\varphi(s)} \,ds +\frac{t^2}{\varphi(t)} \]
for $t>0$.
$\qedd$
\end{proof}

\begin{theorem}\label{th:DC}
If $\theta_\varphi<2$, then the condition~\eqref{del} holds
and we have $u_{\varphi} \in \DC_{N_{\varphi}}$. 
\end{theorem}

\begin{proof}
We deduce from Lemma~\ref{lem:mono} that
\[ 0 \leq \frac{s^{\theta_\varphi}}{\varphi(s)} \leq \frac{1}{\varphi(1)} \]
for all $s \in (0,1)$, this implies \eqref{del} since $\theta_{\varphi}<2$.
Lemma~\ref{lem:mono} also yields that, for any $t>0$,
\[ \int_0^t \frac{s}{\varphi (s)} \,ds 
 \leq \int_0^t \frac{t^{\theta_\varphi}}{\varphi (t)} s^{1-\theta_\varphi} \,ds
=\frac{1}{2-\theta_\varphi}  \frac{t^2}{\varphi (t)}. \]
This is nothing but \eqref{eq:cvx} with $N=N_{\varphi}$ (recall the definition of $N_{\varphi}$ in \eqref{eq:N}),
and hence $u_{\varphi} \in \DC_{N_{\varphi}}$ by Proposition~\ref{prop:dochi}.
$\qedd$
\end{proof}

Recall that $\varphi_m(s)=s^{2-m}$ with $m\in (0,2]$ satisfies $\theta_m=2-m<2$.
Hence Theorem~\ref{th:DC} shows $u_{\varphi_m} \in \DC_{N_m}$.
We close the section with a partial converse of Theorem~\ref{th:DC}.

\begin{proposition}\label{pr:rDC}
If  the condition~\eqref{del} holds, $\delta_{\varphi}<2$
and if we have $u_{\varphi} \in \DC_{N}$ with some $N \in(-\infty,0) \cup (1,\infty]$,
then it holds $\delta_\varphi \le (N+1)/N$
$($where $(N+1)/N=1$ for $N=\infty)$.
\end{proposition}

\begin{proof}
Lemma~\ref{lem:mono} with Proposition~\ref{prop:dochi} yields that, for any $t>0$,
\[  \frac{N}{N-1} \frac{t^2}{\varphi(t)} \ge \int_0^t \frac{s}{\varphi (s)} \,ds
 \ge \int_0^t \frac{t^{\delta_\varphi}}{\varphi (t)} s^{1-\delta_\varphi} \,ds
 =\frac{1}{2-\delta_\varphi}\frac{t^2}{\varphi(t)}, \]
which shows $\delta_\varphi \le (N+1)/N$ as desired.
$\qedd$
\end{proof}

In particular,  for $m \in (0,2]$, $u_{\varphi_m} \in \DC_N$ if and only if $m \ge (N-1)/N$.

%%%%%%%%%%%%%%%%%%%%%%%%%
\section{Admissible spaces}\label{sc:ad}%%%%%%%%%%%

This section is devoted to introducing the class of spaces admissible in our consideration.
Recall our weighted Riemannian manifold $(M,\omega)$
and a function $\varphi$ as in Subsection~\ref{ssc:inf}.
From here on, we further fix the reference measure
\[ \nu=\sigma\omega:= \exp_\varphi(-\Psi)\omega, \]
where $\Psi \in C(M)$ such that
\begin{equation}\label{eq:Mphi}
\Psi>-L_{\varphi}\ \text{on}\ M, \qquad
 M_\varphi^{\Psi}:=\Psi^{-1}\big( (-L_{\varphi}, -l_{\varphi}) \big) \neq \emptyset.
\end{equation}
Note that $\supp\nu=\overline{M_\varphi^\Psi} \neq \emptyset$.
For later convenience, let us define the $K$-convexity of a function on a general metric space.

\begin{definition}[$K$-convexity]\label{df:K-conv}
Given $K \in \R$, we say that a function $\Psi:X \lra (-\infty,\infty]$ on a metric space $(X,d)$
is {\it $K$-convex in the weak sense} (denoted by $\Hess\Psi \ge K$ by slight abuse of notation)
if it is not identically $-\infty$ and, for any two points $x,y \in X$,
there exists a minimal geodesic $\gamma:[0,1] \lra X$ from $x$ to $y$ along which
\begin{equation}\label{eq:Psi}
\Psi\big( \gamma(t) \big) \le (1-t)\Psi(x) +t\Psi(y) -\frac{K}{2}(1-t)td(x,y)^2
\end{equation}
holds for all $t \in [0,1]$.
\end{definition}

We remark that, on a Riemannian manifold $M$, \eqref{eq:Psi} certainly holds for any minimal geodesic
$\gamma:[0,1] \lra M$ by approximation.
Indeed, $\gamma|_{[\ve,1-\ve]}$ is a unique minimal geodesic for any $\ve>0$
and $\Psi$ is continuous.
We are interested in the situation that $\Ric_{N_{\varphi}} \ge 0$ as well as $\Hess\Psi \ge K$ hold
(see Theorem~\ref{th:mCD}).
Finer analysis is possible in the particular case of $K>0$ (Sections~\ref{sc:func}, \ref{sc:conc}).
We prove a lemma in such a case for later use.
The open ball of center $x\in M$ and radius $r>0$ will be denoted by $B(x, r)$.

\begin{lemma}\label{lm:K>0}
Suppose that $\varphi(1)=1$ $($this costs no generality, see Remark~$\ref{rm:sim})$,
$\Hess\Psi \ge K$ for some $K>0$, and take a minimizer $x_0 \in M$ of $\Psi$.
\begin{enumerate}[{\rm (i)}]
\item If $l_{\varphi}>-\infty$, then the set $M_\varphi^\Psi$ as in \eqref{eq:Mphi} is totally convex
and $M_\varphi^\Psi  \subset B (x_0, R)$ holds with $R=\sqrt{-2(l_\varphi +\Psi(x_0))/K}$.
Moreover, $\supp\nu$ is also totally convex and compact.

\item If $l_{\varphi}=-\infty$, $N_\varphi \in [n,\infty)$ and if $\Ric_{N_{\varphi}} \ge 0$,
then we have $M^{\Psi}_{\varphi}=M$, $\nu[M]<\infty$ and
\[ \int_{M} d_g(x_0,x)^p \sigma(x)^a \,d\omega(x)
 \le C_1\nu[M]^{a}+C_2 K^{-aN_{\varphi}} <\infty \]
for any $a\in (1/2,1]$ and $p \in [0,\infty)$ satisfying $(2a-1)N_{\varphi}-p>0$,
where $C_1=C_1(\omega)$ and $C_2=C_2(a,p,\theta_\varphi,\omega)$.
In particular, $\sigma \in L^a(M,\omega)$ for all $a \in (1/2,1]$ and
$\nu[M]^{-1} \cdot \nu \in \cP^p_{\ac}(M,\omega)$ for all $p \in [0,N_{\varphi})$.

\item If $N_{\varphi}=\infty$ and $\Ric_{N_{\varphi}} \ge 0$,
then $\sigma \in L^a(M,\omega)$ for any $a>0$.
\end{enumerate}
\end{lemma}

\begin{proof}
We first remark that the assumption $\Hess\Psi \ge K>0$ guarantees the unique existence
of a point $x_0 \in M_\varphi^\Psi$ such that $\Psi(x_0)=\inf_M \Psi$.
We deduce from the $K$-convexity \eqref{eq:Psi} that
\begin{equation}\label{eq:K>0}
\Psi\big( \gamma(1) \big) \ge \Psi(x_0)+\frac{K}{2}d_g\big( x_0,\gamma(1) \big)^2 
\end{equation}
holds for all minimal geodesics $\gamma:[0,1] \lra M$ with $\gamma(0)=x_0$.

(i)
For any minimal geodesic $\gamma:[0,1] \lra M$ connecting two points $x,y \in M_\varphi^\Psi$, we have 
\[ \Psi\big( \gamma(t) \big) \le (1-t)\Psi(x) +t\Psi(y) -\frac{K}{2}(1-t)td_g(x,y)^2  <-l_\varphi, \]
so that $\gamma$ is contained in $M_\varphi^\Psi$.
The total convexity of $\supp\nu$ can be seen similarly.
Precisely, for any $x,y \in \Psi^{-1}((-L_{\varphi},-l_{\varphi}])$,
$\gamma$ as above satisfies $\gamma((0,1)) \subset M_{\varphi}^{\Psi}$.
This in fact implies that
$\supp\nu=\overline{M^{\Psi}_{\varphi}}=\Psi^{-1}((-L_{\varphi},-l_{\varphi}])$
and is totally convex.
We moreover obtain $M_\varphi^\Psi \subset B(x_0,R)$ from \eqref{eq:K>0},
and thus $\supp\nu$ is compact.

(ii)
The first assertion $M^{\Psi}_{\varphi}=M$ is obvious by definition (see \eqref{eq:Mphi}).
Note that $N_{\varphi} \in [n,\infty)$ implies $\theta_{\varphi} \in (1,(n+1)/n]$ (see \eqref{eq:N}).
Set $m:=2-\theta_{\varphi}<1$ and take $a \in (1/2,1]$ and $p \ge 0$
satisfying $(2a-1)N_{\varphi}-p>0$.
Then \eqref{eq:exp} and \eqref{eq:K>0} imply 
\begin{align*}
&\int_M d_g(x_0,x)^p \sigma(x)^a \,d\omega(x) \\
&\le \int_{B(x_0,1)} \sigma^a \,d\omega
 +\int_1^{\infty} e_m \left(-\Psi(x_0)-\frac{K}{2}r^2\right)^a  r^p \area_{\omega}[S(x_0,r)] \,dr.
\end{align*}
We mention that $m=(N_{\varphi}-1)/N_\varphi$ and, for $s<L_\varphi$ and $t<0$, we have 
\[ e_m(s+t)=\{ e_m(s)^{m-1} +(m-1)t \}^{-N_\varphi}. \]
Thanks to the hypothesis $\Ric_{N_{\varphi}} \ge 0$ with $N_{\varphi} \in [n,\infty)$,
we can estimate the second term by Theorem~\ref{thm:BG} as
\begin{align*}
&\int_1^{\infty} e_m\left(-\Psi(x_0)-\frac{K}{2}r^2\right)^a r^p \area_{\omega}[S(x_0,r)] \,dr \\
&\le \area_{\omega}[S(x_0,1)] \int_1^{\infty} 
 \left\{ e_m \big(\!-\Psi(x_0) \big)^{m-1}  +(1-m)\frac{K}{2}r^2 \right\}^{-aN_\varphi}
 r^{p+N_\varphi-1} \,dr \\
&= \area_{\omega}[S(x_0,1)] \int_1^{\infty}
 \left\{ e_m \big(\!-\Psi(x_0) \big)^{m-1} r^{-2}+(1-m)\frac{K}{2} \right\}^{-aN_\varphi}
 r^{-(2a-1)N_\varphi +p-1} \,dr \\
&\le \area_{\omega}[S(x_0,1)] \left\{ (1-m)\frac{K}{2} \right\}^{-aN_\varphi}
 \int_1^{\infty} r^{-(2a-1)N_\varphi +p-1} \,dr \\
&=\frac{1}{(2a-1)N_\varphi -p} \area_{\omega}[S(x_0,1)]
 \left\{ (1-m)\frac{K}{2} \right\}^{-aN_\varphi}.
\end{align*}
We used the condition $(2a-1)N_{\varphi}-p>0$ in the last equality.
Choosing $a=1$ and $p=0$, we in particular find $\nu[M]<\infty$.
Then the H\"older inequality yields that
\[  \int_{B(x_0,1)} \sigma^a \,d\omega
 \le \left( \int_{B(x_0,1)} \sigma \,d\omega \right)^a \omega[ B(x_0,1) ]^{1-a}
 \le \nu[M]^a \omega[ B(x_0,1) ]^{1-a}. \]
Thus choosing
\begin{align*}
C_1 &:=\max\{ \omega[B(x_0,1)],1 \} \ge \omega [ B(x_0,1) ]^{1-a}, \\
C_2 &:=\frac{1}{(2a-1)N_\varphi -p} \left( \frac{2}{1-m} \right)^{aN_{\varphi}}
 \area_{\omega}[S(x_0,1)]
\end{align*}
gives the desired estimate.

(iii)
Combining~\eqref{eq:K>0} with \eqref{eq:exp} provides, as $\theta_{\varphi}=1$,
\begin{align*}
\int_M \sigma(x)^a \,d\omega(x)
&\le \int_M \exp_{\varphi} \left(-\Psi(x_0)-\frac{K}2d_g(x_0,x)^2\right)^a \,d\omega (x) \\
&\le \exp\big(\!-a\Psi(x_0) \big) \int_M \exp \left(-\frac{aK}2d_g(x_0,x)^2\right) \,d\omega(x).
\end{align*}
Hence the assertion follows from Theorem~\ref{thm:mBG}.
$\qedd$
\end{proof}

Now we introduce the conditions for a quadruple $(M,\omega,\varphi,\Psi)$
to be admissible in our consideration.

\begin{definition}[Admissibility]\label{df:adm}
We say that a quadruple $(M,\omega,\varphi,\Psi)$ is \emph{admissible}
if all the following conditions hold:
\begin{enumerate}[{\rm ({A-}1)}]
\item \label{normal}
$\varphi(1)=1$.

\item \label{joken}
$N_\varphi \in (-\infty,-1] \cup [n, \infty]$ and $N_{\varphi} \neq 2$ or, equivalently,
$\theta_\varphi \in [0,(n+1)/n]$ and $\theta_\varphi<3/2$.

\item \label{katei}
$\Psi>-L_{2-\theta_\varphi}$ on $M$ and
$M_{\varphi}^{\Psi}=\Psi^{-1}((-L_{\varphi},-l_{\varphi})) \neq \emptyset$.

\item \label{integrable}
$h_{\varphi}(\sigma) \in L^1(M,\omega)$ and $\sigma \ln_{\varphi}(\sigma) \in L^1(M,\omega)$,
where $h_{\varphi}:=u_{\varphi}$ if $L_{\varphi}=\infty$ and
$h_{\varphi}(r):=u_{\varphi}(r)-rL_{\varphi}$ if $L_{\varphi}<\infty$
(see \eqref{eq:hphi} below).
\end{enumerate}
\end{definition}

We mention that $l_m \le l_{\varphi}$ and $L_m \le L_{\varphi}$
hold with $m=2-\theta_\varphi$ by \eqref{eq:ln},
and hence $M_{\varphi}^{\Psi} \subset M_{\varphi_m}^{\Psi}$ by (A-\ref{katei}).
The first condition $\varphi(1)=1$ is merely the normalization (see Remark~\ref{rm:sim}),
and (A-\ref{integrable}) is imposed for $\nu$ being adopted as a reference measure of
the Bregman divergence (see \eqref{eq:Breg}).
The next lemma ensures that (A-\ref{integrable}) automatically holds
if $\Ric_{N_{\varphi}} \ge 0$ and $\Hess\Psi \ge K$ for some $K>0$.

\begin{lemma}\label{pro:fit}
Suppose that $(M,\omega,\varphi,\Psi)$ satisfies {\rm (A-\ref{normal})}, {\rm (A-\ref{joken})},
$\Psi>-L_{\varphi}$ on $M$, $M^{\Psi}_{\varphi} \neq \emptyset$,
$\Ric_{N_{\varphi}} \ge 0$ and $\Hess\Psi \ge K$ for some $K>0$.
Then {\rm (A-\ref{integrable})} also holds.
\end{lemma}

\begin{proof}
The case of $l_{\varphi}>-\infty$ is clear due to Lemma~\ref{lm:K>0}(i),
so that we assume $l_{\varphi}=-\infty$ and then $\theta_{\varphi} \ge 1$ (Proposition~\ref{lm:case}(i)).
Observe also that $u_{\varphi}(\sigma) \in L^1(M,\omega)$ implies
$h_{\varphi}(\sigma) \in L^1(M,\omega)$ since $\nu[M]<\infty$ by Lemma~\ref{lm:K>0}(ii), (iii).
Let $x_0$ be the minimizer of $\Psi$ and set $R:=\sqrt{\max\{1,-2\Psi(x_0)\}/K}$.
Note that the $K$-convexity of $\Psi$~\eqref{eq:K>0} guarantees that, on $M \setminus B(x_0,R)$,
\[ 0 \le \sigma=\exp_{\varphi}(-\Psi) \le \exp_{\varphi}\left( -\Psi(x_0)-\frac{K}{2}R^2 \right)
 \le \exp_{\varphi}(0)=1. \]

We first consider the case of $\theta_{\varphi}>1$.
On $M \setminus B(x_0,R)$, \eqref{eq:ln} implies that
\begin{align*}
|\sigma\ln_{\varphi}(\sigma)| &=-\sigma \ln_{\varphi}(\sigma)
 \le -\sigma \ell_{2-\theta_{\varphi}}(\sigma)
 =N_\varphi (\sigma^{2-\theta_\varphi}-\sigma), \\
|u_{\varphi}(\sigma)| &= -\int_0^\sigma \ln_\varphi(t) \,dt
 \le \int_0^\sigma N_\varphi (t^{1-\theta_\varphi}-1) \,dt
 =N_\varphi \left( \frac{ \sigma^{2-\theta_\varphi} } {2-\theta_\varphi}-\sigma \right).
\end{align*}
Thus we have
\begin{align*}
\int_M |\sigma\ln_{\varphi}(\sigma)| \,d\omega
&\le \int_{B(x_0,R)} |\sigma\ln_{\varphi}(\sigma)| \,d\omega
 +\int_{M \setminus B(x_0,R)} N_\varphi (\sigma^{2-\theta_\varphi}-\sigma) \,d\omega, \\
\int_M |u_{\varphi}(\sigma)| \,d\omega
&\le \int_{B(x_0,R)} |u_{\varphi}(\sigma)| \,d\omega
 +\int_{M\setminus B(x_0,R)} N_\varphi \left(\frac{\sigma^{2-\theta_\varphi}}{2-\theta_\varphi}-\sigma \right) \,d\omega.
\end{align*}
As $2-\theta_{\varphi} \in (1/2,1)$ by $\theta_\varphi<3/2$,
Lemma~\ref{lm:K>0}(ii) ensures $u_{\varphi}(\sigma), \sigma\ln_{\varphi}(\sigma) \in L^1(M,\omega)$.

In  the case of $\theta_\varphi=1$, we similarly have on $M\setminus B(x_0,R)$
\[
|\sigma\ln_{\varphi}(\sigma)| \le -\sigma \ln \sigma \le \sqrt{\sigma}, \qquad
|u_{\varphi}(\sigma)| = -\int_0^\sigma \ln_\varphi(t) \,dt
 \le \int_0^\sigma \frac{1}{\sqrt{t}} \,dt =2\sqrt{\sigma}.
\]
Then the claim follows from Lemma~\ref{lm:K>0}(iii).
$\qedd$
\end{proof}

We close the section with an auxiliary lemma on how to normalize $\nu$ when $K>0$.

\begin{lemma}\label{pro:fit2}
Let $(M,\omega,\varphi,\Psi)$ be admissible, $\Ric_{N_{\varphi}} \ge 0$ and
$\Hess\Psi \ge K$ for some $K>0$, and set
$I =(l,L):= (l_{\varphi}+\inf_M \Psi,L_{\varphi}+\inf_M \Psi)$.
We in addition assume that $\Psi$ is differentiable at the minimizer of $\Psi$ if $N_{\varphi}=n$.
Then there exists some $\lambda \in I$
such that $\exp_{\varphi}(\lambda-\Psi) \omega \in \cP_{\ac}(M,\omega)$.
\end{lemma}

\begin{proof}
We first remark that $\inf_M \Psi>-\infty$, and that for any $\lambda \in I$
\[ \Psi-\lambda >\inf_M \Psi -\left( L_{\varphi}+\inf_M \Psi \right) =-L_{\varphi} \]
as well as $\Hess(\Psi-\lambda)=\Hess\Psi \ge K$ hold.
Thus
\[ \Xi(\lambda) :=\int_M \exp_{\varphi}(\lambda-\Psi) \,d\omega <\infty \]
by Lemma~\ref{lm:K>0}.
Since $\Xi$ is non-decreasing and continuous on $I$ by Lebesgue's dominated convergence theorem
(or the monotone convergence theorem),
we are done if $\lim_{\lambda \downarrow l} \Xi(\lambda)<1 < \lim_{\lambda \uparrow L} \Xi(\lambda)$ holds.
We also deduce from the dominated convergence theorem that
$\lim_{\lambda \downarrow l} \Xi(\lambda)=0$.
If $L_\varphi=\infty$, then we find $\lim_{\lambda \uparrow \infty} \Xi(\lambda)=\infty$
by the monotone convergence theorem.

The rest is to prove $\lim_{\lambda \uparrow L} \Xi(\lambda)>1$ when $L_\varphi<\infty$.
Note that $L_\varphi<\infty$ implies $\lim_{s\uparrow \infty}\varphi(s)=\infty$ by definition,
and $\theta_\varphi>1$ (i.e., $N_{\varphi} \in [n,\infty)$) by Proposition~\ref{lm:case}(ii).
Let $x_0 \in M$ be the unique minimizer of $\Psi$ and take $R_0>0$ such that
$B(x_0,R_0) \subset M_\varphi^\Psi$ and that $B(x_0,R_0)$ contains no cut point of $x_0$.
Then, for any $x \in S(x_0,r)$ with $0<r<R \le R_0$, the $K$-convexity of $\Psi$ provides
\begin{equation}\label{eq:fit2}
\Psi(x)
 \le \left(1-\frac{r}{R}\right)\Psi(x_0) +\frac{r}{R}\sup_{S(x_0,R)}\Psi -\frac{K}{2}\left(1-\frac{r}{R}\right) \frac{r}{R} R^2
 =\Psi(x_0) +\frac{K}{2}r^2+ar,
\end{equation}
where we set
\[ a=a(R) :=\frac{1}{R} \left( \sup_{S(x_0,R)}\Psi -\Psi(x_0) -\frac{K}{2}R^2 \right) \ge 0. \]
Observe that $\lim_{R \downarrow 0}a(R)<\infty$ by the $K$-convexity of $\Psi$,
and that $\lim_{R \downarrow 0}a(R)=0$ holds if $N_{\varphi}=n$ since $\Psi$ is assumed to be
differentiable at $x_0$.
In both cases ($N_{\varphi}>n$ or $N_{\varphi}=n$) we can choose $R \in (0,R_0]$ small enough to satisfy
\begin{equation}\label{eq:a(R)}
\frac{K}{2}R^2 +aR <L_{\varphi}, \qquad
 2a<\left( \frac{\area_\omega [S(x_0,R)]}{N_\varphi R^{N_{\varphi}-1}} \right)^{1/N_\varphi}.
\end{equation}
Then take large $\lambda \in I$ such that
\[ \lambda >\Psi(x_0)+\frac{K}{2}R^2+aR. \]
Set
\[ e^\lambda_\varphi(r):=\exp_{\varphi} \left( \lambda-\Psi(x_0) -\frac{K}{2}r^2-a r \right) \]
and note that it is decreasing.

We deduce from Theorem~\ref{thm:BG} and \eqref{eq:fit2} that
\begin{align*}
\Xi(\lambda) &\ge \int_{B(x_0,R)} \exp_{\varphi}(\lambda-\Psi) \,d\omega
 \ge \int_0^R \area_\omega [S(x_0,r)] e^{\lambda}_{\varphi}(r) \,dr \\
&\ge \frac{\area_\omega [S(x_0,R)]}{R^{N_{\varphi}-1}}
 \int_0^R r^{N_{\varphi}-1} e^{\lambda}_{\varphi}(r) \,dr \\
& =\frac{\area_\omega [S(x_0,R)]}{R^{N_{\varphi}-1}}
 \left( \frac{R^{N_{\varphi}}}{N_{\varphi}} e^{\lambda}_{\varphi}(R)
 -\int_0^R \frac{r^{N_{\varphi}}}{N_{\varphi}} (e^{\lambda}_{\varphi})'(r) \,dr\right).
\end{align*}
The concavity of $\ln_\varphi$ yields 
\[ \ln_{\varphi}(t) \le \ln_\varphi\!\big( e_\varphi^\lambda(0) \big)
 +\frac{t-e_\varphi^\lambda(0)}{\varphi(e_\varphi^\lambda(0))}, \]
and hence it holds for  $t:=e_\varphi^\lambda(r)$ that 
\[ 
Kr =-a+\sqrt{a^2+2K \left\{ \ln_{\varphi}\!\big( e_\varphi^\lambda(0) \big) -\ln_{\varphi}(t) \right\}} 
 \ge -a+\sqrt{a^2+2K \frac{e_\varphi^\lambda(0)-t}{\varphi(e_\varphi^\lambda(0))}}. \]
By the change of variables formula for $t=e_\varphi^\lambda(r)$, we have
\begin{align*}
-\int_0^R r^{N_{\varphi}} (e_{\varphi}^{\lambda})'(r) \,dr
&=\frac{1}{K^{N_\varphi}} \int_{e_\varphi^\lambda(R)}^{e_\varphi^\lambda(0)}
 \left(-a+\sqrt{a^2+2K\left\{ \ln_{\varphi}\!\big( e_\varphi^\lambda(0) \big) -\ln_\varphi(t) \right\}} \right)^{N_\varphi} dt \\
&\ge \frac{1}{K^{N_{\varphi}}} \int_{e_\varphi^\lambda(R)}^{e_\varphi^\lambda(0)}
 \left(-a+\sqrt{a^2+2K \frac{e_\varphi^\lambda(0)-t}{\varphi(e_\varphi^\lambda(0))}} \right)^{N_\varphi} dt.
\end{align*}
Combining the triangle inequality
\begin{align*}
&\left\{ \int_{e_\varphi^\lambda(R)}^{e_\varphi^\lambda(0)}
 \bigg(\! -a+\sqrt{a^2+2K \frac{e_\varphi^\lambda(0)-t}{\varphi(e_\varphi^\lambda(0))}} \bigg)^{N_\varphi} dt
 \right\}^{1/N_{\varphi}} \\
&\ge \left\{ \int_{e_\varphi^\lambda(R)}^{e_\varphi^\lambda(0)}
 \bigg(\sqrt{a^2+2K \frac{e_\varphi^\lambda(0)-t}{\varphi(e_\varphi^\lambda(0))}} \bigg)^{N_\varphi} dt
 \right\}^{1/N_{\varphi}}
 -\{e_\varphi^\lambda(0)-e_\varphi^\lambda(R) \}^{1/N_{\varphi}}a
\end{align*}
with Jensen's inequality for the convex function $s \mapsto (\sqrt{a^2+s})^{N_{\varphi}}$
($s \ge 0$), we deduce that
\begin{align*}
&\int_{e_\varphi^\lambda(R)}^{e_\varphi^\lambda(0)}
 \left(-a+\sqrt{a^2+2K \frac{e_\varphi^\lambda(0)-t}{\varphi(e_\varphi^\lambda(0))}} \right)^{N_\varphi} dt \\
&\ge \{ e_\varphi^\lambda(0)-e_\varphi^\lambda(R) \}
 \left\{ \sqrt{ a^2 +\frac{2K}{\varphi(e_\varphi^\lambda(0))}
 \int_{e_\varphi^\lambda(R)}^{e_\varphi^\lambda(0)}
 \frac{e_\varphi^\lambda(0)-t}{e_\varphi^\lambda(0)-e_\varphi^\lambda(R)} dt} -a \right\}^{N_\varphi} \\
&= \{e_\varphi^\lambda(0)-e_\varphi^\lambda(R) \}
 \left( \sqrt{a^2+K\frac{e_\varphi^\lambda(0)-e_{\varphi}^{\lambda}(R)}{\varphi(e_\varphi^\lambda(0))}}
 -a \right)^{N_{\varphi}}.
\end{align*}
Hence we obtain, as $\varphi(s) \le s^{\theta_{\varphi}}$ for $s \ge 1$ by Lemma~\ref{lem:mono},
\begin{align*}
\lim_{\lambda \uparrow L} \Xi(\lambda)
&\ge \frac{\area_{\omega}[S(x_0,R)]}{N_{\varphi} R^{N_{\varphi}-1} K^{N_{\varphi}}}
 \lim_{\lambda \uparrow L} \left[ \{ e_\varphi^\lambda(0)-e_\varphi^\lambda(R) \}
 \left(\sqrt{a^2+K\frac{e_\varphi^\lambda(0)-e_\varphi^\lambda(R)}{\varphi(e_\varphi^\lambda(0))}}
 -a \right)^{N_\varphi} \right] \\
&\ge \frac{\area_{\omega}[S(x_0,R)]}{N_{\varphi} R^{N_{\varphi}-1} K^{N_{\varphi}}}
 \lim_{s \uparrow \infty} \left[ s \left(\sqrt{a^2+\frac{Ks}{(s+e^L_{\varphi}(R))^{\theta_{\varphi}}}}
 -a \right)^{N_\varphi} \right] \\
&= \frac{\area_{\omega}[S(x_0,R)]}{N_{\varphi} R^{N_{\varphi}-1} K^{N_{\varphi}}}
 \lim_{s \uparrow \infty} \left[ s^{\theta_\varphi-1} (\sqrt{a^2+Ks^{1-\theta_\varphi}} -a) \right]^{N_\varphi} \\
&= \frac{\area_{\omega}[S(x_0,R)]}{N_{\varphi} R^{N_{\varphi}-1} K^{N_{\varphi}}}
 \lim_{t \downarrow 0} \left( \frac{\sqrt{a^2+Kt}-a}{t} \right)^{N_\varphi} 
= \frac{\area_{\omega}[S(x_0,R)]}{N_{\varphi} R^{N_{\varphi}-1}} (2a)^{-N_{\varphi}}.
\end{align*}
This is bigger than $1$ by the choice of $R$ (recall \eqref{eq:a(R)}) and we complete the proof.
$\qedd$
\end{proof}

\begin{remark}\label{rem:prob}
If $\Ric_{N_{\varphi}} \ge 0$ and $\Hess\Psi \ge K$ for some $K>0$,
then Lemma~\ref{lm:K>0} yields $\nu[M]<\infty$.
We will sometimes normalize $(M,\omega,\varphi,\Psi)$ so as to satisfy $\nu[M]=1$
(Sections~\ref{sc:func}, \ref{sc:conc}).
There are two ways of such a normalization:
\begin{itemize}
\item
Put $a:=\nu[M]^{-1}$ and consider $(M,\tilde{\omega},\varphi,\Psi):=(M,a\omega,\varphi,\Psi)$,
i.e., $\tilde{\omega}=e^{-f+\ln a}\vol_g$.
\item
Take $\lambda$ as in Lemma~\ref{pro:fit2} and consider
$(M,\omega,\varphi,\widetilde{\Psi}):=(M,\omega,\varphi,\Psi-\lambda)$.
\end{itemize}
In both cases it is easily seen that the conditions $\Ric_{N_{\varphi}} \ge 0$ and $\Hess\Psi \ge K$
are preserved.
These two normalizations are equivalent when $\varphi=\varphi_1$,
where we indeed observe
\[ \exp(-\Psi) \tilde{\omega} =\exp(-\Psi-f+\ln a) \omega, \quad
 \exp(-\widetilde{\Psi}) \omega =\exp(-\Psi -f+\lambda) \omega, \]
and hence $\lambda=\ln a$.
\end{remark}

%%%%%%%%%%%%%%%%%%%%%%%%%%%%%%%%%%%%%%%%%%%
\section{$\varphi$-relative entropy and its displacement convexity}\label{sc:ent}%%%%%

In this section, we introduce a generalization of the relative entropy,
that we call the \emph{$\varphi$-relative entropy}, associated with functions $\varphi$
in an appropriate class.
For the relative entropy on a (unweighted) Riemannian manifold,
it is known by von Renesse and Sturm~\cite{vRS} that its $K$-convexity in the Wasserstein space
$(\cP^2(M),W_2)$ (in the sense of Definition~\ref{df:K-conv})
is equivalent to the lower Ricci curvature bound $\Ric \ge K$.
Then it was shown by Lott and Villani~\cite{LV2} that $\Ric \ge K$
further implies a kind of convexity property of a class of entropies
including the relative entropy.
In this sense, the relative entropy is an extremal element in such a class of entropies.
In the same spirit, our main theorem in the section (Theorem~\ref{th:mCD})
asserts that the $m$-relative entropy induced from $\varphi=\varphi_m$
(studied in \cite{mCD}, recall Subsection~\ref{ssc:phi_m} as well)
is an extremal one in an appropriate class of $\varphi$-relative entropies.

\subsection{Curvature-dimension condition}%%%%%%%%%%%%%%

We begin with a brief review of Lott, Sturm and Villani's curvature-dimension condition.
To formulate it, we need to introduce the two functions often used in comparison theorems
in Riemannian geometry.

For $K \in \R$, $N \in (1,\infty)$ and $0<r$ ($<\pi\sqrt{(N-1)/K}$ if $K>0$),
we consider the function
\[ \bs_{K,N}(r) := \left\{
 \begin{array}{cl}
 \sqrt{(N-1)/K} \sin (r\sqrt{K/(N-1)}) & {\rm if}\ K>0, \vspace{1mm}\\
 r & {\rm if}\ K=0, \vspace{1mm}\\
 \sqrt{-(N-1)/K} \sinh (r\sqrt{-K/(N-1)}) & {\rm if}\ K<0.
 \end{array} \right. \]
This is the solution to the differential equation
\[ \bs_{K,N}''+\frac{K}{N-1}\bs_{K,N}=0 \]
with the initial conditions $\bs_{K,N}(0)=0$ and $\bs'_{K,N}(0)=1$.
For $n \in \N$ with $n \ge 2$, $\bs_{K,n}(r)^{n-1}$ is proportional to
the area of the sphere of radius $r$ in the $n$-dimensional space form
of constant sectional curvature $K/(n-1)$.
Using $\bs_{K,N}$, we define
\[ \beta^t_{K,N}(r) :=\bigg( \frac{\bs_{K,N}(tr)}{t\bs_{K,N}(r)} \bigg)^{N-1}, \qquad
 \beta^t_{K,\infty}(r) :=e^{K(1-t^2)r^2/6} \]
for $K,N,r$ as above and $t \in (0,1)$.
Now, we are ready to state Sturm's \emph{curvature-dimension condition}
characterizing lower Ricci curvature bounds,
developed in \cite{vRS}, \cite{Stcon} and \cite{StII} in gradually general situations
(see also \cite{CMS}, \cite{CMS2} for related important work).
Recall \eqref{eq:S_N} and \eqref{eq:Ent} for the definitions of
the R\'enyi entropy $S_N$ and the relative entropy $\Ent_{\omega}$.

\begin{theorem}[Sturm's curvature-dimension condition]\label{th:CD}
Let $(M,\omega)$ be a weighted Riemannian manifold.
We have $\Ric_N \ge K$ for some $K \in \R$
and $N \in [n,\infty)$ if and only if any pair of measures
$\mu_0=\rho_0 \omega, \mu_1=\rho_1 \omega \in \cP^2_{\ac}(M,\omega)$ satisfies
\begin{align}
S_N(\mu_t) &\le -(1-t)\int_{M \times M}
 \beta^{1-t}_{K,N}\big( d_g(x,y) \big)^{1/N} \rho_0(x)^{-1/N} \,d\pi(x,y) \nonumber\\
&\quad -t\int_{M \times M}
 \beta^t_{K,N}\big( d_g(x,y) \big)^{1/N} \rho_1(y)^{-1/N} \,d\pi(x,y) \label{eq:StCD}
\end{align}
for all $t \in (0,1)$, where $(\mu_t)_{t \in [0,1]} \subset \cP^2_{\ac}(M,\omega)$
is the unique  minimal geodesic from $\mu_0$ to $\mu_1$,
and $\pi$ is the unique optimal coupling of $\mu_0$ and $\mu_1$.

Similarly, $\Ric_{\infty} \ge K$ is equivalent to the $K$-convexity of $\Ent_{\omega}$,
\[ \Ent_{\omega}(\mu_t) \le (1-t)\Ent_{\omega}(\mu_0)
 +t\Ent_{\omega}(\mu_1) -\frac{K}{2}(1-t)t W_2(\mu_0,\mu_1)^2. \]
\end{theorem}

For $K=0$, we find $\beta^t_{0,N} \equiv 1$ and \eqref{eq:StCD} is nothing but the convexity of $S_N$,
\[ S_N(\mu_t) \le (1-t)S_N(\mu_0)+tS_N(\mu_1). \]
For $K \neq 0$, however, \eqref{eq:StCD} is not simply the $K$-convexity of $S_N$.

Lott and Villani's version of the curvature-dimension condition requires a similar
convexity condition, but for all entropies induced from functions in $\DC_N$
(\cite{LV1}, \cite{LV2}, \cite[Part~III]{Vi2}).
For $U \in \DC_N$ and $\mu \in \cP^2(M)$, we denote by $\mu=\rho\omega+\mu^s$
its Lebesgue decomposition into absolutely continuous and singular parts with respect to the base measure $\omega$,
and define
\[ U_{\omega}(\mu):=\lim_{\ve \downarrow 0} \int_{\{\rho \ge \ve\}} U(\rho) \,d\omega +U'(\infty)\mu^s[M],
 \qquad U'(\infty):=\lim_{r \to \infty}\frac{U(r)}{r}. \]
We set $\infty \cdot 0=0$ by convention, and remark that $U'(\infty)=\lim_{r \to \infty}U'(r)$
holds due to the convexity of $U$.

\begin{theorem}[Lott and Villani's version]\label{th:LVCD}
We have $\Ric_N \ge K$ for some $K \in \R$ and $N \in [n,\infty]$ if and only if,
given any pair of measures $\mu_0, \mu_1 \in \cP^2(M)$ decomposed
as $\mu_i=\rho_i \omega +\mu_i^s$ $(i=0,1)$, there is a minimal geodesic
$(\mu_t)_{t \in [0,1]} \subset \cP^2(M)$ between them such that
\begin{align}
U_{\omega}(\mu_t)
&\le (1-t)\int_M \int_ M \beta^{1-t}_{K,N}\big( d_g(x,y) \big)
 U\bigg( \frac{\rho_0(x)}{\beta^{1-t}_{K,N}(d_g(x,y))} \bigg) \,d\pi_x(y) d\omega(x) \nonumber\\
&\quad +t\int_M \int_M \beta^t_{K,N}\big( d_g(x,y) \big)
 U\bigg( \frac{\rho_1(y)}{\beta^t_{K,N}(d_g(x,y))} \bigg) \,d\pi_y(x) d\omega(y) \nonumber\\
&\quad +U'(\infty)\{ (1-t)\mu_0^s[M]+t\mu_1^s[M] \} \label{eq:LVCD}
\end{align}
holds for all $U \in \DC_N$ and $t \in (0,1)$, where $\pi_x$ $(\in \cP(M)$ $\mu_0$-a.e.\ $x)$
and $\pi_y$ $(\in \cP(M)$ $\mu_1$-a.e.\ $y)$ denote the disintegrations of $\pi$ by $\mu_0$ and $\mu_1$, i.e.,
\[ d\pi(x,y)=d\pi_x(y)d\mu_0(x)=d\pi_y(x)d\mu_1(y). \]
\end{theorem}

Recall that $\DC_{N'} \subset \DC_N$ for $n \le N<N'$ (Lemma~\ref{lm:DC_N}).
This agrees with the monotonicity $\Ric_N \le \Ric_{N'}$ for $n \le N \le N'$.
In the case where both $\mu_0$ and $\mu_1$ are absolutely continuous with respect to $\omega$, we find
\[ d\pi(x,y)=\rho_0(x)d\pi_x(y)d\omega(x)=\rho_1(y)d\pi_y(x)d\omega(y) \]
and $(\ref{eq:LVCD})$ is rewritten in the more symmetric form
\begin{align*}
U_{\omega}(\mu_t)
&\le (1-t)\int_{M \times M} \frac{\beta^{1-t}_{K,N}(d_g(x,y))}{\rho_0(x)}
 U\bigg( \frac{\rho_0(x)}{\beta^{1-t}_{K,N}(d_g(x,y))} \bigg) \,d\pi(x,y) \\
&\quad +t\int_{M \times M} \frac{\beta^t_{K,N}(d_g(x,y))}{\rho_1(y)}
 U\bigg( \frac{\rho_1(y)}{\beta^t_{K,N}(d_g(x,y))} \bigg) \,d\pi(x,y).
\end{align*}

Besides Riemannian manifolds, these two versions of the curvature-dimension condition
are equivalent for metric measure spaces where geodesics do not branch,
such as Finsler manifolds and Alexandrov spaces.
In other words, Sturm's version implies Lott and Villani's one.
Roughly speaking, this implication can be seen by localizing Sturm's \eqref{eq:StCD}
thanks to the non-branching property, and then integrating these local inequalities
for each $U \in \DC_N$ yields \eqref{eq:LVCD}.
The same infinitesimal estimate (Claim~\ref{cl:mCD}) will appear in our discussion.
Theorem~\ref{th:LVCD} is extended to general Finsler manifolds by introducing
the appropriate notion of the weighted Ricci curvature
(see Section~\ref{sc:Fins} and \cite{Oint}).

General (not necessarily differentiable) metric measure spaces satisfying
the condition in Theorem~\ref{th:CD} or \ref{th:LVCD} are known
to behave like Riemannian manifolds of $\Ric \ge K$ and $\dim \le N$
in geometric and analytic respects (\cite{StI}, \cite{StII}, \cite{LV1}, \cite{LV2}, \cite[Part~III]{Vi2}).
We shall generalize this technique to $\varphi$-relative entropies in the following sections.

%%%%%%%%%%%%%%%%%%%%%%%%%%%%
\subsection{$\varphi$-relative entropy $H_{\varphi}$}%%%%%%%%%%%%%%%%

Let $(M,\omega,\varphi,\Psi)$ be an admissible space in the sense of Definition~\ref{df:adm}.
We modify $u_{\varphi}$ as, for $r \ge 0$,
\begin{equation}\label{eq:hphi}
h_{\varphi}(r) :=\begin{cases}
 u_{\varphi}(r) & \text{if } L_{\varphi}=\infty,\\
 u_{\varphi}(r)-rL_{\varphi} & \text{if } L_{\varphi}<\infty.
\end{cases}
\end{equation}
We also define
\[ h'_{\varphi}(\infty) :=\lim_{r\to \infty} h'_{\varphi}(r)
  =\begin{cases}
 \infty & \text{if } L_{\varphi}=\infty,\\
 0 & \text{if } L_{\varphi}<\infty.
 \end{cases} \]
Note that $u_{\varphi} \in \DC_{N_{\varphi}}$ (by Theorem~\ref{th:DC} thanks to the admissibility)
immediately implies $h_{\varphi} \in \DC_{N_{\varphi}}$.
Moreover, if $L_{\varphi}<\infty$, then $h_\varphi$ is non-increasing and hence nonpositive.
We set 
\begin{align*}
L^\varphi(M,\omega)
&:=\{ \rho:M \lra \R \,|\, \text{measurable}, h_{\varphi}(\rho) \in L^1(M,\omega)\}, \\
\cP^{\Psi}(M)
&:=\{ \mu\in\cP(M) \,|\, \Psi  \in L^1(M^{\Psi}_{\varphi},\mu) \}
\end{align*}
(we will use these notations only in Remark~\ref{rm:Hm}).
Now the Bregman divergence \eqref{eq:Breg} leads us to the following generalization
of the relative entropy.

\begin{definition}[$\varphi$-relative entropy]\label{df:Hm}
Given $\mu \in \cP(M)$, letting $\mu=\rho\omega +\mu^s$ be its Lebesgue decomposition,
we define the \emph{$\varphi$-relative entropy} of $\mu$ by
\begin{align}
H_\varphi(\mu)
&:= \int_M \{ h_{\varphi}(\rho)-h'_{\varphi}(\sigma) \rho \} \,d\omega
 -\int_M h'_{\varphi}(\sigma) \,d\mu^s +h'_{\varphi}(\infty) \mu^s[M] \nonumber\\
&= \int_M h_{\varphi} (\rho) \,d\omega -\int_M h'_{\varphi}(\sigma)  \,d\mu
 +h'_{\varphi}(\infty)\mu^s[M] \label{eq:Hm}
\end{align}
if $h_{\varphi}(\rho) \in L^1(M,\omega)$ and $h'_{\varphi}(\sigma) \in L^1(M,\mu)$,
otherwise we set $H_\varphi(\mu):=\infty$.
\end{definition}

Let us summarize several remarks on Definition~\ref{df:Hm}.

\begin{remark}\label{rm:Hm}
(1) In the second term of \eqref{eq:Hm}, to be precise, we set
\[ h'_{\varphi}(\sigma) :=\begin{cases}
 l_{\varphi} & \text{if } L_{\varphi}=\infty,\\
 l_{\varphi}-L_{\varphi} & \text{if } L_{\varphi}<\infty \end{cases}
 \qquad \text{on}\ M \setminus M^{\Psi}_{\varphi}. \]
This causes no problem because $M=M_{\varphi}^{\Psi}$ if $l_{\varphi}=-\infty$.
The additional condition $\mu[M^{\Psi}_{\varphi}]=1$
(in other words, $\mu \in \cP(M^{\Psi}_{\varphi})$) will be imposed
only when we compare the behavior of $\Psi$ with that of $H_{\varphi}$
(as in Theorems~\ref{th:mCD}, \ref{th:gf} and so forth).

(2) We remark that the condition (A-\ref{integrable}) in the admissibility guarantees
that $\sigma \in L^\varphi(M,\omega)$ as well as $\Psi \in L^1(M^{\Psi}_{\varphi},\nu)$.
Thus we have $H_{\varphi}(\nu) \in \R$ (by extending the definition~\eqref{eq:Hm} verbatim).

(3) The validity of the definition of $h'_\varphi(\infty)$
(for the lower semi-continuity of $H_{\varphi}$, see Lemma~\ref{lm:lsc})
would be understood by the following observation:
For small $\ve>0$, put  $\mu_\ve =\rho_\ve \omega :=\omega[B(x,\ve)]^{-1}\chi_{B(x,\ve)} \omega$,
where $\chi_{B(x,\ve)}$ stands for the characteristic function of $B(x,\ve)$.
Then we have 
\[ \int_{B(x,\ve)} h_\varphi(\rho_\ve) \,d\omega
 =\omega[B(x,\ve)] \cdot h_\varphi \left(\frac1{\omega[B(x,\ve)] } \right)
 \to h'_{\varphi}(\infty) \]
as $\ve$ tends to zero.

(4) Finally, as for the domain of $H_{\varphi}$,
it is more consistent to set $H_{\varphi}(\mu)=\infty$ only if
$h_{\varphi}(\rho)-\rho h'_{\varphi}(\sigma) \not\in L^1(M,\omega)$.
However, as we will sometimes treat the internal energy $\int_M h_{\varphi}(\rho) \,d\omega$
and the potential energy $\int_M h'_{\varphi}(\sigma) \,d\mu$ separately, we consider the smaller domain
in Definition~\ref{df:Hm}.
This may cause a problem when considering the lower semi-continuity of $H_{\varphi}$,
whereas we need it only for compact $M$ (see Lemma~\ref{lm:lsc} below)
where $h'_{\varphi}(\sigma) \in L^1(M,\mu)$ is always true
(so that $h_{\varphi}(\rho) \in L^1(M,\omega)$ if and only if
$h_{\varphi}(\rho)-\rho h'_{\varphi}(\sigma) \in L^1(M,\omega)$).

Let us add a comment on the relation between
$\rho \in L^{\varphi}(M,\omega)$ and $\mu \in \cP^{\Psi}(M)$.
Assume $L_\varphi <\infty$ and $\mu=\rho \omega +\mu^s \in \cP^\Psi(M)$.
The nonpositivity and the convexity of $h_\varphi$ yield
\begin{align*}
\int_M |h_\varphi(\rho)| \,d\omega
&= -\int_M h_{\varphi}(\rho) \,d\omega
 \le -\int_M \{ h_{\varphi}(\sigma)+h'_{\varphi}(\sigma)(\rho -\sigma) \} \,d\omega \\
&= -\int_M \{ h_\varphi(\sigma) -h'_\varphi(\sigma) \sigma  \} \,d\omega
 -\int_M \rho h'_\varphi(\sigma) \,d\omega <\infty.
\end{align*}
Hence $\rho \in L^\varphi(M,\omega)$ automatically holds.
One can also see the converse implication
($\rho \in L^\varphi(M,\omega)$ $\Rightarrow$ $\mu \in \cP^\Psi(M)$)
for the special case $\varphi_m(s)=s^m$ with $m>1$, where $L_m=\infty$ as in \eqref{eq:lLm}
(\cite[Remark~3.2(2)]{mCD}).
\end{remark}

It is easily observed that $\nu$ is a unique \emph{ground state} of $H_{\varphi}$
(provided that $\nu[M]=1$).

\begin{lemma}\label{lm:Hm}
Suppose $\nu[M]=1$.
For any $\mu=\rho\omega + \mu^s \in \cP(M)$,
we have $H_{\varphi}(\mu) \ge H_\varphi(\nu)$ and equality holds if and only if $\mu=\nu$.
\end{lemma}

\begin{proof}
We assume $H_{\varphi}(\mu)<\infty$ without loss of generality.
Observe that
\begin{align}\label{eq:rel}
&H_\varphi(\mu)-H_{\varphi}(\nu) \nonumber\\
&= \int_M \{ h_{\varphi}(\rho)-h_{\varphi}(\sigma)-h'_{\varphi}(\sigma) (\rho-\sigma) \} \,d\omega
 -\int_M h'_{\varphi}(\sigma) \,d\mu^s\,+ h'_{\varphi}(\infty) \mu^s[M].
\end{align}
On the one hand, if $\mu^s[M]>0$, then the singular part  
\[ -\int_M h'_{\varphi}(\sigma)  \,d\mu^s + h'_{\varphi}(\infty) \mu^s[M] \]
in \eqref{eq:rel} is positive if $L_\varphi<\infty$
(recall that $h'_{\varphi}(\sigma)=l_{\varphi}-L_{\varphi}<0$ on $M \setminus M^{\Psi}_{\varphi}$)
or infinity if $L_\varphi=\infty$.
On the other hand,
the strict convexity of $h_\varphi$ implies that the absolutely continuous part    
\[ \int_M \{ h_{\varphi}(\rho)-h_{\varphi}(\sigma)-h'_{\varphi}(\sigma) (\rho-\sigma) \} \,d\omega \]
in \eqref{eq:rel}  is nonnegative and equality holds if and only if $\rho=\sigma$ $\omega$-a.e..
Thus $H_{\varphi}(\mu) \ge H_{\varphi}(\nu)$ and equality holds if and only if
$\mu^s[M]=0$ and $\rho=\sigma$ holds $\omega$-a.e..
$\qedd$
\end{proof}

Observe from \eqref{eq:rel} that it holds $D_\varphi(\mu|\nu)=H_\varphi(\mu)-H_\varphi(\nu)$
for any absolutely continuous measure $\mu$ with respect to $\omega$.
Thus the Bregman divergence $D_\varphi(\mu|\nu)$ measures the difference between
the entropies at $\mu$ and the ground state $\nu$.
In \cite{mCD}, we have studied the specific function $\varphi_m(s)=s^{2-m}$
and the associated \emph{$m$-relative entropy} $H_m(\mu|\nu)$
for $m \in [(n-1)/n,1) \cup (1,\infty)$.
In the present context, $H_m(\mu|\nu)$ coincides with $H_{\varphi_m}(\mu)-H_{\varphi_m}(\nu)$.

The following lemma will be used in Section~\ref{sc:gf} (Claim~\ref{cl:gf})
to construct a discrete approximation of the gradient flow of $H_{\varphi}$,
where $M$ is assumed to be compact.

\begin{lemma}\label{lm:lsc}
Let $M$ be compact.
Then the $\varphi$-relative entropy $H_\varphi$ is lower semi-continuous with respect to the weak topology, 
that is to say, if a sequence $\{ \mu_i \}_{i \in \N} \subset \cP(M)$
weakly converges to $\mu \in \cP(M)$, then we have
\[ H_\varphi(\mu) \le \liminf_{i \to \infty} H_\varphi(\mu_i). \]
\end{lemma}

\begin{proof}
We divide $H_\varphi(\mu)$ into two parts as 
\[ H^{(1)}_\varphi(\mu):= \int_M h_{\varphi}(\rho) \,d\omega +h'_{\varphi}(\infty) \mu^s[M],
 \qquad H^{(2)}_\varphi(\mu):=- \int_M h'_\varphi(\sigma) \,d\mu, \]
where $\mu=\rho\omega +\mu^s$.
Note that $\|h'_{\varphi}(\sigma)\|_{\infty}<\infty$ thanks to the compactness of $M$
(recall Remark~\ref{rm:Hm}(1)).
Then $H^{(2)}_\varphi(\mu)$ is clearly continuous in $\mu$ and
the lower semi-continuity of $H^{(1)}_\varphi(\mu)$ follows from \cite[Theorem~B.33]{LV2}.
$\qedd$
\end{proof}

%%%%%%%%%%%%%%%%%%%%%%%%%%%%%%%%%%
\subsection{Displacement convexity of $H_{\varphi}$}%%%%%%%%%%%

In our previous work~\cite{mCD}, we showed that the displacement $K$-convexity of the $m$-relative entropy
$H_m(\mu|\nu)=H_{\varphi_m}(\mu)-H_{\varphi_m}(\nu)$ with respect to $\mu \in \cP^2(M)$
is equivalent to the combination of $\Ric_N \ge 0$ and $\Hess\Psi \ge K$,
where $N=1/(1-m)$ (\cite[Theorem~4.1]{mCD}).
This characterization can be regarded as to correspond to Sturm's version of
the curvature-dimension condition~\eqref{eq:StCD}
($\Hess H_m(\cdot|\nu) \ge K$ and \eqref{eq:StCD} actually coincide if $\Psi$ is constant and $K=0$).
In the reminder of the section, we shall consider the convexity of appropriate families of
the $\varphi$-relative entropies corresponding to Lott and Villani's version
of the curvature-dimension condition~\eqref{eq:LVCD}.
Recall \eqref{eq:theta} for the definition of $\theta_{\varphi}$.

\begin{theorem}[Displacement convexity of families of $H_{\varphi}$]\label{th:mCD}
Given $K \in \R$, $N \in \R\setminus (-1,n)$ and an admissible space $(M,\omega,\varphi_m,\Psi)$,
the following three conditions are mutually equivalent, where $m=(N-1)/N:$
\begin{enumerate}[{\rm (A)}]
\item
We have $\Ric_N \ge 0$ and $\Hess\Psi \ge K$ on $M^{\Psi}_{\varphi_m}$.

\item 
For any $\mu_0,\mu_1 \in \cP^2(M)$ with $\mu_0[M^{\Psi}_{\varphi_m}]=\mu_1[M^{\Psi}_{\varphi_m}]=1$
such that any pair of points $x_i \in \supp\mu_i \cap M^{\Psi}_{\varphi_m}$
$(i=0,1)$ are joined by some minimal geodesic contained in $M_{\varphi_m}^{\Psi}$,
there exists a minimal geodesic $(\mu_t)_{t \in [0,1]} \subset \cP^2(M_{\varphi_m}^\Psi)$
along which
\[ H_{\varphi_m}(\mu_t) \le (1-t)H_{\varphi_m}(\mu_0) +tH_{\varphi_m}(\mu_1)
 -\frac{K}{2}(1-t)t W_2(\mu_0,\mu_1)^2 \]
holds for all $t \in [0,1]$.

\item 
Take any $\varphi$ with $\theta_{\varphi} \le 2-m$ such that $(M,\omega,\varphi,\Psi)$ is admissible.
Then, for any $\mu_0,\mu_1 \in \cP^2(M)$
with $\mu_0[M^{\Psi}_{\varphi}]=\mu_1[M^{\Psi}_{\varphi}]=1$
such that any pair of points $x_i \in \supp\mu_i \cap M^{\Psi}_{\varphi}$
$(i=0,1)$ are joined by some minimal geodesic contained in $M_{\varphi}^{\Psi}$,
there exists a minimal geodesic $(\mu_t)_{t \in [0,1]} \subset \cP^2(M_\varphi^\Psi)$
along which
\begin{equation}\label{eq:phiCD}
H_{\varphi}(\mu_t) \le (1-t)H_{\varphi}(\mu_0) +tH_{\varphi}(\mu_1)
 -\frac{K}{2}(1-t)t W_2(\mu_0,\mu_1)^2
\end{equation}
holds for all $t \in [0,1]$.
\end{enumerate}
\end{theorem}

\begin{proof}
The equivalence between (A) and (B) has been established in \cite[Theorem~4.1]{mCD}.
As (C) $\Rightarrow$ (B) is trivial (recall $\theta_{\varphi_m}=2-m$, see Subsection~\ref{ssc:phi_m}),
it is enough to show (A) $\Rightarrow$ (C).

We can assume that both $H_\varphi(\mu_0)$ and $H_\varphi(\mu_1)$ are finite,
otherwise the assertion $(\ref{eq:phiCD})$ is obvious.
We first consider the case where both $\mu_0$ and $\mu_1$ are absolutely continuous
with respect to $\omega$.
By Theorem~\ref{th:FG}, there exists an almost everywhere twice differentiable function $\phi:\Omega \lra \R$
with $\mu_0[\Omega]=1$ such that the map $\cT_t(x):=\exp_x(t\nabla\phi(x))$ ($t \in [0,1]$)
gives the unique minimal geodesic $\mu_t:=(\cT_t)_{\sharp}\mu_0$ from $\mu_0$ to $\mu_1$.
Given $\mu_0$-a.e.\ $x$, $\cT_1(x)$ is not a cut point of $x$ due to \cite[Proposition~4.1]{CMS},
so that the geodesic $(\cT_t(x))_{t \in [0,1]}$ is unique and contained in $M_\varphi^\Psi$.
Put $\mu_t=\rho_t \omega$ and $\bJ^{\omega}_t(x):=e^{f(x)-f(\cT_t(x))}\det(D\cT_t(x))$.
By the change of variables formula with the Jacobian equation
$(\rho_t \circ \cT_t) \bJ^{\omega}_t =\rho_0$ $\mu_0$-a.e.\ (Theorem~\ref{th:MA}),
we deduce that
\begin{align*}
H_{\varphi}^{(1)}(\mu_t) 
&:= \int_M h_{\varphi}(\rho_t) \,d\omega
 =\int_M h_{\varphi} \big( \rho_t(\cT_t) \big) \bJ^{\omega}_t \,d\omega \\
&=\int_M h_{\varphi} \bigg( \frac{\rho_0}{\bJ^{\omega}_t} \bigg)
 \frac{\bJ^{\omega}_t}{\rho_0} \,d\mu_0
 =\int_M \psi \bigg( \bigg( \frac{\bJ^{\omega}_t}{\rho_0} \bigg)^{1/N} \bigg) \,d\mu_0,
\end{align*}
where we set $\psi(r):=r^N h_{\varphi}(r^{-N})$.
As Theorem~\ref{th:DC} together with the monotonicity of $\DC_N$ in $m$
(Lemma~\ref{lm:DC_N}) ensures $h_{\varphi} \in \DC_N$,
the function $\psi(r)$ is non-increasing (resp.\ non-decreasing) and convex in $r$
if $N \ge 1$ (resp.\ $N<0$) due to Lemma~\ref{lm:psi_N}.

Then the essential ingredient is the concavity of $N \bJ_t^{\omega}(x)^{1/N}$ as in the next claim.
We give a sketch of the proof for completeness
(see \cite{StII} or \cite{Oint} for a detailed proof, where a more delicate estimate
under $\Ric_N \ge K$ for general $K \in \R$ is discussed).

\begin{claim}\label{cl:mCD}
Under $\Ric_N \ge 0$, $N \bJ_t^{\omega}(x)^{1/N}$ is concave in $t$ for $\mu_0$-a.e.\ $x$.
\end{claim}

\begin{proof}
Take an orthonormal basis $\{e_i\}_{i=1}^n$ of $T_xM$ 
and extend each $e_i$ to the vector field $E_i(t):=D(\cT_t)_x(e_i)$ for $t \in [0,1]$.
Note that every $E_i$ is a Jacobi field along the geodesic $\gamma(t):=\cT_t(x)$,
since $\cT_t$ is a transport along geodesics.
Let us consider the $n \times n$ matrix-valued functions $A(t)=(a_{ij}(t))$ and $B(t)=(b_{ij}(t))$
given by
\[ a_{ij}(t):=\langle E_i(t),E_j(t) \rangle, \qquad
 D_{\dot{\gamma}}E_i(t)=\sum_{j=1}^n b_{ij}(t)E_j(t), \]
where $D_{\dot{\gamma}}$ is the covariant derivative along $\gamma$.
Observe that $\bJ_t^{\omega}(x)=e^{f(x)-f(\gamma(t))}\sqrt{\det A(t)}$.
We see by calculations
$A'=2BA$ and $A''=-2\Ric_{\dot{\gamma}} +2B^2A$,
where we set $\Ric_{\dot{\gamma}}:=(\langle R(E_i,\dot{\gamma})\dot{\gamma},E_j \rangle)_{i,j=1}^n$
and $R$ stands for the Riemannian curvature tensor of $(M,g)$.
Combining these with the symmetry of $B$, we have
\begin{align*}
\frac{d^2}{dt^2} \Big[ (\det A)^{1/2n} \Big]
 = \bigg\{ \frac{(\trace B)^2}{n}-\trace(\Ric_{\dot{\gamma}} A^{-1})-\trace(B^2) \bigg\}
 \frac{(\det A)^{1/2n}}{n}
 \le -\frac{\Ric(\dot{\gamma})}{n} (\det A)^{1/2n}.
\end{align*}
Put
\[ v(t):=\bJ_t^{\omega}(x)^{1/N}, \quad v_1(t):=e^{\{f(x)-f(\gamma(t))\}/(N-n)},
 \quad v_2(t):=\{\det A(t) \}^{1/2n}. \]
As $v=v_1^{(N-n)/N} v_2^{n/N}$, we obtain
\begin{align*}
Nv^{-1}v'' &= (N-n)v_1^{-1} v''_1 +nv_2^{-1} v''_2
 -\frac{(N-n)n}{N}(v_1^{-1}v'_1 -v_2^{-1} v'_2)^2 \\
&\le -(f \circ \gamma)''+\frac{\{ (f \circ \gamma)' \}^2}{N-n} -\Ric(\dot{\gamma}) =-\Ric_N(\dot{\gamma}).
\end{align*}
Note that the range of $N \in (-\infty,0) \cup [n,\infty)$ is essential here
for making $(N-n)/N$ nonnegative.
Thus the assumption $\Ric_N \ge 0$ implies $Nv'' \le 0$, so that
$Nv=N\bJ_t^{\omega}(x)^{1/N}$ is concave in $t$.
$\hfill \diamondsuit$
\end{proof}

Therefore we have, as $\bJ_0^{\omega} \equiv 1$,
\begin{align}
\psi \big( (\bJ^{\omega}_t/\rho_0)^{1/N} \big)
&\le \psi \big( (1-t)(1/\rho_0)^{1/N} +t(\bJ^{\omega}_1/\rho_0)^{1/N} \big) \nonumber\\
&\le (1-t)\psi \big( (1/\rho_0)^{1/N} \big) +t\psi \big( (\bJ^{\omega}_1/\rho_0)^{1/N} \big) \label{eq:0}
\end{align}
$\mu_0$-a.e..
This implies  
\begin{align}\label{eq:1}
H^{(1)}_\varphi(\mu_t)
&= \int_M \psi \bigg( \bigg( \frac{\bJ^{\omega}_t}{\rho_0} \bigg)^{1/N} \bigg) \,d\mu_0 \nonumber\\
&\le \int_M \bigg\{ (1-t)\psi \bigg( \bigg( \frac{1}{\rho_0} \bigg)^{1/N} \bigg)
 +t\psi \bigg( \bigg( \frac{\bJ^{\omega}_1}{\rho_0} \bigg)^{1/N} \bigg) \bigg\} \,d\mu_0 \nonumber\\ 
&= (1-t)H^{(1)}_\varphi(\mu_0) + tH^{(1)}_\varphi(\mu_1). 
\end{align}
On the other hand, it follows from $\Hess\Psi \geq K$ that
\begin{align*}
\int_M \Psi \,d\mu_t &=\int_M \Psi(\cT_t) \,d\mu_0 
\le \int_M \bigg\{ (1-t)\Psi(\cT_0)+t\Psi(\cT_1)
 -\frac{K}{2}(1-t)td_g(\cT_0,\cT_1)^2 \bigg\} d\mu_0,
\end{align*}
and hence
\begin{align}\label{eq:2}
H^{(2)}_\varphi(\mu_t) &:=-\int_M h'_{\varphi}(\sigma) \,d\mu_t %\nonumber\\
\le (1-t)H^{(2)}_\varphi(\mu_0) +tH^{(2)}_\varphi(\mu_1)
 -\frac{K}{2}(1-t)t W_2(\mu_0,\mu_1)^2. 
\end{align}
Combining \eqref{eq:1} with \eqref{eq:2}, we obtain the desired inequality \eqref{eq:phiCD}.

Let us next consider the case where $\mu_0$ or $\mu_1$ has nontrivial singular part.
We may assume $L_{\varphi}<\infty$, otherwise \eqref{eq:phiCD} trivially holds
by the definition of $h'_{\varphi}(\infty)$.
Decompose $\mu_0$ and $\mu_1$ as $\mu_0=\rho_0 \omega+\mu_0^s$ and
$\mu_1=\rho_1 \omega+\mu_1^s$, and take an optimal coupling $\pi$ of $\mu_0$ and $\mu_1$.
Let $p_1,p_2:M \times M \lra M$ denote the projections to the first and second components.
Now, $\pi$ is decomposed into four parts
\[ \pi=\pi_{aa}+\pi_{as}+\pi_{sa}+\pi_{ss} \]
such that $(p_1)_{\sharp}(\pi_{aa})$, $(p_1)_{\sharp}(\pi_{as})$, $(p_2)_{\sharp}(\pi_{aa})$ and
$(p_2)_{\sharp}(\pi_{sa})$ are absolutely continuous
and that $(p_1)_{\sharp}(\pi_{sa})$, $(p_1)_{\sharp}(\pi_{ss})$, $(p_2)_{\sharp}(\pi_{as})$
and $(p_2)_{\sharp}(\pi_{ss})$ are singular or null measures.
We divide optimal transport between $\mu_0$ and $\mu_1$ into two parts,
corresponding to $\pi-\pi_{ss}$ and $\pi_{ss}$.
For $\hat{\mu}_0:=(p_1)_{\sharp}(\pi-\pi_{ss})$ and $\hat{\mu}_1:=(p_2)_{\sharp}(\pi-\pi_{ss})$, 
Theorem~\ref{th:FG} guarantees the existence of a geodesic
\[ \hat{\mu}_t  \in (1-\pi_{ss}[M \times M]) \cdot \cP^2_{\ac}(M,\omega),
 \qquad t \in (0,1), \]
(i.e., $\hat{\mu}_t[M]=1-\pi_{ss}[M \times M]$) such that
$\hat{\mu}_t[M^{\Psi}_{\varphi}]=\hat{\mu}_t[M]$.
Setting $\hat{\mu}_t=\hat{\rho}_t\omega$, we observe
\begin{align*}
\int_M h_\varphi(\hat{\rho}_t) \,d\omega 
&\le  (1-t) \int_M h_\varphi(\rho_0) \,d\omega +t\int_M  h_\varphi(\rho_1) \,d\omega,\\
 -\int_M h'_{\varphi} (\sigma) \,d\hat{\mu}_t
&\le -(1-t)\int_M h'_{\varphi}(\sigma) \,d\hat{\mu}_0 -t\int_M h'_{\varphi}(\sigma) \,d\hat{\mu}_1\\
&\quad -\frac{K}{2}(1-t)t \int_{M \times M} d_g(x,y)^2 \,d(\pi-\pi_{ss})(x,y).
\end{align*}
To be precise, in the first inequality, we used $h_{\varphi} \le 0$ along the transports corresponding to
$\pi_{as}$ and $\pi_{sa}$. 
By Proposition~\ref{pr:LV}, we find a minimal geodesic
\[ \tilde{\mu}_t=\tilde{\rho}_t \omega +\tilde{\mu}_t^s \in \pi_{ss}[M \times M] \cdot \cP^2(M) \]
from $\tilde{\mu}_0:=(p_1)_{\sharp}(\pi_{ss})$ to $\tilde{\mu}_1:=(p_2)_{\sharp}(\pi_{ss})$
realized through a family of geodesics in $M_\varphi^\Psi$.
Then the condition $\Hess\Psi \ge K$ implies
\begin{align*}
-\int_M h'_{\varphi} (\sigma) \,d\tilde{\mu}_t 
&\le -(1-t)\int_M h'_{\varphi} (\sigma) \,d\tilde{\mu}_0
 -t\int_M h'_{\varphi} (\sigma)  \,d\tilde{\mu}_1 \\
&\quad-\frac{K}{2}(1-t)t \int_{M \times M} d_g(x,y)^2 \,d\pi_{ss}(x,y).
\end{align*}
We put $\mu_t:=\hat{\mu}_t+\tilde{\mu}_t$ and conclude that
\begin{align*}
H_\varphi(\mu_t) 
&= \int_M h_{\varphi} ( \hat{\rho}_t+\tilde{\rho}_t)\,d\omega -\int_M  h'_{\varphi} (\sigma )\,d\mu_t 
\leq \int_M h_{\varphi} ( \hat{\rho}_t)\,d\omega -\int_M  h'_{\varphi} (\sigma ) \,d(\hat{\mu}_t+\tilde{\mu}_t) \\
&\le (1-t)H_\varphi(\mu_0) +tH_\varphi(\mu_1) -\frac{K}{2}(1-t)tW_2(\mu_0,\mu_1)^2,
\end{align*}
where we used the fact that $h_{\varphi}$ is non-increasing (since $L_{\varphi}<\infty$) in the second line.
$\qedd$
\end{proof}

\begin{remark}\label{rm:conn}
Recall that $M^{\Psi}_{\varphi}=M^{\Psi}_{\varphi_m}=M$ if $l_{\varphi}=-\infty$ by definition
(and admissibility).
Hence the condition in (B) and (C) that $\supp\mu_0$ and $\supp\mu_1$
are connected in $M^{\Psi}_{\varphi}$ is nontrivial only if $l_{\varphi}>-\infty$.
Even when $l_{\varphi}>-\infty$, Lemma~\ref{lm:K>0}(i) guarantees that
$M^{\Psi}_{\varphi}$ is totally convex if $\Hess \Psi \ge K>0$.
\end{remark}

In the limit case of $N=\infty$ ($m=1$), we can follow the proof of 
(A) $\Rightarrow$ (C) using $\psi(r)=e^r h_{\varphi}(e^{-r})$
as well as the concavity of $\log(\bJ_t^{\omega}(x))$.
However, the implication (B) $\Rightarrow$ (A) is not true.
This is because the two weights $f$ and $\Psi$ are synchronized as $\nu=e^{-f-\Psi}\vol_g$
and we can control only the behavior of $f+\Psi$
(see the proof of (B) $\Rightarrow$ (A) sketched in the next subsection).

Instead, one can see from Theorem~\ref{th:LVCD} that $\Ric_{\infty} \ge K$ (of $(M,\omega)$)
implies the $\lambda(K,U)$-convexity of $U_{\omega}$ for all $U \in \DC_{\infty}$, where
\[ \lambda(K,U) :=\inf_{r>0}K\frac{rU'(r)-U(r)}{r} =\left\{ \begin{array}{cl}
 K\lim_{r \downarrow 0} \{rU'(r)-U(r)\}/r & \text{if}\ K>0, \smallskip\\
 0 & \text{if}\ K=0, \smallskip\\
 K\lim_{r \to \infty} \{ rU'(r)-U(r) \}/r & \text{if}\ K<0
 \end{array} \right. \]
(\cite[Theorem~7.3]{LV2}, \cite[Theorem~30.5]{Vi2}).

%%%%%%%%%%%%%%%%%%%%%%%%%%%%%%%%
\section{Functional inequalities}\label{sc:func}%%%%%%%%%%%

If $\Ric_{N_{\varphi}} \ge 0$ and $\Hess\Psi \ge K$ for some $K>0$, then we can obtain variants of
the Talagrand inequality, the HWI inequality, the logarithmic Sobolev inequality and the global Poincar\'e inequality.
These are derived from fundamental properties of convex functions along the lines of
\cite{OV} and \cite[Section~6]{LV2} (see also \cite[Section~5]{mCD} where
we studied the special case of the $m$-relative entropies).
We will impose only the strictly weaker condition $\Hess H_{\varphi} \ge K>0$ (for single $\varphi$)
in the $\varphi$-Talagrand inequality for the use in the next section.

For $\mu=\rho\omega \in \cP^2_{\ac}(M,\omega)$ with $\mu[M^{\Psi}_{\varphi}]=1$,
we define the \emph{$\varphi$-relative Fisher information} with respect to $\nu=\sigma\omega$ by
\begin{equation}\label{eq:Im}
I_{\varphi} (\mu)
 := \int_M \left|  \nabla [\ln_{\varphi}(\rho)-\ln_{\varphi}(\sigma)]  \right|^2 \,d\mu
 =\int_M \left| \frac{\nabla\rho}{\varphi(\rho)}+\nabla\Psi \right|^2 d\mu
\end{equation}
provided that it is well-defined, otherwise we set $I_{\varphi}(\mu):=\infty$.
This quantity describes the directional derivatives of $H_{\varphi}$ as follows.
(At this point the treatment in \cite{mCD} was somewhat too rough,
the argument in the present paper is correct.)

\begin{proposition}[Directional derivatives of $H_{\varphi}$]\label{pr:dHm}
Let $(M,\omega,\varphi,\Psi)$ be an admissible space with
$\Ric_{N_{\varphi}} \ge 0$ and $\Hess\Psi \ge K$ on $M^{\Psi}_{\varphi}$ for some $K \in \R$,
and $\mu=\rho\omega \in \cP^2_{\ac}(M,\omega)$ be such that
$\mu[M_{\varphi}^{\Psi}]=1$, $H_{\varphi}(\mu)<\infty$,
$\rho h'_{\varphi}(\rho)-h_{\varphi}(\rho) \in H^1_{\loc}(M)$
and that $|\nabla\rho/\varphi(\rho)|+|\nabla\Psi| \in L^2(M,\mu)$.
Take a minimal geodesic $(\mu_t)_{t \in [0,1]} \subset \cP^2(M)$
emanating from $\mu_0=\mu$ generated by a locally semi-convex function $\phi:M \lra \R$
as $\mu_t=(\cT_t)_{\sharp}\mu$ with $\cT_t(x)=\exp_x(t\nabla\phi(x))$.
If $\theta_{\varphi}<1$, then we further suppose that $\supp\mu_0$ and $\supp\mu_1$ are compact.
Then we have
\begin{equation}\label{eq:dHm}
\liminf_{t \downarrow 0}\frac{H_{\varphi}(\mu_t)-H_{\varphi}(\mu)}{t}
 \ge \int_M \bigg\langle \frac{\nabla\rho}{\varphi(\rho)}+\nabla\Psi,\nabla\phi \bigg\rangle \,d\mu.
\end{equation}
Moreover, equality holds in \eqref{eq:dHm}
$($with $\lim_{t \downarrow 0}$ in place of $\liminf_{t \downarrow 0})$ if $\phi \in C_c^{\infty}(M)$.
\end{proposition}

\begin{proof}
We first deduce equality in \eqref{eq:dHm} for $\phi \in C_c^{\infty}(M)$.
Put $\mu_t=\rho_t \omega$ and $\bJ^{\omega}_t:=e^{f-f(\cT_t)}\det(D\cT_t)$
as in the proof of Theorem~\ref{th:mCD}.
We follow the calculation in Theorem~\ref{th:mCD} and see
\[ H_{\varphi}(\mu_t)-H_{\varphi}(\mu)
 =\int_M \left\{ h_{\varphi}\bigg( \frac{\rho}{\bJ^{\omega}_t} \bigg) \bJ_t^{\omega}-h_{\varphi}(\rho) \right\} d\omega
 -\int_M \{ h'_{\varphi}(\sigma \circ \cT_t)-h'_{\varphi}(\sigma) \} \,d\mu. \]
By the convexity \eqref{eq:0} of $h_{\varphi}(\rho/\bJ^{\omega}_t) \bJ_t^{\omega}$ in $t$
and $\Hess\Psi \ge K$, we can apply the monotone convergence theorem and obtain
\[ \lim_{t \downarrow 0} \frac{H_{\varphi}(\mu_t)-H_{\varphi}(\mu)}{t}
 =\int_M \lim_{t \downarrow 0} \frac{h_{\varphi}(\rho/\bJ^{\omega}_t) \bJ_t^{\omega}-h_{\varphi}(\rho)}{t} \,d\omega
 +\int_M \langle \nabla\Psi,\nabla\phi \rangle \,d\mu. \]
Note that, by using the weighted Laplacian $\Delta^{\omega}$
introduced at the beginning of Subsection~\ref{ssc:heat} below,
\begin{align*}
\lim_{t \downarrow 0} \frac{\bJ^{\omega}_t -1}{t}
&= \lim_{t \downarrow 0} \frac{e^{f-f(\cT_t)}\det(D\cT_t)-1}{t}
 =\trace(\Hess \phi) -\langle \nabla\phi,\nabla f \rangle \\
&=\Delta\phi -\langle \nabla\phi,\nabla f \rangle =\Delta^{\omega}\phi.
\end{align*}
Thus we have
\[ \lim_{t \downarrow 0} \frac{h_{\varphi}(\rho/\bJ^{\omega}_t) \bJ_t^{\omega}-h_{\varphi}(\rho)}{t}
 =\{ h_{\varphi}(\rho) -h'_{\varphi}(\rho)\rho \}
 \lim_{t \downarrow 0} \frac{\bJ^{\omega}_t -1}{t}
 =\{ h_{\varphi}(\rho) -h'_{\varphi}(\rho)\rho \big\} \Delta^{\omega}\phi. \]
Therefore we conclude, by the integration by parts for $\Delta^{\omega}$ (since $\phi \in C_c^{\infty}(M)$),
\begin{align*}
\lim_{t \downarrow 0}\frac{ H_\varphi(\mu_t)-H_\varphi(\mu)}{t}
&= \int_M \langle \nabla[h'_{\varphi}(\rho)\rho -h_{\varphi}(\rho)]+\rho\nabla\Psi, \nabla\phi \rangle \,d\omega \\
&= \int_M \bigg\langle \frac{\nabla\rho}{\varphi(\rho)}+\nabla\Psi,\nabla\phi \bigg\rangle \,d\mu.
\end{align*}

In the case of $\phi \not\in C_c^{\infty}(M)$, we need to take care about the last step
of integration by parts.
If $\theta_{\varphi} \ge 1$ (equivalently, $N_{\varphi} \in [n,\infty] \cap (2,\infty]$),
then we can directly apply \cite[Theorem~23.14]{Vi2} to see \eqref{eq:dHm}.
For $\theta_{\varphi}<1$, the same proof (Step~3 in \cite[Theorem~23.14]{Vi2}) still works
provided that $\supp\mu_0$ and $\supp\mu_1$ are compact.
$\qedd$
\end{proof}

\begin{remark}\label{rm:dHm}
Let us add some more remarks on the case of $\theta_{\varphi}<1$.
A large part of the proof of \cite[Theorem~23.14]{Vi2} also works in this case
(even without the approximation procedure based on \cite[Proposition~17.7]{Vi2}).
Proposition~\ref{lm:case}(i) ensures $l_{\varphi}>-\infty$
so that $u_{\varphi}(r) \ge l_{\varphi}r$ for all $r \ge 0$,
and Lemma~\ref{lem:mono} shows
\[ \frac{s}{\varphi(s)} \le \bigg( \frac{s}{t} \bigg)^{1-\theta_{\varphi}} \frac{t}{\varphi(t)}
 \le \frac{t}{\varphi(t)} \]
for all $0<s<t$, which corresponds to \cite[(23.52)]{Vi2} with $A=\infty$
(hence (23.53) and (23.54) are unnecessary).
Note that $p'(s)$ in \cite{Vi2} is $s/\varphi(s)$ in our context.
The only problem is that $s/\varphi(s)$ is never bounded for large $s$,
as seen from the model case of $s/\varphi_m(s)=s^{m-1}$ with $m>1$.
The boundedness is used to guarantee $p'(\rho) \in L^2(M,\mu)$,
so that we can assume $\rho/\varphi(\rho) \in L^2(M,\mu)$
instead of the compactness of $\supp\mu_0 \cup \supp\mu_1$
in Proposition~\ref{pr:dHm} above.
\end{remark}

We assume $\nu[M]=1$ by scaling (recall Remark~\ref{rem:prob}),
and prove functional inequalities associated with $H_{\varphi}$.

\begin{theorem}\label{thm:inequ}
Let $(M,\omega,\varphi, \Psi)$ be admissible with $\nu \in \cP_{\ac}^2(M,\omega)$.
We set $H_{\nu}:=H_\varphi(\nu)$ for brevity.

\begin{enumerate}[{\rm (i)}]
\item{\rm ($\varphi$-Talagrand inequality)}
Suppose that $\Hess H_{\varphi} \ge K$ for some $K>0$.
For any $\mu \in \cP^2(M)$, we have
\begin{equation}\label{eq:tal} 
W_2(\mu,\nu) \le \sqrt{\frac{2}{K}(H_\varphi(\mu)-H_{\nu})}. 
\end{equation}

\item{\rm ($\varphi$-HWI, $\varphi$-logarithmic Sobolev inequalities)}
Assume $\Ric_{N_{\varphi}} \ge 0$ and $\Hess\Psi \ge K$ on $M^{\Psi}_{\varphi}$ for some $K>0$.
Given $\mu=\rho\omega \in \cP^2_{\ac}(M)$ with $\mu[M^{\Psi}_{\varphi}]=1$
such that $H_{\varphi}(\mu)<\infty$ and that $\rho$ is locally Lipschitz, we have 
\begin{align}
H_\varphi(\mu)-H_\nu
&\le \sqrt{I_\varphi(\mu)} \cdot W_2(\mu,\nu) -\frac{K}{2}W_2(\mu,\nu)^2, \label{eq:HWI} \\
H_\varphi(\mu)-H_\nu &\le \frac{1}{2K}I _\varphi(\mu). \label{eq:LS}
\end{align}

\item{\rm ($\varphi$-global Poincar\'e inequality)}
Let $(M,g)$ be compact and $\varphi$ be $C^1$, and
assume $\Ric_{N_{\varphi}} \ge 0$ and $\Hess\Psi \ge K$ on $M^{\Psi}_{\varphi}$ for some $K>0$.
Then for any Lipschitz function $w:M_\varphi^\Psi \lra \R$ such that $\int_{M_\varphi^\Psi}w \,d\nu=0$, we have
\begin{equation}\label{eq:sakana}
\int_{M_\varphi^\Psi} \frac{w^2\sigma}{ \varphi(\sigma)} \,d\nu
 \le \frac{1}{K}\int_{M_\varphi^\Psi} \left|\nabla \left( \frac{w\sigma}{\varphi(\sigma)} \right) \right|^2 \,d\nu.
\end{equation}
\end{enumerate}
\end{theorem}

\begin{proof}
We first remark that $M^{\Psi}_{\varphi}$ is totally convex if $\Hess\Psi \ge K>0$
(see Lemma~\ref{lm:K>0} and Remark~\ref{rm:conn} as well).
Thus $M^{\Psi}_{\varphi}$ is totally convex in (ii) and (iii).

(i)
There is nothing to prove if $H_\varphi(\mu)=\infty$, so that we assume $H_\varphi(\mu)<\infty$.
By the hypothesis $\Hess H_{\varphi} \ge K$, there is a minimal geodesic
$(\mu_t)_{t \in [0,1]} \subset \cP^2(M)$ from $\mu_0=\mu$ to $\mu_1=\nu$ such that
\begin{equation}\label{eq:TaH}
0 \leq H_\varphi(\mu_t) -H_\nu
\leq (1-t)H_\varphi(\mu) -(1-t)H_\nu-\frac{K}{2}(1-t)tW_2(\mu,\nu)^2
\end{equation}
for all $t \in [0,1]$.
Dividing both sides with $1-t$ and letting $t$ go to $1$, we obtain the desired inequality~\eqref{eq:tal}.

(ii)
As the case of $I_\varphi(\mu)=\infty$ is trivial, we assume $I_\varphi(\mu)<\infty$.
For the minimal geodesic $(\mu_t)_{t \in [0,1]}$ from $\mu_0=\mu$ to $\mu_1=\nu$,
Theorem~\ref{th:FG} ensures that $\mu_t \in \cP^2_{\ac}(M,\omega)$ for all $t \in [0,1]$ and
there is a locally semi-convex function $\phi$ such that $\mu_t=(\cT_t)_{\sharp}\mu$
with $\cT_t(x)=\exp_x(t\nabla\phi(x))$.
Due to $(\ref{eq:TaH})$, we have 
\begin{equation}\label{eq:LSH}
\frac{ H_\varphi(\mu_t)-H_\varphi(\mu)}{t}
 \le -H_\varphi(\mu)+H_\nu -\frac{K}{2}(1-t)W_2(\mu,\nu)^2.
\end{equation}
Moreover, Proposition~\ref{pr:dHm} shows that
\[ \liminf_{t \downarrow 0}\frac{ H_\varphi(\mu_t)-H_\varphi(\mu)}{t}
 \ge \int_M \langle \nabla[\ln_\varphi(\rho)- \ln_\varphi( \sigma)], \nabla \phi\rangle) \,d\mu. \]
We remark that $M^{\Psi}_{\varphi}$ is bounded if $\theta_{\varphi}<1$
(by Proposition~\ref{lm:case}(i) and Lemma~\ref{lm:K>0}(i)),
so that Proposition~\ref{pr:dHm} is certainly available.
We obtain from the Cauchy--Schwarz inequality that
\begin{align*}
\liminf_{t \downarrow 0}\frac{ H_\varphi(\mu_t)-H_\varphi(\mu)}{t}
&\ge -\bigg( \int_M | \nabla[\ln_\varphi(\rho)- \ln_\varphi( \sigma)]|^2 \,d\mu \bigg)^{1/2}
 \bigg( \int_M |\nabla\phi|^2 \,d\mu \bigg)^{1/2} \\
&= -\sqrt{I_\varphi(\mu)} \cdot W_2(\mu,\nu),
\end{align*}
where the last equality follows from
\[ |\nabla\phi (x)| =d_g\big( x, \exp_x(\nabla \phi) \big)
 =d_g \big( \cT_0(x),\cT_1(x) \big) \qquad \mu\text{-a.e. }x. \]
Combining this with $(\ref{eq:LSH})$, we obtain \eqref{eq:HWI}.
By completing the square, we deduce \eqref{eq:LS} from \eqref{eq:HWI} as
\[ \sqrt{I_\varphi(\mu)} \cdot W_2(\mu,\nu) -\frac{K}{2}W_2(\mu,\nu)^2
 =-\frac{K}{2}\left( W_2(\mu,\nu) -\frac{1}{K}\sqrt{I_\varphi(\mu)}\right)^2
 +\frac{I_\varphi(\mu)}{2K}
 \le \frac{I_\varphi(\mu)}{2K}. \]

(iii)
For small $\ve >0$, we put $\mu=\rho \omega:=(1+\ve w)\sigma\omega$.
We remark that $H_\varphi(\mu)$ and  $I_\varphi(\mu)$ are finite as $M$ is compact.
It follows from \eqref{eq:LS} that
\[  \int_{M^\Psi_\varphi} \{u_\varphi (\rho)-u_\varphi (\sigma)-u'_\varphi(\sigma) (\rho-\sigma)\} \,d\omega
 \le \frac{1}{2K}
 \int_{M^\Psi_\varphi} \left| \nabla [\ln_{\varphi}(\rho)-\ln_{\varphi}(\sigma)] \right|^2 \,d\mu. \]
On the one hand, we have by expansion
\[ u_\varphi (\rho)-u_\varphi (\sigma)-u'_\varphi(\sigma) (\rho-\sigma)
 =\frac{(\rho-\sigma)^2 }{2} u''_\varphi(\sigma) +O\left( (\rho-\sigma)^3 \right)
 =\frac{\ve^2 w^2 \sigma^2}{2\varphi(\sigma)} +O( \ve^3), \]
where $O(\ve^3)$ is uniform on $M$ (for fixed $w$) thanks to the compactness of $M$.
On the other hand, it holds
\begin{align*}
\left| \nabla [\ln_{\varphi}(\rho)-\ln_{\varphi}(\sigma)] \right|^2
&=\left| \nabla \big[ (\rho-\sigma) \ln'_\varphi(\sigma) +O\big( (\rho-\sigma)^2 \big) \big] \right|^2 
=\left| \nabla \left( \frac{\ve w\sigma}{\varphi(\sigma)} \right) +O(\ve ^2) \right|^2 \\
&=\ve^2 \left| \nabla \left( \frac{w\sigma}{\varphi(\sigma)} \right) \right|^2 + O(\ve^3).
\end{align*}
Thus we have, letting $\ve$ go to zero,
\[ \int_{M^\Psi_\varphi} \frac{w^2 \sigma }{\varphi(\sigma)} \,d\nu
 =\int_{M^\Psi_\varphi} \frac{w^2 \sigma^2 }{\varphi(\sigma)} \,d\omega
\le \frac{1}{K} \int_{M^\Psi_\varphi} \left| \nabla \left( \frac{w\sigma}{\varphi(\sigma)} \right) \right|^2 \,d\nu. \]
$\qedd$
\end{proof}

The $\varphi$-Talagrand inequality~\eqref{eq:tal} is regarded as a comparison between
the distance functions appearing in Wasserstein geometry and information geometry,
since the square root of the Bregman divergence behaves like a distance function (see Subsection~\ref{ssc:inf}).
Note that, in the $\varphi$-global Poincar\'e inequality~\eqref{eq:sakana},
the usual global Poincar\'e inequality $\int_M w^2 \,d\nu \le K^{-1} \int_M |\nabla w|^2 \,d\nu$
is indeed recovered when $\varphi(s)=\varphi_1(s)=s$.
Other inequalities are also clearly reduced to the usual ones for $\varphi=\varphi_1$.

%%%%%%%%%%%%%%%%%%%%%%%%%%%%%%%%%%%%%%%%%
\section{Concentration of measures}\label{sc:conc}%%%%%%%%%%%%%%%

The aim of this section is to derive  the concentration of measures from the $\varphi$-Talagrand inequality~\eqref{eq:tal}.
Let us assume $\nu[M]=1$ (see Remark~\ref{rem:prob}) and define the \emph{concentration function} by
\[ \alpha(r) =\alpha_{(M,\nu)}(r)
 :=\sup \big\{ 1-\nu[B(A,r) ] \,|\, A \subset M \text{: measurable},\, \nu[A] \ge 1/2 \big\} \]
for $r>0$, where
$B(A,r):=\{ y \in M \,|\, \inf_{x \in A}d_g(x,y)<r \}$.
The function $\alpha$ describes how the probability measure $\nu$ is concentrating
on the neighborhood of an \emph{arbitrary} set $A$ of half the total measure in a quantitative way
(in other words, a kind of large (or moderate) deviation principle).
Equivalently, $\alpha$ measures how any $1$-Lipschitz function is close to the constant function
at its mean.
We refer to the excellent book~\cite{Le} for an introduction to the concentration
of measure phenomenon.

In the classical case of $\varphi_1(s)=s$, the Talagrand inequality~\eqref{eq:tal} implies the
\emph{normal concentration} $\alpha(r) \le C\exp(-cr^2)$ or equivalently
$\alpha(r)^{-1} \ge C^{-1} \exp(cr^2)$ 
with constants $c,C>0$ depending only on $K$ (see~\cite[Section~6.1]{Le}).
For general $\varphi$, we will similarly derive from \eqref{eq:tal} the \emph{$m$-normal concentration}
involving the $m$-exponential function $e_m$ (see Subsection~\ref{ssc:phi_m}).
Precisely, we have
$\alpha(r) \le Ce_m(-cr^2)$ with $m=m(\varphi) \le 2-\theta_{\varphi}$ if $\theta_{\varphi}>1$,
and $\alpha(r)^{-1} \ge C^{-1} e_m(cr^2)$ with $m=m(\varphi) \ge 2-\theta_{\varphi}$ if $\theta_{\varphi} \le 1$
(Theorem~\ref{thm:conc}).

%%%%%%%%%%%%%%%%%%%%%%%%%%%%%
\subsection{General estimate}%%%%%%%%%%%%%

For each measurable set $A \subset M$ with $0<\nu[A]<\infty$,
denote the normalized restriction of $\nu$ on $A$ by
\[ \nu_A :=\frac{\chi_A}{\nu[A]}\nu \ \in \cP_{\ac}(M,\omega). \]
To analyze its entropy $H_{\varphi}(\nu_A)$, we introduce the function
\[ U(\xi,t) :=u_{\varphi} \!\left( \frac{\xi}{t} \right) -\frac{\xi}{t} \ln_{\varphi}(\xi),
 \qquad (\xi,t) \in (0,\infty) \times (0,1]. \]
Note that $H_\varphi(\nu_A)=\int_A U(\sigma, \nu[A]) \,d\omega$.
To be precise, we can set $U(\sigma,\nu[A]):=0$ on $M \setminus M^{\Psi}_{\varphi}$
thanks to the following lemma.

\begin{lemma}\label{lm:U}
If $\theta_\varphi<2$, then we have $\lim_{\xi\downarrow 0} U(\xi,t)=0$ for every $t \in (0,1]$.
\end{lemma}

\begin{proof}
Note that
\begin{align}
U(\xi,t) &=\int_0^{\xi/t} \{ \ln_{\varphi}(s) -\ln_{\varphi}(\xi) \} \,ds
 = \int_0^{\xi/t} \int_{\xi}^s \frac{1}{\varphi(r)} \,drds \nonumber\\
&= -\int_0^{\xi} \int_s^{\xi} \frac{1}{\varphi(r)} \,drds
 +\int_{\xi}^{\xi/t} \int_{\xi}^s \frac{1}{\varphi(r)} \,drds = -\int_0^{\xi/t}  \frac{r}{\varphi(r)} \,dr +\frac{\xi}{t} \int_{\xi}^{\xi/t} \frac{1}{\varphi(r)} \,dr.
 \label{eq:U}
\end{align}
Then \eqref{del} (deduced from Theorem~\ref{th:DC}) shows the claim.
$\qedd$
\end{proof}

\begin{lemma}\label{lem:hojo}
Assume $\theta_\varphi<2$.
For any measurable set $A$ with $\nu[A] \ge 1/2$, we have $H_{\varphi}(\nu_A) \le 0$.
In particular, it holds $H_{\varphi}(\nu) \le 0$ if $\nu[M]=1$.
\end{lemma}

\begin{proof}
For any $\xi>0$, $U(\xi,t)$ is non-increasing in $t \in(0,1]$ since we have
\[ \frac{\del U}{\del t}(\xi,t)
 =-\frac{\xi}{t^2} \ln_{\varphi} \!\left( \frac{\xi}{t} \right) +\frac{\xi}{t^2} \ln_{\varphi}(\xi)
 =\frac{\xi}{t^2} \int_{\xi/t}^{\xi} \frac{1}{\varphi(s)} \,ds \le 0. \]
Similarly, $U(\xi,1/2)$ is non-increasing in $\xi$ due to
\[ \frac{dU}{d\xi}\left(\xi,\frac12\right) =2\ln_{\varphi}(2\xi)-2\ln_{\varphi}(\xi)-\frac{2\xi}{\varphi(\xi)}
 =2\int_{\xi}^{2\xi} \left( \frac{1}{\varphi(s)}-\frac{1}{\varphi(\xi)} \right) \,ds \le 0. \]
Thus we deduce from Lemma~\ref{lm:U} that, for any $\xi>0$ and $t \ge 1/2$,
\[ U(\xi,t) \le U(\xi,1/2) \le \lim_{\xi \downarrow 0} U(\xi,1/2) =0, \]
which shows
$H_{\varphi}(\nu_A) =\int_A U(\sigma,\nu[A]) \,d\omega \le 0$.
$\qedd$
\end{proof}

Next we give an estimate on $H_{\varphi}(\nu_B)$ for $B \subset M$
not necessarily $\nu[B] \ge 1/2$.
Recall \eqref{eq:delta} for the definition of $\delta_{\varphi}$.

\begin{lemma}\label{lem:hojojo}
Assume $\theta_\varphi <\infty$ and $\|\sigma\|_{\infty}<\infty$.
Given any measurable set $B \subset M$ with $0<\nu[B]<\infty$ and any
$\xi_0 \ge \max\{\nu[B],\|\sigma\|_{\infty}\}$, we have
\begin{equation}\label{eq:nu_B}
H_\varphi(\nu_B)
 \le -\nu[B]^{\delta_{\varphi}-2} \ln_\varphi(\nu[B]) \xi_0^{\theta_\varphi-\delta_\varphi}
 \int_B \sigma^{2-\theta_\varphi} \,d\omega.
\end{equation}
\end{lemma}

\begin{proof}
For $t,\xi \in (0,\xi_0]$, we deduce from \eqref{eq:U} that
\[ U(\xi,t) \le \frac{\xi}{t} \int_\xi^{\xi/t} \frac{1}{\varphi(r)} \,dr
 =\frac{\xi^2}{t^2} \int_t^1 \frac{1}{\varphi(s\xi t^{-1})} \,ds, \]
where we changed the variables as $r=s\xi t^{-1}$.
Note that Lemma~\ref{lem:mono} shows that for all $s>0$
\[ \frac{\xi^2}{\varphi(s\xi t^{-1})}
 \le \frac{\xi^{2-\theta_{\varphi}} \xi_0^{\theta_{\varphi}}}{\varphi(s\xi_0 t^{-1})}
 \le \frac{\xi^{2-\theta_{\varphi}} \xi_0^{\theta_{\varphi}}}{\varphi(s)}
 \bigg( \frac{\xi_0}{t} \bigg)^{-\delta_{\varphi}}. \]
Thus we find
\[ U(\xi,t) \le -t^{\delta_\varphi-2} \ln_\varphi(t) \xi_0^{\theta_\varphi-\delta_\varphi} \xi^{2-\theta_\varphi}. \]
This implies
\[ H_\varphi(\nu_B) =\int_B U(\sigma,\nu[B]) \,d\omega
 \le -\nu[B]^{\delta_\varphi-2} \ln_\varphi(\nu[B]) \xi_0^{\theta_\varphi-\delta_\varphi}
 \int_B \sigma^{2-\theta_\varphi} \,d\omega \]
as desired.
$\qedd$
\end{proof}

We remark that, if $\delta_{\varphi} \le 2$, then we have at any $s \in (0,1)$
\[ \frac{d}{ds}\big[ s^{\delta_{\varphi}-2} \ln_{\varphi}(s) \big]
 =s^{\delta_{\varphi}-3} \left\{ (\delta_{\varphi}-2) \ln_{\varphi}(s)+\frac{s}{\varphi(s)} \right\} >0. \]
Therefore the right hand side of \eqref{eq:nu_B} is non-increasing in $\nu[B]$
provided that $\nu$ is a probability measure.

Now we show a general estimate of $\alpha(r)$ under the strict convexity of $H_{\varphi}$.

\begin{proposition}\label{pr:conc}
Assume that $(M,\omega,\varphi, \Psi)$ is admissible, $\nu \in \cP_{\ac}(M,\omega)$,
$\Hess H_{\varphi} \ge K$ for some $K>0$ and that $\|\sigma\|_{\infty}<\infty$.
We set $H_{\nu}:=H_{\varphi}(\nu)\ (\le 0)$ as in Theorem~{\rm \ref{thm:inequ}}.
Then, for any $\xi_0 \geq \max\{1/2, \|\sigma\|_\infty\}$ and any $r>0$ with $\alpha(r)>0$, we have
\begin{equation}\label{eq:concc}
\alpha(r)^{\delta_\varphi-2} \ln_\varphi \!\big( \alpha(r) \big) \xi_0^{\theta_\varphi-\delta_\varphi}
 \cdot \sup_A \int_B \sigma^{2-\theta_\varphi} \,d\omega
 \le -\bigg( \sqrt{\frac{K}{2}}r -\sqrt{-H_\nu} \bigg)^2 -H_\nu,
\end{equation}
where $A \subset M$ runs over all measurable sets of $\nu[A] \ge 1/2$ and we set $B:=M \setminus B(A,r)$.
\end{proposition}

\begin{proof}
Since $\alpha(r) \le 1$ by definition, the left hand side of \eqref{eq:concc} is always nonpositive.
Therefore the assertion is clear if $r^2 \le -8H_\nu/K$.
Suppose $r^2>-8H_\nu/K$, take a measurable set $A \subset M$ with
$\nu[A] \ge 1/2$ and put $B:=M \setminus B(A,r)$.
We also assume $\nu[B]>0$ since we have $\alpha(r)=0$ if $\nu[B]=0$ for all such $A$.

Observe that $W_1(\nu_A,\nu_B) \ge r$ as $d_g(x,y) \ge r$ for all $x \in A$ and $y \in B$.
Note also that $W_1 \le W_2$ holds by the Cauchy--Schwarz inequality.
Then the triangle inequality for $W_1$ and the $\varphi$-Talagrand inequality~\eqref{eq:tal} yield
\[ r \le W_1(\nu_A,\nu_B) \le W_1(\nu_A,\nu)+W_1(\nu,\nu_B)
 \le \sqrt{\frac{2}{K} \left( H_\varphi(\nu_A)-H_\nu \right)} +\sqrt{\frac{2}{K}\left(H_\varphi(\nu_B)-H_\nu\right)}. \]
Applying Lemma~\ref{lem:hojo} gives, as $r^2>-8H_\nu/K$ ensures
$\sqrt{K/2}r \ge \sqrt{-H_{\nu}}$,
\[ H_\varphi(\nu_B) \ge \bigg( \sqrt{\frac{K}2}r -\sqrt{-H_\nu} \bigg)^2+H_\nu. \]
Since $A$ is arbitrary and $\nu[B] \le 1/2 \le \xi_0$,
combining the above estimate with Lemma~\ref{lem:hojojo} yields
\begin{align*}
-\bigg( \sqrt{\frac{K}{2}}r -\sqrt{-H_\nu} \bigg)^2 -H_\nu
&\ge \sup_A \bigg\{ \nu[B]^{\delta_{\varphi}-2} \ln_\varphi(\nu[B]) \xi_0^{\theta_\varphi-\delta_\varphi}
 \int_B \sigma^{2-\theta_\varphi} \,d\omega \bigg\} \\
&\ge \alpha(r)^{\delta_\varphi-2} \ln_\varphi \!\big( \alpha(r) \big) \xi_0^{\theta_\varphi-\delta_\varphi}
 \cdot \sup_A \int_B \sigma^{2-\theta_\varphi} \,d\omega.
\end{align*}
$\qedd$
\end{proof}

%%%%%%%%%%%%%%%%%%%%%%%%%%
\subsection{Concentration of measures}%%%%%

We shall obtain the concentration of $\{(M,\omega,\varphi,\Psi_i)\}_{i \in \N}$,
i.e., $\lim_{i \to \infty}\alpha_{(M,\nu_i)}(r)=0$
for all $r>0$ with $\nu_i:=\exp_{\varphi}(-\Psi_i)\omega$,
under an appropriate condition on the convexity of $H_{\varphi}$ associated with $\Psi_i$.
We first prove an auxiliary lemma.

\begin{lemma}\label{lm:fitt}
Assume that $(M,\omega,\varphi,\Psi)$ is admissible, $\nu \in \cP_{\ac}(M,\omega)$
and that $\|\sigma\|_{\infty}<\infty$.
Set $H_\nu:=H_\varphi(\nu)$ and take arbitrary $\xi_0 \ge \|\sigma\|_\infty$.

\begin{enumerate}[{\rm (i)}]
\item
If $\theta_\varphi  \leq 1$, then we have
\[ \int_M \sigma^{2-\theta_\varphi} \,d\omega \le \xi_0^{1-\theta_\varphi}, \qquad
 H_\nu \ge -\frac{\xi_0}{(2-\theta_\varphi)\varphi(\xi_0)}. \]

\item
If $\theta_\varphi \in(1,3/2)$ and $\omega[M]<\infty$, then we have
\[ \int_M \sigma^{2-\theta_\varphi} \,d\omega \le \omega[M]^{\theta_\varphi-1}, \qquad
 H_\nu \ge -\frac{\xi_0^{\theta_\varphi} \omega[M]^{\theta_\varphi-1}}{(2-\theta_\varphi)\varphi(\xi_0)}. \]
\end{enumerate}
\end{lemma}

\begin{proof}
It follows from \eqref{eq:U} and Lemma~\ref{lem:mono} that
\begin{align*}
H_\nu &=\int_M U(\sigma,1) \,d\omega
 \ge -\int_M \int_0^\sigma \frac{r}{\varphi(r)} \,dr d\omega
 \ge -\int_M \int_0^\sigma \frac{\xi_0^{\theta_\varphi} r^{1-\theta_\varphi}}{\varphi(\xi_0)} \,dr d\omega \\
&=-\frac{\xi_0^{\theta_\varphi}}{\varphi(\xi_0)} \int_M \frac{\sigma^{2-\theta_\varphi}}{2-\theta_\varphi} \,d\omega.
\end{align*}

(i)
The assertion immediately follows from
\[ \int_M \sigma^{2-\theta_\varphi} \,d\omega
 \le \|\sigma\|_{\infty}^{1-\theta_{\varphi}} \int_M \sigma \,d\omega
 \le \xi_0^{1-\theta_\varphi}. \]

(ii)
The H\"older inequality yields that
\[ \int_M \sigma^{2-\theta_\varphi} \,d\omega
 \le \bigg( \int_M \sigma \,d\omega \bigg)^{2-\theta_\varphi} \omega[M]^{\theta_\varphi-1}
 =\omega[M]^{\theta_\varphi-1}, \]
which shows the claim.
$\qedd$
\end{proof}

\begin{theorem}\label{th:conc}
Let $\{(M,\omega,\varphi,\Psi_i)\}_{i \in \N}$ be a sequence of admissible spaces such that
\begin{enumerate}[{\rm (a)}]
\item $\omega[M]<\infty$ if $\theta_{\varphi}>1$,
\item $\nu_i =\sigma_i\omega :=\exp_\varphi(-\Psi_i) \omega \in \cP_{\ac}(M,\omega)$ for all $i$,
\item $\xi_i:=\max\{ 1,\|\sigma_i\|_{\infty} \}<\infty$ for all $i$,
\item $\Hess H^i_{\varphi} \ge K_i$ for some $K_i>0$,
where $H^i_{\varphi}$ is the $\varphi$-relative entropy for $(M,\omega,\varphi,\Psi_i)$,
\item \label{K_i}
$\lim_{i \to \infty}K_i \xi_i^{\delta_{\varphi}-1}=\infty$ if $\theta_{\varphi} \le 1$,
and $\lim_{i \to \infty}K_i \xi_i^{\delta_{\varphi}-\theta_{\varphi}}=\infty$ if $\theta_{\varphi}>1$.
\end{enumerate}
Then the concentration function $\alpha_i(r):=\alpha_{(M,\nu_i)}(r)$ satisfies
$\lim_{i \to \infty}\alpha_i(r)=0$ for all $r>0$.
\end{theorem}

\begin{proof}
Fix $r>0$ and put $\alpha_i:=\alpha_i(r)$ and $H_{\nu_i}:=H_{\varphi}^i(\nu_i)$ for brevity.
It follows from Proposition~\ref{pr:conc} that
\[ \alpha_i^{\delta_\varphi-2} \ln_\varphi(\alpha_i) \xi_i^{\theta_\varphi-\delta_\varphi}
 \le -\bigg\{ \bigg( \sqrt{\frac{K_i}2}r -\sqrt{-H_{\nu_i}} \bigg)^2 +H_{\nu_i} \bigg\}
 \left(\int_M \sigma_i^{2-\theta_\varphi} \,d\omega \right)^{-1}. \]
We have $\ln_\varphi(\alpha_i) \ge \ell_{2-\theta_\varphi}(\alpha_i)$ by \eqref{eq:ln},
so that
\begin{equation}\label{eq:icon}
\alpha_i^{\delta_{\varphi}-2} \ell_{2-\theta_\varphi}(\alpha_i)
 \le -\xi_i^{\delta_{\varphi}-\theta_\varphi} \bigg( \frac{K_i}{2}r^2 -\sqrt{2K_i}\sqrt{-H_{\nu_i}}r \bigg)
 \left(\int_M \sigma_i^{2-\theta_\varphi} \,d\omega \right)^{-1}.
\end{equation}

Now, for $\theta_{\varphi} \le 1$, Lemma~\ref{lm:fitt}(i) and Lemma~\ref{lem:mono} yield
\[ \int_M \sigma_i^{2-\theta_{\varphi}} \,d\omega \le \xi_i^{1-\theta_{\varphi}}, \qquad
 H_{\nu_i} \ge -\frac{\xi_i}{(2-\theta_{\varphi}) \varphi(\xi_i)}
 \ge -\frac{\xi_i^{1-\delta_{\varphi}}}{2-\theta_{\varphi}} \]
since $\xi_i \ge  1$.
Hence the right hand side of \eqref{eq:icon} is bounded from above by (for large $i$)
\[ -\xi_i^{\delta_{\varphi}-1} \bigg( \frac{K_i}{2}r^2
 -\sqrt{2K_i}\sqrt{\frac{\xi_i^{1-\delta_{\varphi}}}{2-\theta_{\varphi}}} r \bigg)
 = -\frac{K_i}{\xi_i^{1-\delta_{\varphi}}} \bigg( \frac{r^2}{2}
 -\sqrt{\frac{2}{2-\theta_{\varphi}} \frac{\xi_i^{1-\delta_{\varphi}}}{K_i}}r \bigg)
 \to -\infty \]
as $i$ goes to infinity due to the condition \eqref{K_i}.
Therefore we obtain
\[ \lim_{i \to \infty} \alpha_i^{-2} \ell_{2-\theta_\varphi}(\alpha_i)
 \le \lim_{i \to \infty} \alpha_i^{\delta_{\varphi}-2} \ell_{2-\theta_\varphi}(\alpha_i)=-\infty, \]
and hence $\lim_{i \to \infty} \alpha_i =0$.

For $\theta_{\varphi}>1$, we similarly deduce from Lemma~\ref{lm:fitt}(ii) that
\[ \int_M \sigma_i^{2-\theta_{\varphi}} \,d\omega \le \omega[M]^{\theta_{\varphi}-1}, \qquad
 H_{\nu_i} \ge -\frac{\xi_i^{\theta_{\varphi}} \omega[M]^{\theta_{\varphi}-1}}{(2-\theta_{\varphi}) \varphi(\xi_i)}
 \ge -\frac{\xi_i^{\theta_{\varphi}-\delta_{\varphi}} \omega[M]^{\theta_{\varphi}-1}}{2-\theta_{\varphi}}. \]
Hence the right hand side of \eqref{eq:icon} is bounded from above by
\begin{align*}
&-\frac{\xi_i^{\delta_{\varphi}-\theta_{\varphi}}}{\omega[M]^{\theta_{\varphi}-1}} \bigg( \frac{K_i}{2}r^2
 -\sqrt{2K_i}\sqrt{\frac{\xi_i^{\theta_{\varphi}-\delta_{\varphi}}\omega[M]^{\theta_{\varphi}-1}}
 {2-\theta_{\varphi}}} r \bigg) \\
&= -\frac{K_i \xi_i^{\delta_{\varphi}-\theta_{\varphi}}}{\omega[M]^{\theta_{\varphi}-1}} \bigg( \frac{r^2}{2}
 -\sqrt{\frac{2\omega[M]^{\theta_{\varphi}-1}}{2-\theta_{\varphi}}
 \frac{\xi_i^{\theta_{\varphi}-\delta_{\varphi}}}{K_i}}r \bigg)
 \ \to -\infty \quad (i \to \infty).
\end{align*}
Thus we have $\lim_{i \to \infty} \alpha_i =0$.
$\qedd$
\end{proof}

\begin{remark}\label{rm:conc}
(1)
For $\varphi_m$ with $m<1$, we have $\theta_{\varphi_m}=\delta_{\varphi_m}$
and hence the condition \eqref{K_i} is reduced to $\lim_{i \to \infty}K_i =\infty$,
whereas the condition $\|\sigma_i\|_{\infty}<\infty$ was implicitly used in our discussion.
See \cite[Section~6]{mCD} for a more precise estimate associated with $\varphi_m$
without assuming $\|\sigma_i\|_{\infty}<\infty$.

(2)
We stress that only $\Hess H^i_{\varphi} \ge K_i$ for single $\varphi$
is assumed in Theorem~\ref{th:conc}, rather than $\Ric_{N_{\varphi}} \ge 0$ and $\Hess \Psi_i \ge K_i$.
If $\Hess \Psi_i \ge K_i$ and $l_{\varphi}>-\infty$, then Lemma~\ref{lm:K>0}(i)
gives a stronger estimate on the diameter of $M_{\varphi}^{\Psi_i}$ as
\[ \diam M_{\varphi}^{\Psi_i}
 \le 2\sqrt{\frac{2}{K_i} \left\{ \ln_{\varphi}(\|\sigma_i\|_{\infty}) -l_{\varphi} \right\}}
 \le 2\sqrt{\frac{2}{K_i} \left\{ \ln_{\varphi}(\xi_i) -l_{\varphi} \right\}}. \]
Indeed, we observe from \eqref{eq:ln} and Lemma~\ref{lem:mono} that
\[ \frac{\ln_{\varphi}(\xi_i)}{K_i}
 \le \frac{\xi_i^{\theta_{\varphi}} \ell_{2-\theta_{\varphi}}(\xi_i)}{K_i \varphi(\xi_i)}
 \le \frac{\xi_i^{\theta_{\varphi}-\delta_{\varphi}} \ell_{2-\theta_{\varphi}}(\xi_i)}{K_i}
 =\frac{\xi_i^{\theta_{\varphi}-\delta_{\varphi}}(\xi_i^{1-\theta_{\varphi}}-1)}{K_i(1-\theta_{\varphi})}, \]
provided that $\theta_{\varphi} \neq 1$.
If $\theta_{\varphi}<1$, then the leading term (as $\xi_i \to \infty$) is
\[ \frac{1}{1-\theta_{\varphi}} \frac{\xi_i^{1-\delta_{\varphi}}}{K_i} \to 0
 \quad (i \to \infty) \]
under the condition \eqref{K_i} in Theorem~\ref{th:conc}.
Similarly, for $\theta_{\varphi}>1$ the leading term is
\[ \frac{1}{\theta_{\varphi}-1} \frac{\xi_i^{\theta_{\varphi}-\delta_{\varphi}}}{K_i} \to 0
 \quad (i \to \infty). \]
Therefore, in both cases, $\lim_{i \to \infty}\diam M^{\Psi_i}_{\varphi}=0$ holds
and it is obviously stronger than $\lim_{i \to \infty}\alpha_{(M,\nu_i)}(r)=0$.
\end{remark}

%%%%%%%%%%%%%%%%%%%%%%%%%%%%%%%%
\subsection{$m(\varphi)$-normal concentration}%%%%%%%%%%%%%%%%%%

In order to derive the $m$-normal concentration for some $m=m(\varphi)$
from the general estimate~\eqref{eq:concc},
we prove a computational lemma on $e_m$ (see also~\cite[Lemma~6.4]{mCD}).
Recall from Subsection~\ref{ssc:phi_m} that
$e_m(\tau)=\exp_{\varphi_m}(\tau)=[1+(m-1)\tau]_+^{1/(m-1)}$. 

\begin{lemma}\label{lm:conc}
\begin{enumerate}[{\rm (i)}]
\item
Given $0<m \leq m' <1$ with $ m+m'>1$, set 
\[
\beta=\beta(m,m') =1+\frac{1-m'}{1-m} \in (1,2]. 
\]
Then we have $\beta m' > 1$ and, for any $a,r>0$,
\begin{gather*} e_m\bigg( -\left(ar-\frac1{\sqrt{m'}}\right)^2+\frac1{m'} \bigg)
 \le e_m(\beta) e_m\left( -\left( 1-\frac{1}{\beta m'} \right) \frac{a^2r^2}{m+m'-1} \right).
\end{gather*}

\item
For any $m \in [1,2)$ and $a,r>0$, we have
\[ e_m\big( (ar-1)^2-1 \big) \ge e_m \bigg(\! -\frac{2}{m} \bigg) e_m\bigg( \frac{a^2}{2}r^2 \bigg). \]
\end{enumerate}
\end{lemma}

\begin{proof}
(i) 
The assumptions  $m' <1$ and  $ m+m'>1$ yields  $(m+m')(1-m')>(1-m')$, and hence
\[
\beta m'= \frac{m' \{2-(m+m')\} }{1-m}
=\frac{ (m'+m)(1-m')+m'-m  }{1-m} 
>\frac{ 1-m' +m'-m  }{1-m}  = 1.
\]
From the direct calculation  
\[ -\left( ar-\frac{1}{\sqrt{m'}} \right)^2 +\frac{1}{m'}
 =-a^2r^2 +\frac{2ar}{\sqrt{m'}}
 \le -a^2 r^2 +\frac{a^2 r^2}{\beta m'} +\beta  \]
and the monotonicity of $e_m$, we deduce that 
\begin{align*}
&e_m\bigg( -\bigg(ar-\frac1{\sqrt{m'}}\bigg)^2+\frac1{m'} \bigg)
 \le e_m\left( -\left(1-\frac{1}{\beta m'}\right)a^2r^2 +\beta   \right) \\
&= \left[ 1+(m-1) \left\{  -\left(1-\frac{1}{\beta m'}\right)a^2r^2 +\beta  \right\} \right]^{1/(m-1)} \\
&= \{1+(m-1) \beta\}^{1/(m-1)} \left\{ 1-(m-1) \left( 1-\frac{1}{\beta m'} \right) \frac{a^2r^2}{m+m'-1}  \right\}^{1/(m-1)}.
\end{align*}

(ii) 
The assertion for $m=1$ (with $e_1(\tau)=e^{\tau}$) is easily checked.
For $m\in (1,2)$, we deduce from
\[ (ar-1)^2-1 =a^2r^2 -2ar
 \ge a^2r^2 -\frac{m}{2} \bigg\{ a^2r^2+\left( \frac{2}{m} \right)^2 \bigg\}
 =\left( 1-\frac{m}{2} \right) a^2r^2 -\frac{2}{m} \]
that
\begin{align*}
e_m\big( (ar-1)^2-1 \big)
&\ge e_m\bigg( \bigg(1-\frac{m}{2} \bigg) a^2r^2 -\frac{2}{m} \bigg) \\
&= \bigg[ 1+(m-1) \left\{ \bigg(1-\frac{m}{2} \bigg) a^2r^2 -\frac{2}{m} \right\} \bigg]^{1/(m-1)} \\
&= \bigg\{ 1-(m-1)\frac{2}{m} \bigg\}^{1/(m-1)} \bigg\{ 1+\frac{m-1}{2} ma^2r^2 \bigg\}^{1/(m-1)} \\
&= e_m \bigg(\! -\frac{2}{m} \bigg) e_m\bigg( \frac{ma^2}{2}r^2 \bigg)
 \ge e_m \bigg(\! -\frac{2}{m} \bigg) e_m\bigg( \frac{a^2}{2}r^2 \bigg).
\end{align*}
$\qedd$
\end{proof}

\begin{theorem}[$m(\varphi)$-normal concentration]\label{thm:conc}
Assume that $(M,\omega,\varphi, \Psi)$ is admissible, $\nu \in \cP_{\ac}(M,\omega)$,
$\Hess H_{\varphi} \ge K$ for some $K>0$ and that $\|\sigma\|_{\infty}<\infty$.
Fix arbitrary $\xi_0 \ge \max\{1, \|\sigma\|_\infty\}$.

\begin{enumerate}[{\rm (i)}]
\item
If $\theta_\varphi<1$ and $\delta_\varphi >0$, then we have for any $r>0$
\[ \alpha(r)^{-1} \ge
 \left\{ \frac{\delta_\varphi(1-\theta_\varphi)}{(1-\delta_\varphi)(2-\delta_\varphi)} \right\}^{1/(1-\delta_\varphi)}
 \cdot e_{2-\delta_\varphi} \left( \frac{K}{4} \xi_0^{\delta_\varphi-1}r^2 \right). \]

\item 
If $\theta_\varphi \in (1,3/2)$, $\delta_{\varphi}>3(\theta_{\varphi}-1)$ and if $\omega[M]<\infty$,
then we have for any $r>0$
\begin{align*}
\alpha (r) &\le
 \bigg( \frac{(\theta_\varphi-1)(3-3\theta_{\varphi}+\delta_{\varphi})}
 {2\theta_{\varphi}-\delta_{\varphi}-1} \bigg) \bigg\}^{1/(1-2\theta_\varphi+\delta_\varphi)} \\
 &\quad \times
  e_{2(1-\theta_\varphi)+\delta_\varphi}
 \left( -\frac{K}{2} \frac{\theta_\varphi-1}{(2-\theta_\varphi)(3\theta_\varphi-\delta_\varphi-2)}
 \xi_0^{ \delta_\varphi-\theta_\varphi} \omega[M]^{1-\theta_\varphi} r^2 \right).
\end{align*}

\item
If $\theta_\varphi=1$ and $\delta_\varphi >1/2$, then we have for any $r>0$
\[ \alpha(r)^{-1} \ge
 e_{3-2\delta_\varphi} \left( \frac{-2}{3-2\delta_\varphi}\right) \cdot
 e_{3-2\delta_\varphi} \left( \frac{K}{4} \xi_0^{\delta_\varphi-1} r^2 \right). \]
\end{enumerate}
\end{theorem}

\begin{proof}
We abbreviate $\alpha(r)$ as $\alpha$ in this proof, and assume $\alpha>0$ without loss of generality.
Let $A \subset M$ be a measurable set of $\nu[A] \ge 1/2$ and put $B:=M \setminus B(A,r)$.

(i) 
We first observe
\[ \int_B \sigma^{2-\theta_\varphi} \,d\omega
 \le \|\sigma\|_{\infty}^{1-\theta_\varphi} \int_B \sigma \,d\omega
 \le \alpha \xi_0^{1-\theta_\varphi}. \]
Then \eqref{eq:concc} yields
\begin{equation}\label{eq:conce}
\alpha^{\delta_\varphi-1} \ln_\varphi(\alpha) \le
 \xi_0^{\delta_\varphi-1} \bigg\{ -\bigg( \sqrt{\frac{K}{2}}r -\sqrt{-H_\nu} \bigg)^2-H_\nu \bigg\},
\end{equation}
where $H_{\nu}:=H_{\varphi}(\nu)$.
On the one hand, it follows from \eqref{eq:ln} that
\[ \alpha^{\delta_\varphi-1} \ln_\varphi(\alpha)
 \ge \alpha^{\delta_\varphi-1} \ell_{2-\theta_{\varphi}}(\alpha)
 =\frac{\alpha^{\delta_\varphi-\theta_\varphi}-\alpha^{\delta_\varphi-1}}{1-\theta_\varphi}. \]
Since $\alpha^{\delta_{\varphi}-\theta_{\varphi}} \ge 1 \ge (1-\theta_{\varphi})/(1-\delta_{\varphi})$,
we obtain
\[ \alpha^{\delta_\varphi-1} \ln_\varphi(\alpha)
 \ge \frac{1- (C\alpha^{-1})^{1-\delta_\varphi}}{1-\delta_\varphi}
 = -\ell_{2-\delta_\varphi}(C\alpha^{-1}), \quad
 C:=\bigg( \frac{1-\delta_{\varphi}}{1-\theta_{\varphi}} \bigg)^{1/(1-\delta_{\varphi})} \ge 1. \]
On the other hand, Lemmas~\ref{lm:fitt}(i) and \ref{lem:mono} give
\begin{equation}\label{eq:xi_0}
-\xi_0^{\delta_\varphi-1} H_\nu
 \le \frac{\xi_0^{\delta_\varphi}}{(2-\theta_\varphi) \varphi(\xi_0)}
 \le \frac1{2-\theta_\varphi} \le 1.
\end{equation}
Hence we have
\[ \ell_{2-\delta_\varphi}(C\alpha^{-1})
 \ge \bigg( \sqrt{\frac{K \xi_0^{\delta_\varphi-1}}2} r-1 \bigg)^2- 1. \]
We apply Lemma~\ref{lm:conc}(ii) and obtain
\begin{align*}
\alpha^{-1}
&\ge C^{-1}e_{2-\delta_\varphi} \bigg( \bigg(\sqrt{\frac{K}{2}} \xi_0^{(\delta_\varphi-1)/2}r-1 \bigg)^2 -1 \bigg) \\
&\ge C^{-1}e_{2-\delta_\varphi} \bigg( \frac{2}{\delta_\varphi-2} \bigg)
 e_{2-\delta_\varphi} \left( \frac{K}{4} \xi_0^{\delta_\varphi-1} r^2 \right) \\
&= \left\{ \frac{\delta_\varphi(1-\theta_\varphi)}{(1-\delta_\varphi)(2-\delta_\varphi)} \right\}^{1/(1-\delta_\varphi)}
 \cdot e_{2-\delta_\varphi} \left( \frac{K}{4} \xi_0^{\delta_\varphi-1} r^2 \right).
\end{align*}

(ii) 
We deduce from the H\"older inequality that
\[ \int_B \sigma^{2-\theta_\varphi} \,d\omega
 \le \bigg( \int_B \sigma \,d\omega \bigg)^{2-\theta_\varphi} \omega[B]^{\theta_\varphi-1}
 \le  \alpha^{2-\theta_\varphi}  \omega[M]^{\theta_\varphi-1}. \]
Then \eqref{eq:concc} gives
\[ \alpha^{\delta_\varphi-\theta_\varphi} \ln_\varphi(\alpha)
 \le \xi_0 ^{\delta_\varphi-\theta_\varphi}  \omega[M]^{1-\theta_\varphi}
 \bigg\{\! -\bigg( \sqrt{\frac{K}2}r -\sqrt{-H_\nu} \bigg)^2 -H_\nu \bigg\}. \]
Set $m:=2(1-\theta_\varphi)+\delta_\varphi$ and $m':=2-\theta_\varphi$,
and observe $0<m \le m'<1$ as well as $m+m'>1$.
Similarly to (i), \eqref{eq:ln} yields
\[ \alpha^{\delta_\varphi-\theta_\varphi} \ln_\varphi(\alpha)
 \ge \alpha^{\delta_{\varphi}-\theta_{\varphi}} \ell_{2-\theta_{\varphi}}(\alpha)
 =\frac{\alpha^{\delta_\varphi-2\theta_\varphi+1}-\alpha^{\delta_\varphi-\theta_\varphi}}{1-\theta_\varphi}
 =\frac{\alpha^{m-1}-\alpha^{m-m'}}{m'-1}. \]
As $\alpha^{m-m'} \ge 1 \ge (1-m')/(1-m)$, we find
\[ \alpha^{\delta_\varphi-\theta_\varphi} \ln_\varphi(\alpha)
 \ge \frac{(c\alpha)^{m-1}-1}{m-1}=\ell_m(c\alpha), \quad
 c:=\bigg( \frac{1-m}{1-m'} \bigg)^{1/(m-1)} \le 1. \]
Lemmas~\ref{lm:fitt}(ii) and \ref{lem:mono} imply
\[ -\xi_0^{\delta_\varphi-\theta_\varphi} \omega[M]^{1-\theta_\varphi} H_\nu
 \le \frac{\xi_0^{\delta_\varphi}}{(2-\theta_\varphi) \varphi(\xi_0)}
 \le \frac1{2-\theta_\varphi} =\frac{1}{m'}, \]
and hence
\[ \ell_{m}(c\alpha) \le
 -\bigg( \sqrt{ \frac{ K\xi_0^{\delta_\varphi-\theta_\varphi} \omega[M]^{1-\theta_\varphi}}{2}}r
 -\frac{1}{\sqrt{m'}} \bigg)^2+\frac{1}{m'}. \]
Then we apply Lemma~\ref{lm:conc}(i) to have, with $\beta=(2-m-m')/(1-m)$,
\[ \alpha \le
 c^{-1} e_m(\beta)  e_m \left( -\left(1-\frac{1}{\beta m'} \right)
 \frac{K\xi_0^{\delta_\varphi-\theta_\varphi} \omega[M]^{1-\theta_\varphi}}{2(m+m'-1)} r^2 \right). \]

(iii)
It immediately follows from \eqref{eq:conce} and \eqref{eq:xi_0} that 
\[ \alpha^{\delta_\varphi-1} \ln_\varphi(\alpha)
 \le \xi_0 ^{\delta_\varphi-1}
 \bigg\{ -\bigg( \sqrt{\frac{K}{2}}r -\sqrt{-H_\nu} \bigg)^2-H_\nu \bigg\}
 \le -\bigg( \sqrt{\frac{K \xi_0 ^{\delta_\varphi-1}}{2}}r -1 \bigg)^2 +1. \]
Note that \eqref{eq:ln} provides
$\alpha^{\delta_\varphi-1} \ln_\varphi (\alpha) \ge \alpha^{\delta_\varphi-1} \ln (\alpha)$.
If $\delta_{\varphi}=1$, then it holds $\alpha^{\delta_{\varphi}-1}\ln(\alpha)=-\ln(\alpha^{-1})$.
Otherwise, the numerical estimate
\[ \ln(t) \ge \frac{t^s-t^{-s}}{2s} \qquad \text{for}\ t \in (0,1],\ s>0, \]
shows $\alpha^{\delta_\varphi-1} \ln (\alpha) \ge -\ell_{3-2\delta_\varphi}(\alpha^{-1})$
(let $s=1-\delta_{\varphi}$ and $t=\alpha$).
Therefore we have, thanks to Lemma~\ref{lm:conc}(ii) with $m=3-2\delta_{\varphi}<2$,
\[
\alpha^{-1} \ge
e_{3-2\delta_\varphi} \bigg( \bigg( \sqrt{\frac{K\xi_0^{\delta_\varphi-1}}{2}}r -1 \bigg)^2-1 \bigg)
 \ge e_{3-2\delta_\varphi} \left( \frac{-2}{3-2\delta_\varphi}\right)
 \cdot e_{3-2\delta_\varphi} \left( \frac{K}4\xi_0 ^{\delta_\varphi-1} r^2 \right).
\]
$\qedd$
\end{proof}

Note that $3(\theta_{\varphi}-1) <\theta_{\varphi}$ in (ii) by $\theta_{\varphi}<3/2$,
so that the condition $\delta_{\varphi}>3(\theta_{\varphi}-1)$ is not vacuous.

\begin{remark}
Letting $\delta_\varphi=\theta_\varphi$ and then $\theta_{\varphi} \to 1$,
all of the estimates (i)--(iii) in Theorem~\ref{thm:conc} tend to the normal concentration $\alpha(r) \leq e^2 \exp(-Kr^2/4)$.
\end{remark}

%%%%%%%%%%%%%%%%%%%%%%%%%%%%%%%%%%%%%%%%%%
\section{Gradient flow of $H_\varphi$: Compact case}\label{sc:gf}%%%%%%%%%%%%%

In this and the next sections, we show that the gradient flow of the $\varphi$-relative entropy
produces weak solutions to the nonlinear evolution equation
\[ \frac{\partial\rho}{\partial t}
 =\div_{\omega} \bigg( \frac{\rho\nabla\rho}{\varphi(\rho)} +\rho\nabla\Psi \bigg) \]
on the weighted Riemannian manifold $(M, \omega)$.
See the beginning of Subsection~\ref{ssc:heat} for more explanation and background.
This kind of interpretation of evolution equations has turned out extremely
useful after the pioneering work due to Jordan et al.~\cite{JKO}.
There are several ways of interpreting this coincidence.
In this section, we adapt the rather `metric geometric' approach developed in \cite{Ogra}
inspired by \cite{PP} and \cite{Ly} (see also \cite{Pesig}).
This formulation of gradient flows requires a strong structure theorem (Theorem~\ref{th:angle})
of the Wasserstein space, which is known only for compact spaces.
The noncompact situation will be treated in the next section in a different strategy
along \cite{AGS} and \cite{Er}.

Before beginning the review of the structure of Wasserstein spaces,
let us recall basic notions of calculus on our weighted Riemannian manifold
$(M,\omega)$ with $\omega=e^{-f}\vol_g$.
For a differentiable vector field $V$ on $M$, the \emph{weighted divergence} is defined as
\[ \div_{\omega}V:=\div V-\langle V,\nabla f \rangle, \]
where $\div V$ denotes the usual divergence of $V$ for the unweighted space $(M,\vol_g)$.
Note that $\div_{\omega}V=e^f \div(e^{-f}V)$ and, for any $w \in C^1_c(M)$,
the integration by parts holds:
\begin{align*}
\int_M \langle \nabla w,V \rangle \,d\omega
 =\int_M \langle \nabla w,e^{-f}V \rangle \,d\!\vol_g
 =-\int_M w\div(e^{-f}V) \,d\!\vol_g
 = -\int_M w\div_{\omega}V \,d\omega.
\end{align*}
Through this formula, the weighted divergence is defined in the weak sense
also for measurable vector fields.
For $\rho \in H^1_{\loc}(M)$, the {\it weighted Laplacian} is defined in the weak form by
\[ \Delta^{\omega}\rho:=\div_{\omega}(\nabla \rho)
 =\Delta \rho -\langle \nabla \rho, \nabla f \rangle. \]

\subsection{Geometric structure of  $(\cP(M),W_2)$}\label{ssc:gf1}%%%%%%%
%%%%%%%%%%%%%%%%%%%%%%%%%%%%%%%%%

Let $M$ be compact throughout the section, so that $\cP(M)=\cP^2(M)$.
It is known that $(\cP(M),W_2)$ is an Alexandrov space of nonnegative
curvature if and only if $(M,g)$ has the nonnegative sectional curvature
(\cite[Proposition~2.10]{StI}, \cite[Theorem~A.8]{LV2}).
Alexandrov spaces are metric spaces whose sectional curvature is bounded
from below by a constant in the sense of the triangle comparison property,
and such spaces are known to possess nice infinitesimal structures
(we refer to \cite{BBI} for the basic theory).
We remark that it is in most cases impossible to bound the curvature of $\cP(M)$ from above
(cf.\ \cite[Example~7.3.3]{AGS}).
In the case where $(M,g)$ is not nonnegatively curved, although $(\cP(M),W_2)$ does not
admit any lower curvature bound in the sense of Alexandrov (\cite[Proposition~2.10]{StI}),
we can consider the `angle' between geodesics (see also \cite[Theorem~3.6]{Ogra}).

\begin{theorem}\label{th:angle}{\rm (\cite[Theorem~3.4, Remark~3.5]{Gi})}
For any $\mu \in \cP(M)$ and unit speed geodesics
$\alpha,\beta:[0,\delta) \lra \cP(M)$ with $\alpha(0)=\beta(0)=\mu$, the joint limit
\[ \lim_{s,t \downarrow 0}\frac{s^2 +t^2-W_2(\alpha(s),\beta(t))^2}{2st}\ \in [-1,1] \]
exists.
\end{theorem}

Theorem~\ref{th:angle} in particular guarantees that the scaling limit
\[ \lim_{\ve \downarrow 0}\frac{(s\ve)^2+(t\ve)^2-W_2(\alpha(s\ve),\beta(t\ve))^2}{2st\ve^2} \]
exists, and is independent of the choices of the parameters $s,t>0$.
This means that an angle between $\alpha$ and $\beta$ makes sense,
so that $(\cP(M),W_2)$ looks like a Riemannian space (rather than a Finsler space).
This observation makes it possible to investigate the infinitesimal structure
of $(\cP(M),W_2)$ in the manner of the theory of Alexandrov spaces.
For $\mu \in \cP(M)$, denote by $\Sigma'_{\mu}[\cP(M)]$ the set of all
nontrivial unit speed minimal geodesics emanating from $\mu$.
Given $\alpha,\beta \in \Sigma'_{\mu}[\cP(M)]$, Theorem~\ref{th:angle}
verifies that the {\it angle}
\[ \angle_{\mu}(\alpha,\beta)
 :=\arccos \bigg( \lim_{s,t \downarrow 0}\frac{s^2 +t^2-W_2(\alpha(s),\beta(t))^2}{2st} \bigg)
 \in [0,\pi] \]
is well-defined.
We define the {\it space of directions} $(\Sigma_{\mu}[\cP(M)],\angle_{\mu})$
as the completion of $(\Sigma'_{\mu}[\cP(M)]/\!\!\sim, \angle_{\mu})$,
where $\alpha \sim \beta$ holds if $\angle_{\mu}(\alpha,\beta)=0$.
The angle $\angle_{\mu}$ provides a natural distance structure of $\Sigma_{\mu}[\cP(M)]$.
The {\it tangent cone} $(C_{\mu}[\cP(M)],\sigma_{\mu})$ is defined as
the Euclidean cone over $(\Sigma_{\mu}[\cP(M)],\angle_{\mu})$, i.e.,
\begin{align*}
C_{\mu}[\cP(M)] &:=\big( \Sigma_{\mu}[\cP(M)] \times [0,\infty) \big)
 \big/ \big( \Sigma_{\mu}[\cP(M)] \times \{0\} \big), \\
\sigma_{\mu}\big( (\alpha,s),(\beta,t) \big)
&:=\sqrt{s^2+t^2-2st\cos\angle_{\mu}(\alpha,\beta)}.
\end{align*}
By means of this infinitesimal structure, we introduce a class of `differentiable curves'.

\begin{definition}[Right differentiability]\label{df:rd}
We say that a curve $\xi:[0,l) \lra \cP(M)$ is {\it right differentiable} at $t \in [0,l)$
if there is $\bv \in C_{\xi(t)}[\cP(M)]$ such that, for any sequences
$\{\ve_i\}_{i \in \N}$ of positive numbers tending to zero and $\{ \alpha_i \}_{i \in \N}$
of unit speed minimal geodesics from $\xi(t)$ to $\xi(t+\ve_i)$, the sequence
$\{ (\alpha_i,W_2(\xi(t),\xi(t+\ve_i))/\ve_i) \}_{i \in \N} \subset C_{\xi(t)}[\cP(M)]$
converges to $\bv$.
Such $\bv$ is clearly unique if it exists, and then we write $\dot{\xi}(t)=\bv$.
\end{definition}

%%%%%%%%%%%%%%%%%%%%%%%%%%%%%%%%%%%%
\subsection{Gradient flows in $(\cP(M),W_2)$}\label{ssc:gf2}%%%%%%%%%%%%

Consider a lower semi-continuous function $H:\cP(M) \lra (-\infty,+\infty]$
which is $K$-convex in the weak sense for some $K \in \R$.
We in addition suppose that $H$ is not identically $+\infty$, and define
$\cP_H^*(M):=\{ \mu \in \cP(M) \,|\, H(\mu)<\infty \}$.

Given $\mu \in \cP_H^*(M)$ and $\alpha \in \Sigma_{\mu}[\cP(M)]$, we set
\[ D_{\mu}H(\alpha):=\liminf_{\Sigma'_{\mu}[\cP(M)] \ni \beta \to \alpha}
 \lim_{t \downarrow 0}\frac{H(\beta(t))-H(\mu)}{t}, \]
where the convergence $\beta \to \alpha$ is with respect to $\angle_{\mu}$.
Define the {\it absolute gradient} (also called the {\it local slope})
of $H$ at $\mu \in \cP_H^*(M)$ by
\[ |\grad H|(\mu):=\max\bigg\{ 0, \limsup_{\tilde{\mu} \to \mu}
 \frac{H(\mu)-H(\tilde{\mu})}{W_2(\mu,\tilde{\mu})} \bigg\}, \]
where $\tilde{\mu} \to \mu$ is with respect to $W_2$.
Note that $-D_{\mu}H(\alpha) \le |\grad H|(\mu)$ for any $\alpha \in \Sigma_{\mu}[\cP(M)]$.
The $K$-convexity of $H$ guarantees the unique existence of the direction
along which $H$ decreases the most.

\begin{lemma}\label{lm:gv}{\rm (\cite[Lemma~4.2]{Ogra})}
For each $\mu \in \cP_H^*(M)$ with $0<|\grad H|(\mu)<\infty$,
there exists a unique direction $\alpha \in \Sigma_{\mu}[\cP(M)]$ satisfying
$D_{\mu}H(\alpha)=-|\grad H|(\mu)$.
\end{lemma}

Using $\alpha$ in the above lemma, we define the {\it negative gradient vector}
of $H$ at $\mu$ by
\[ \grad H(\mu):=\big( \alpha,|\grad H|(\mu) \big) \in C_{\mu}[\cP(M)]. \]
If $|\grad H|(\mu)=0$, then we simply define $\grad H(\mu)$ as the origin
of $C_{\mu}[\cP(M)]$.
A trajectory of the gradient flow of $H$ (which will be called a gradient curve)
should be understood as a curve $\xi$ solving $\dot{\xi}(t)=\grad H(\xi(t))$.
Precisely, we adopt the following definition.

\begin{definition}[Gradient curves]\label{df:gc}
A continuous curve $\xi:[0,l) \lra \cP_H^*(M)$ which is locally Lipschitz on $(0,l)$
is called a {\it gradient curve} of $H$ if $|\grad H|(\xi(t))<\infty$ for all $t \in (0,l)$
and if it is right differentiable with $\dot{\xi}(t)=\grad H(\xi(t))$ at all $t \in (0,l)$.
We say that a gradient curve $\xi$ is {\it complete} if it is defined on entire $[0,\infty)$.
\end{definition}

By virtue of the $K$-convexity of $H$ as well as the compactness of $M$,
there starts a unique gradient curve from an arbitrary initial point $\mu \in \cP_H^*(M)$
enjoying the \emph{$K$-contraction property} as follows.

\begin{theorem}\label{th:cont}
{\rm (\cite[Theorem~5.11, Corollary~6.3]{Ogra}, \cite[Theorem~4.2]{GO})}
Let $M$ be compact and $H:\cP(M) \lra (-\infty,+\infty]$ be a $K$-convex function for some $K \in \R$.
\begin{enumerate}[{\rm (i)}]
\item From any $\mu \in \cP_H^*(M)$, there exists a unique complete gradient curve
$\xi:[0,\infty) \lra \cP_H^*(M)$ of $H$ with $\xi(0)=\mu$.

\item {\rm ($K$-contraction property)}
Given any two gradient curves $\xi,\zeta:[0,\infty) \lra \cP_H^*(M)$ of $H$, we have
\begin{equation}\label{eq:cont}
W_2 \big(\xi(t),\zeta(t) \big) \le e^{-Kt}W_2\big( \xi(0),\zeta(0) \big)
\end{equation}
for all $t \in [0,\infty)$.
\end{enumerate}
\end{theorem}

The uniqueness in (i) is indeed a consequence of
the $K$-contraction property \eqref{eq:cont}.
Thus the {\it gradient flow} $G:[0,\infty) \times \cP_H^*(M) \lra \cP_H^*(M)$
of $H$, given as $G(t,\mu)=\xi(t)$ for $\xi$ in Theorem~\ref{th:cont}(i),
is uniquely determined and continuously extended to the closure
$G:[0,\infty) \times \overline{\cP_H^*(M)} \lra \overline{\cP_H^*(M)}$.

%%%%%%%%%%%%%%%%%%%%%%%%%%%%%%%%%%%%%
\subsection{$H_{\varphi}$ and the $\varphi$-heat equation}\label{ssc:heat}%%%%%%%%%%

It is an established fact that the gradient flow of the relative entropy
(or the {\it free energy}) with respect to $\omega$,
\[ \Ent_{\omega}(\rho\omega) =\int_M \rho \ln\rho \,d\omega
 =\int_M (\rho e^{-f}) \ln(\rho e^{-f}) \,d\!\vol_g +\int_M f \,d\mu, \]
produces solutions to the associated {\it heat equation}
(or the {\it Fokker--Planck equation})
\[ \frac{\del\rho}{\del t} =\Delta^{\omega}\rho
 =e^f\big\{ \Delta(\rho e^{-f})
 +\div\big( (\rho e^{-f})\nabla f \big) \big\}. \]
See \cite[Theorem~5.1]{JKO}, \cite[Subsection~8.4.2]{Vi1} for the Euclidean case,
\cite[Theorem~6.6]{Ogra}, \cite[Theorem~4.6]{GO}, \cite[Corollary~23.23]{Vi2}
for the Riemannian case, \cite[Section~7]{OS} for the Finsler case,
and \cite{FSS}, \cite{Ju}, \cite{GKO}, \cite{Ma}, \cite{AGS2} for further related work on various kinds of spaces.

We shall see that a similar argumentation gives weak solutions to the equation
\begin{equation}\label{eq:pme}
\frac{\del \rho}{\del t}=\div_{\omega}\bigg( \frac{\rho \nabla\rho}{\varphi(\rho)}+\rho\nabla\Psi \bigg)
\end{equation}
as the gradient flow of the $\varphi$-relative entropy $H_{\varphi}$.
We will call \eqref{eq:pme} the \emph{$\varphi$-heat equation}.
In the special case of $\varphi_m(s)=s^{2-m}$, \eqref{eq:pme} is called the \emph{fast diffusion equation}
(for $m<1$) or the \emph{porous medium equation} (for $m>1$).
Then this identification was demonstrated by Otto~\cite{Ot} on $(\R^n,\cL^n)$,
and by \cite[Theorem~23.19]{Vi2} as well as \cite{mCD} on weighted Riemannian manifolds
(by the different means).
We can follow the strategy of \cite{mCD} for general $\varphi$,
up to some technical difficulties.

We first observe $|\grad H_{\varphi}|(\mu)=\sqrt{I_{\varphi}(\mu)}$
as Proposition~\ref{pr:dHm} suggests.

\begin{proposition}\label{pr:DH}
Let $(M,\omega,\varphi,\Psi)$ be a compact admissible space such that
$\Ric_{N_{\varphi}} \ge 0$ and $\Hess\Psi \ge K$ for some $K \in \R$.
Take $\mu=\rho\omega \in \cP_{\ac}(M,\omega)$ with
$\mu[M_{\varphi}^{\Psi}]=1$, $H_{\varphi}(\mu)<\infty$,
$\rho h'_{\varphi}(\rho)-h_{\varphi}(\rho) \in H^1(M)$
and with $|\nabla\rho/\varphi(\rho)| \in L^2(M,\mu)$.
Then we have $|\grad H_{\varphi}|(\mu)=\sqrt{I_{\varphi}(\mu)}$,
and the negative gradient vector $\grad H_{\varphi}(\mu)$ is given by
$-\nabla\rho/\varphi(\rho)-\nabla\Psi$.
\end{proposition}

\begin{proof}
Given any $\mu_1 \in \cP(M)$ with $H_{\varphi}(\mu_1)<\infty$,
let $(\mu_t)_{t \in [0,1]} \subset \cP(M)$
be a minimal geodesic from $\mu_0=\mu$ to $\mu_1$ along which $H_{\varphi}$
is $K$-convex (Theorem~\ref{th:mCD}).
Letting $\mu_t=(\cT_t)_{\sharp}\mu$ with $\cT_t(x)=\exp_x(t\nabla\phi(x))$,
we deduce from the $K$-convexity of $H_{\varphi}$ that
\[ \lim_{t \downarrow 0}\frac{H_{\varphi}(\mu_t)-H_{\varphi}(\mu)}{t}
 \le H_{\varphi}(\mu_1) -H_{\varphi}(\mu) -\frac{K}{2}W_2(\mu,\mu_1)^2. \]
Combining this with Proposition~\ref{pr:dHm}, we have
\begin{align*}
\frac{H_{\varphi}(\mu)-H_{\varphi}(\mu_1)}{W_2(\mu,\mu_1)}
&\le -\frac{1}{W_2(\mu,\mu_1)}
 \int_M \bigg\langle \frac{\nabla\rho}{\varphi(\rho)}+\nabla\Psi,\nabla\phi \bigg\rangle \,d\mu
 -\frac{K}{2}W_2(\mu,\mu_1) \\
&\le \sqrt{I_{\varphi}(\mu)} -\frac{K}{2}W_2(\mu,\mu_1).
\end{align*}
Thus we obtain $|\grad H_{\varphi}|(\mu) \le \sqrt{I_{\varphi}(\mu)}$,
and equality follows also from Proposition~\ref{pr:dHm} by choosing $\{\phi_i\}_{i \in \N} \subset C^{\infty}(M)$
which approximates $-\ln_{\varphi}(\rho)+\ln_{\varphi}(\sigma)$ in $H^1(M,\mu)$.
Then, moreover, $\grad H_{\varphi}(\mu)$ is achieved by $-\nabla\rho/\varphi(\rho)-\nabla\Psi$
(to be precise, $((\mu_t)_{t \in [0,1]},W_2(\mu,\mu_1))$ associated with $\phi_i$ converges to
$\grad H_{\varphi}(\mu)$ in $C_{\mu}[\cP(M)]$).
$\qedd$
\end{proof}

Now we are ready to show the main theorem of the section.

\begin{theorem}[Gradient flow of $H_{\varphi}$]\label{th:gf}
Let $(M,\omega,\varphi,\Psi)$ be a compact admissible space such that
$\Ric_{N_{\varphi}} \ge 0$ and $\Hess\Psi \ge K$ on $M^{\Psi}_{\varphi}$ for some $K \in \R$.
We in addition assume that $\theta_{\varphi} \in (0,(n+1)/n)$,
$\lim_{s \to \infty}s^{\theta_{\varphi}}/\varphi(s)<\infty$ and that $\Psi$ is Lipschitz.
If a curve $(\mu_t)_{t \in [0,\infty)} \subset \cP_{\ac}(M,\omega)$
with $\mu_t[M^{\Psi}_{\varphi}] \equiv 1$
is a gradient curve of $H_{\varphi}$, then its density function $\rho_t$ is a weak solution
to the $\varphi$-heat equation \eqref{eq:pme}.
To be precise, $\rho_t$ is weakly differentiable as well as
$|\nabla\rho_t/\varphi(\rho_t)| \in L^2(M,\mu_t)$ a.e.\ $t$, and we have
\begin{equation}\label{eq:wpme}
\int_M w_{t_1} \,d\mu_{t_1} -\int_M w_{t_0} \,d\mu_{t_0}
 = \int_{t_0}^{t_1} \int_M \bigg\{ \frac{\del w_t}{\del t}
 -\bigg\langle \frac{\nabla\rho_t}{\varphi(\rho_t)}+\nabla\Psi,\nabla w_t \bigg\rangle \bigg\} \,d\mu_t \,dt
\end{equation}
for all $0\le t_0<t_1<\infty$ and $w \in C^{\infty}(\R \times M)$,
where $\mu_t=\rho_t \omega$ and $w_t=w(t,\cdot)$.
\end{theorem}

\begin{proof}
First of all, the weak differentiability of $\rho_t$ and $|\nabla\rho_t/\varphi(\rho_t)| \in L^2(M,\mu_t)$
follow from (I) $\Rightarrow$ (II) of Proposition~\ref{pr:gf+} below.
Fix $t \in (0,\infty)$ and, given small $\delta>0$, choose
$\mu^{\delta} \in \cP(M)$ as a minimizer of the function
\begin{equation}\label{eq:apprx}
\mu\ \longmapsto\ H_{\varphi}(\mu)+\frac{W_2(\mu,\mu_t)^2}{2\delta}.
\end{equation}
We postpone the proof of the following technical claim until the end of the section.
The condition $\theta_{\varphi}<(n+1)/n$ will come into play in (i),
while $\theta_{\varphi}>0$ and $\lim_{s \to \infty}s^{\theta_{\varphi}}/\varphi(s)<\infty$ will be used in (iii).

\begin{claim}\label{cl:gf}
\begin{enumerate}[{\rm (i)}]
\item Such a minimizer $\mu^{\delta}$ of \eqref{eq:apprx} indeed exists
and is absolutely continuous with respect to $\omega$.

\item We have
\[ \lim_{\delta \downarrow 0} \frac{W_2(\mu^{\delta},\mu_t)^2}{\delta}=0, \qquad
 \lim_{\delta \downarrow 0} H_\varphi(\mu^{\delta})=H_\varphi(\mu_t). \]
In particular, $\mu^{\delta}$ converges to $\mu_t$ weakly.

\item Moreover, by putting $\mu^{\delta}=\rho^{\delta}\omega$,
$h_{\varphi}(\rho^{\delta})-h'_{\varphi}(\rho^{\delta})\rho^{\delta}$ converges to
$h_{\varphi}(\rho_t)-h'_{\varphi}(\rho_t)\rho_t$ in $L^1(M,\omega)$
as $\delta \downarrow 0$.
\end{enumerate}
\end{claim}

Take a semi-convex function $\phi:M \lra \R$ such that
$\cT(x):=\exp_x(\nabla\phi(x))$ gives the optimal transport
from $\mu^{\delta}$ to $\mu_t$ (recall Theorem~\ref{th:FG}).
We also consider the transport $\mu^{\delta}_{\ve}:=(\cF_{\ve})_{\sharp}\mu^{\delta}$
in another direction for small $\ve>0$, where $\cF_{\ve}(x):=\exp_x(\ve\nabla w_t(x))$.
It immediately follows from the choice of $\mu^{\delta}$ that
\begin{equation}\label{eq:gf1}
H_{\varphi}(\mu^{\delta}_{\ve})
 +\frac{W_2(\mu^{\delta}_{\ve},\mu_t)^2}{2\delta}
 \ge H_{\varphi}(\mu^{\delta}) +\frac{W_2(\mu^{\delta},\mu_t)^2}{2\delta}.
\end{equation}
We first estimate the difference of the Wasserstein distances.
Observe that, as $(\cF_{\ve} \times \cT)_{\sharp}\mu^{\delta}$ is a
(not necessarily optimal) coupling of $\mu^{\delta}_{\ve}$ and $\mu_t$,
\begin{align*}
&\limsup_{\ve \downarrow 0}
 \frac{W_2(\mu^{\delta}_{\ve},\mu_t)^2-W_2(\mu^{\delta},\mu_t)^2}{\ve} \\
&\le \limsup_{\ve \downarrow 0}\frac{1}{\ve}\int_M \big\{
 d_g \big( \cF_{\ve}(x),\cT(x) \big)^2-d_g \big( x,\cT(x) \big)^2 \big\} \,d\mu^{\delta}(x)
 = -\int_M 2\langle \nabla w_t,\nabla\phi \rangle \,d\mu^{\delta}.
\end{align*}
We used the first variation formula for the Riemannian distance function $d_g$ in the last line
(cf., e.g., \cite[Theorem~II.4.1]{Ch}).
Thanks to the compactness of $M$, there is a constant $C>0$
(depending only on $(M,g)$ and $w$) such that
\[ w_t\big( \cT(x) \big) \le w_t(x)+\langle \nabla w_t(x),\nabla\phi(x) \rangle
 +Cd_g \big( x,\cT(x) \big)^2 \]
for a.e.\ $x \in M$.
Thus we obtain, by virtue of Claim~\ref{cl:gf}(ii),
\begin{align*}
&\liminf_{\delta \downarrow 0} \frac{1}{2\delta} \limsup_{\ve \downarrow 0}
 \frac{W_2(\mu^{\delta}_{\ve},\mu_t)^2-W_2(\mu^{\delta},\mu_t)^2}{\ve}
 \le -\limsup_{\delta \downarrow 0} \frac{1}{\delta} \int_M
 \langle \nabla w_t,\nabla\phi \rangle \,d\mu^{\delta} \\
&\le \liminf_{\delta \downarrow 0} \frac{1}{\delta} \bigg[ \int_M
 \{ w_t-w_t(\cT) \} \,d\mu^{\delta} +CW_2(\mu^{\delta},\mu_t)^2 \bigg]
 = \liminf_{\delta \downarrow 0} \frac{1}{\delta} \bigg\{
 \int_M w_t \,d\mu^{\delta} -\int_M w_t \,d\mu_t \bigg\}.
\end{align*}

Next we calculate the difference of the entropies in $(\ref{eq:gf1})$.
We put $\mu^{\delta}=\rho^{\delta}\omega$,
$\mu^{\delta}_{\ve}=\rho^{\delta}_{\ve} \omega$ and
$\bJ^{\omega}_{\ve}:=e^{f-f(\cF_{\ve})}\det(D\cF_{\ve})$.
Then we obtain from Proposition~\ref{pr:dHm} that, as $w_t \in C^{\infty}(M)$,
\[ \lim_{\ve \downarrow 0} \frac{H_{\varphi}(\mu^{\delta})-H_{\varphi}(\mu^{\delta}_{\ve})}{\ve}
 =\int_M \big[ \{ h'_{\varphi}(\rho^{\delta}) \rho^{\delta}-h_{\varphi}(\rho^{\delta}) \} \Delta^{\omega}w_t
 +\langle \rho^{\delta} \nabla[\ln_{\varphi}(\sigma)],\nabla w_t \rangle \big] \,d\omega \]
(we need the conditions $\Ric_{N_{\varphi}} \ge 0$ and $\Hess\Psi \ge K$
only here for applying Proposition~\ref{pr:dHm}).
Hence we deduce that, together with Claim~\ref{cl:gf}(ii), (iii),
\begin{align*}
&\lim_{\delta \downarrow 0} \lim_{\ve \downarrow 0}
 \frac{H_{\varphi}(\mu^{\delta})-H_{\varphi}(\mu^{\delta}_{\ve})}{\ve} 
 = \int_M \big[ \{ h'_{\varphi}(\rho_t)\rho_t -h_{\varphi}(\rho_t) \}
 \Delta^{\omega}w_t -\langle \rho_t\nabla\Psi,\nabla w_t \rangle \big] \,d\omega \\
&= -\int_M \langle \nabla[h'_{\varphi}(\rho_t)\rho_t -h_{\varphi}(\rho_t)]+\rho_t\nabla\Psi,\nabla w_t \rangle \,d\omega
 =-\int_M \bigg\langle \frac{\nabla\rho_t}{\varphi(\rho_t)}+\nabla\Psi,\nabla w_t \bigg\rangle \,d\mu_t.
\end{align*}

These together imply
\[ \liminf_{\delta \downarrow 0} \frac{1}{\delta} \bigg\{
 \int_M w_t \,d\mu^{\delta} -\int_M w_t \,d\mu_t \bigg\}
 \ge -\int_M \bigg\langle \frac{\nabla\rho_t}{\varphi(\rho_t)}+\nabla\Psi,\nabla w_t \bigg\rangle \,d\mu_t. \]
Moreover, equality holds since we can change $w$ into $-w$.
Recall from \cite[(5)]{GO} (see also \cite[Lemma~6.4]{Ogra}) that
\[ \lim_{\delta \downarrow 0} \frac{1}{\delta}\bigg\{ \int_M \eta \,d\mu_{t+\delta}
 -\int_M \eta \,d\mu^{\delta} \bigg\}=0 \]
holds for all $\eta \in C^{\infty}(M)$.
Therefore we conclude
\begin{align*}
&\lim_{\delta \downarrow 0} \frac{1}{\delta} \bigg\{
 \int_M w_{t+\delta} \,d\mu_{t+\delta} -\int_M w_t \,d\mu_t \bigg\} \\
&= \lim_{\delta \downarrow 0} \frac{1}{\delta} \bigg\{
 \int_M (w_{t+\delta}-w_t) \,d\mu_{t+\delta} +\int_M w_t \,d\mu_{t+\delta}
 -\int_M w_t \,d\mu_t \bigg\} \\
&= \int_M \bigg\{ \frac{\del w_t}{\del t}
 -\bigg\langle \frac{\nabla\rho_t}{\varphi(\rho_t)}+\nabla\Psi,\nabla w_t \bigg\rangle \bigg\} \,d\mu_t.
\end{align*}
This shows \eqref{eq:wpme} by integration in $t$.
$\qedd$
\end{proof}

\begin{remark}\label{rm:gfac}
In Theorem~\ref{th:gf}, assuming that $\mu_t$ is absolutely continuous is in fact redundant.
If $L_{\varphi}=\infty$, then $H_{\varphi}(\mu_t)<\infty$ guarantees
$\mu_t \in \cP_{\ac}(M,\omega)$ by definition.
As for $L_{\varphi}<\infty$, if $\mu_t$ with $t>0$ has a nontrivial singular part $\mu^s$,
then we can modify $\mu_t$ as in the proof of Claim~\ref{cl:gf}(i) below
(with $\mu^{\delta}=\mu_t$ and $\pi=\diag_{\sharp}\mu_t$ where $\diag(x):=(x,x)$) and obtain
$\hat{\mu}_r \in \cP_{\ac}(M,\omega)$ for small $r>0$ such that
\[ W_2(\hat{\mu}_r,\mu_t)^2 \le \mu^s[M]r^2, \qquad
 \lim_{r \downarrow 0} \frac{H_{\varphi}(\hat{\mu}_r)-H_{\varphi}(\mu_t)}{r}=-\infty. \]
This yields $|\grad H_{\varphi}|(\mu_t)=\infty$ and contradicts the definition of gradient curves
(compare this discussion with \cite[Theorem~10.4.8]{AGS}).
\end{remark}

Combining Theorems~\ref{th:mCD}, \ref{th:cont}, \ref{th:gf}, we obtain the following.

\begin{corollary}\label{cr:gf}
Let $(M,\omega,\varphi,\Psi)$ be an admissible space as in Theorem~$\ref{th:gf}$,
and further suppose that $M_{\varphi}^{\Psi}$ is totally convex.
Then the weak solution $(\mu_t)_{t \in [0,\infty)} \subset \cP_{\ac}(M_\varphi^\Psi,\omega)$
to the $\varphi$-heat equation constructed in Theorem~$\ref{th:gf}$ satisfies
the $K$-contraction property $(\ref{eq:cont})$.
\end{corollary}

%%%%%%%%%%%%%%%%%%%%%%%%%%%%%%%%%%%
\subsection{Proof of Claim~\ref{cl:gf}}%%%%%%%%%%%%%%%%%

(i) The existence follows from the compactness of $\cP(M)$
and the lower semi-continuity of $H_{\varphi}$ (Lemma~\ref{lm:lsc}).
The absolute continuity is obvious if $L_{\varphi}=\infty$.

Assume $L_{\varphi}<\infty$, so that $\theta_{\varphi} \in (1,(n+1)/n)$
and $N_{\varphi}=(\theta_{\varphi}-1)^{-1} \in (n,\infty)$ (Proposition~\ref{lm:case}(ii)).
We decompose $\mu^{\delta}$ into absolutely continuous and singular
parts $\mu^{\delta}=\rho\omega+\mu^s$ and suppose $\mu^s[M]>0$.
For small $r>0$, we modify $\mu^{\delta}$ into
$\hat{\mu}_r =\hat{\rho}_r \omega \in \cP_{\ac}(M,\omega)$ as
\[ \hat{\rho}_r(x)
 :=\rho(x) +\int_M \frac{\chi_{B(y,r)}(x)}{\omega[B(y,r)]} \,d\mu^s(y). \]
We shall show that $\hat{\mu}_r$ gives a better choice than $\mu^{\delta}$
in our approximation scheme \eqref{eq:apprx}, which is a contradiction and hence $\mu^s[M]=0$.
We first observe
\begin{align}
\int_M h'_{\varphi}(\sigma) \,d\hat{\mu}_r
&\ge \int_M h'_{\varphi}(\sigma) \,d\mu^{\delta}
 -\int_M \bigg| h'_{\varphi}\big( \sigma(y)\big)
 -\frac{1}{\omega[B(y,r)]}\int_{B(y,r)} h'_{\varphi}(\sigma) \,d\omega \bigg| \,d\mu^s(y) \nonumber\\
&\ge \int_M h'_{\varphi}(\sigma) \,d\mu^{\delta}
 -\Big\{ \sup_M |\nabla(h'_{\varphi} \circ \sigma)| \cdot r \Big\} \mu^s[M]. \label{eq:cla1}
\end{align}
Note that,  on $M^{\Psi}_{\varphi}$, $h'_{\varphi}(\sigma)=-\Psi-L_{\varphi}$
is Lipschitz since $\Psi$ is Lipschitz.
Given an optimal coupling $\pi=\pi_1+\pi_2$ of $\mu^{\delta}$ and $\mu_t$
such that $(p_1)_{\sharp}\pi_1=\rho\omega$ and $(p_1)_{\sharp}\pi_2=\mu^s$,
\[ d\hat{\pi}_r(x,z):=d\pi_1(x,z)
 +\int_{y\in M} \frac{\chi_{B(y,r)}(x)}{\omega[B(y,r)]} \,d\omega(x) \,d\pi_2(y,z) \]
is a coupling of $\hat{\mu}_r$ and $\mu_t$.
Hence we find
\begin{align}
W_2(\hat{\mu}_r,\mu_t)^2
&\le \int_{M \times M} d_g(x,z)^2 \,d\pi_1(x,z)
 +\int_{M \times M} \{ d_g(y,z)+r \}^2 \,d\pi_2(y,z) \nonumber\\
&\le \int_{M \times M} d_g(x,z)^2 \,d\pi(x,z)
 +\{ 2\diam M +r \} r\pi_2[M \times M] \nonumber\\
&\le W_2(\mu^{\delta},\mu_t)^2 +\{ 3\diam M \cdot r \} \mu^s[M]. \label{eq:cla2}
\end{align}

Next, observe that
\[ \int_M h_{\varphi}(\hat{\rho}_r) \,d\omega
 = \int_M h_{\varphi} \bigg( \int_M \bigg\{ \frac{\rho(x)}{\mu^s[M]}
 +\frac{\chi_{B(y,r)}(x)}{\omega[B(y,r)]} \bigg\} \,d\mu^s(y) \bigg) \,d\omega(x). \]
As $h_{\varphi}$ is convex, Jensen's inequality shows
\begin{align*}
&h_{\varphi} \bigg( \int_M \bigg\{ \frac{\rho(x)}{\mu^s[M]}
 +\frac{\chi_{B(y,r)}(x)}{\omega[B(y,r)]} \bigg\} \,d\mu^s(y) \bigg) \\
&\le \frac{1}{\mu^s[M]} \int_M h_{\varphi} \bigg( \rho(x)
 +\frac{\chi_{B(y,r)}(x)}{\omega[B(y,r)]}\mu^s[M] \bigg) \,d\mu^s(y).
\end{align*}
Since $h_{\varphi}$ is non-increasing, we deduce from the Fubini theorem that
\begin{align*}
\int_M h_{\varphi}(\hat{\rho}_r) \,d\omega 
&\le \frac{1}{\mu^s[M]} \int_M \bigg\{ \int_{M \setminus B(y,r)} h_{\varphi}(\rho) \,d\omega
 +\int_{B(y,r)} h_{\varphi}\bigg( \frac{\mu^s[M]}{\omega[B(y,r)]} \bigg) \,d\omega \bigg\} \,d\mu^s(y) \\
&\le \int_M h_{\varphi}(\rho) \,d\omega
 -\frac{1}{\mu^s[M]} \int_M \bigg( \int_{B(y,r)} h_{\varphi}(\rho) \,d\omega \bigg) \,d\mu^s(y) \\
&\quad +\sup_{y \in M} \bigg\{ \omega[B(y,r)] \cdot
 h_{\varphi}\bigg( \frac{\mu^s[M]}{\omega[B(y,r)]} \bigg) \bigg\}.
\end{align*}
By virtue of the compactness of $M$, there are constants $0<C_1 \le C_2$ such that
\[ C_1r^n \le \omega[B(y,r)] \le C_2r^n \]
for all $y \in M$ and small $r>0$.
Hence we have, as $h_{\varphi}$ is non-increasing and nonpositive,
\[ \sup_{y \in M} \bigg\{ \omega[B(y,r)] \cdot
 h_{\varphi}\bigg( \frac{\mu^s[M]}{\omega[B(y,r)]} \bigg) \bigg\}
 \le C_1r^n h_{\varphi}\bigg( \frac{\mu^s[M]}{C_2r^n} \bigg). \]
We  find, by the monotonicity of $\ln_\varphi$, Lemma~\ref{lem:mono}
and $N_{\varphi}=(\theta_{\varphi}-1)^{-1}$,
\begin{align*}
&\limsup_{r \downarrow 0} r^{N_{\varphi}-1}h_{\varphi}(r^{-N_{\varphi}})
 = \limsup_{r \downarrow 0} \left\{ r^{N_{\varphi}-1} \int_0^{r^{-N_{\varphi}}} \ln_{\varphi}(s) \,ds
 -r^{-1}L_\varphi \right\} \\
&\le \limsup_{r \downarrow 0} \left\{ r^{-1} \ln_\varphi ( r^{-N_{\varphi}})-r^{-1} L_\varphi \right\}
 = -\liminf_{r \downarrow 0} \int_{r^{-N_{\varphi}}}^{\infty} \frac{1}{r\varphi(s)} \,ds \\
&\le -\lim_{r \downarrow 0} \int_{r^{-N_{\varphi}}}^{\infty} \frac{s^{-\theta_{\varphi}}}{r} \,ds
 =\lim_{r \downarrow 0} \frac{r^{N_{\varphi}(\theta_{\varphi}-1)}}{(1-\theta_{\varphi})r}
 = \frac{1}{1-\theta_{\varphi}} <0.
\end{align*}
Hence we obtain, since $n<N_{\varphi}<\infty$,
\[ r^{n-1}h_{\varphi}(r^{-n})
 =r^{(n-N_{\varphi})/N_{\varphi}} \cdot (r^{n/N_{\varphi}})^{N_{\varphi}-1}
 h_{\varphi}\big( (r^{n/N_{\varphi}})^{-N_{\varphi}} \big)
 \to -\infty \]
as $r \downarrow 0$ (here we need the hypothesis $\theta_{\varphi}<(n+1)/n$).
Finally, for all $y \in \supp\mu^s$, the convexity of $h_{\varphi}$ yields
\begin{align*}
\int_{B(y,r)} h_{\varphi}(\rho) \,d\omega
&\ge \int_{B(y,r)} \{ h_{\varphi}(\sigma)+h'_{\varphi}(\sigma)(\rho-\sigma) \} \,d\omega \\
&= \int_{B(y,r)} \{ h_{\varphi}(\sigma)-h'_{\varphi}(\sigma) \sigma \} \,d\omega
 +\int_{B(y,r)} h'_{\varphi}(\sigma) \,d\mu.
\end{align*}
We therefore obtain
\begin{align*}
&\frac{1}{r} \bigg\{ \int_M h_{\varphi}(\hat{\rho}_r) \,d\omega
 -\int_M h_{\varphi}(\rho) \,d\omega \bigg\} \\
&\le -\frac{1}{r} \inf_{y \in M} \bigg[ \int_{B(y,r)} \{ h_{\varphi}(\sigma)-h'_{\varphi}(\sigma) \sigma \} \,d\omega
 +\int_{B(y,r)} h'_{\varphi}(\sigma) \,d\mu \bigg]
 +C_1 r^{n-1} h_{\varphi}\bigg( \frac{\mu^s[M]}{C_2 r^n} \bigg) \\
&\to -\infty
\end{align*}
as $r \downarrow 0$.
Combining this with \eqref{eq:cla1} and \eqref{eq:cla2}, we conclude that
\[ \lim_{r \downarrow 0} \frac{1}{r} \bigg\{ 
 H_{\varphi}(\hat{\mu}_r)+\frac{W_2(\hat{\mu}_r,\mu_t)^2}{2\delta}
 -H_{\varphi}(\mu^{\delta})-\frac{W_2(\mu^{\delta},\mu_t)^2}{2\delta} \bigg\}
 =-\infty. \]
This contradicts the choice of $\mu^{\delta}$ as a minimizer of \eqref{eq:apprx},
so that it holds $\mu^s[M]=0$.

(ii) By the choice of $\mu^{\delta}$, we have
\[ H_{\varphi}(\mu^{\delta}) +\frac{W_2(\mu^{\delta},\mu_t)^2}{2\delta}
 \le H_{\varphi}(\mu_t). \]
Together with $H_{\varphi}(\mu^{\delta}) \ge H_{\varphi}(\nu)$ (Lemma~\ref{lm:Hm}),
we immediately observe
\[ \lim_{\delta \downarrow 0}W_2(\mu^{\delta},\mu_t)^2
 \le \lim_{\delta \downarrow 0} 2\delta \{ H_{\varphi}(\mu_t)-H_{\varphi}(\nu) \}=0. \]
Thus $\mu^{\delta}$ converges to $\mu_t$ weakly, and hence
\[ \limsup_{\delta \downarrow 0} \frac{W_2(\mu^{\delta},\mu_t)^2}{2\delta}
\le H_{\varphi}(\mu_t) -\liminf_{\delta \downarrow 0} H_{\varphi}(\mu^{\delta}) \le 0 \]
by the lower semi-continuity of $H_{\varphi}$ (Lemma~\ref{lm:lsc}).
These further yield
\[ H_{\varphi}(\mu_t) \le \liminf_{\delta \downarrow 0} H_{\varphi}(\mu^{\delta})
 \le \limsup_{\delta \downarrow 0} H_{\varphi}(\mu^{\delta}) \le H_{\varphi}(\mu_t). \]

(iii) This is a consequence of the following lemma.

\begin{lemma}\label{lm:gf}
Assume that $\theta_\varphi \in (0,2)$ and
\[ C_\varphi:=\lim_{s \uparrow \infty} \frac{s^{\theta_\varphi}}{\varphi(s)} <\infty. \]
If a sequence $\{\mu_i\}_{i \in \N} \subset \cP_{\ac}(M,\omega)$
converges to $\mu \in \cP_{\ac}(M,\omega)$ weakly and satisfies
$\lim_{i \to \infty}H_{\varphi}(\mu_i)=H_{\varphi}(\mu)<\infty$,
then, by setting $\mu_i=\rho_i \omega$ and $\mu=\rho\omega$,
the function $h_{\varphi}(\rho_i)-\rho_i h'_{\varphi}(\rho_i)$
converges to $h_{\varphi}(\rho)-\rho h'_{\varphi}(\rho)$ in $L^1(M,\omega)$.
\end{lemma}

\begin{proof}
We first show the following claim by using $\theta_{\varphi}<2$.

\begin{claim}\label{cl:gf1}
For any $C>0$, it holds
\[ \lim_{i \to \infty}\| \min\{\rho,C\} -\min\{\rho_i,C\} \|_{L^2(M,\omega)}=0. \]
\end{claim}

\begin{proof}
Assume the contrary, that is, there are some constants $C,\ve>0$
such that, taking a subsequence of $\{\rho_i\}_{i \in \N}$ if necessary, we have
\begin{equation}\label{eq:not*}
\| \min\{\rho,C\} -\min\{\rho_i,C\} \|_{L^2(M,\omega)} \ge \ve
\end{equation}
for all $i$.
Now, since $h''_{\varphi}(s)=\varphi(s)^{-1}$ is positive and non-increasing, we find
\[ h_{\varphi}\bigg( \frac{\rho+\rho_i}{2} \bigg)
 \le \frac{h_{\varphi}(\rho)+h_{\varphi}(\rho_i)}{2}
 -\frac{|\rho-\rho_i|^2}{8\max\{ \varphi(\rho),\varphi(\rho_i) \}}. \]
We shall further deduce from $\theta_{\varphi} < 2$ that
\begin{equation}\label{eq:for*}
\frac{|\rho-\rho_i|^2}{\max\{ \varphi(\rho),\varphi(\rho_i) \}}
 \ge \frac{|\min\{\rho,C\} -\min\{\rho_i,C\}|^2}{\varphi(C)}.
\end{equation}
This is clear if $\max\{\rho,\rho_i\} \le C$ or $\min\{\rho,\rho_i\} \ge C$.
Otherwise, \eqref{eq:for*} is reduced to
\[ \frac{(\tau-\ve)^2}{\varphi(\tau)} \ge \frac{(C-\ve)^2}{\varphi(C)},
 \qquad \ve \le C \le \tau, \]
and to the monotonicity of the function $s \mapsto (s-\ve)^2/\varphi(s)$ for $s>\ve$.
This monotonicity is easily seen by Lemma~\ref{lem:mono}, since $\theta_{\varphi}<2$ and
\[ \frac{(s-\ve)^2}{\varphi(s)}=\frac{s^{\theta_{\varphi}}}{\varphi(s)} \cdot s^{2-\theta_{\varphi}}
 \cdot \bigg( \frac{s-\ve}{s} \bigg)^2. \]

Thus we obtain from the hypothesis \eqref{eq:not*} that
\begin{align*}
\int_M h_{\varphi}\bigg( \frac{\rho+\rho_i}{2} \bigg) \,d\omega
&\le \int_M \frac{h_{\varphi}(\rho)+h_{\varphi}(\rho_i)}{2} \,d\omega
 -\frac{1}{8\varphi(C)} \| \min\{\rho,C\} -\min\{\rho_i,C\} \|^2_{L^2(M,\omega)} \\
&\le \frac{1}{2} \int_M h_{\varphi}(\rho) \,d\omega
 +\frac{1}{2} \int_M h_{\varphi}(\rho_i) \,d\omega -\frac{1}{8\varphi(C)} \ve^2.
\end{align*}
However, as $\lim_{i \to \infty}H_{\varphi}(\mu_i)=H_{\varphi}(\mu)$ by assumption,
this means that $\bar{\mu}_i:=\{(\rho+\rho_i)/2\}\omega$ satisfies
\[ \limsup_{i \to \infty}H_{\varphi}(\bar{\mu}_i)
 \le H_{\varphi}(\mu) -\frac{1}{8\varphi(C)} \ve^2. \]
This contradicts the lower semi-continuity of $H_{\varphi}$ (Lemma~\ref{lm:lsc})
and we complete the proof of Claim~\ref{cl:gf1}.
$\hfill \diamondsuit$
\end{proof}

Observe that
\[ h_{\varphi}(r)-r h'_{\varphi}(r)
 =\int_0^r \{ \ln(s)-\ln(r) \} \,ds =-\int_0^r \int_s^r \frac{1}{\varphi(t)} \,dtds
 =-\int_0^r \frac{t}{\varphi(t)} \,dt. \]
Combining this with Lemma~\ref{lem:mono}, we have for any $r,s>0$
\begin{align*}
|h_{\varphi}(r)-r h'_{\varphi}(r)-h_{\varphi}(s)-s h'_{\varphi}(s)|
&=\left| \int_s^r \frac{t}{\varphi(t)} \,dt \right|
 \le C_{\varphi} \left| \int_s^r t^{1-\theta_{\varphi}} \,dt \right| 
= \frac{C_\varphi}{2-\theta_{\varphi}} |r^m-s^m|,
\end{align*}
where we set $m=2-\theta_\varphi>0$.
Thus we deduce that
\begin{equation}\label{eq:delta1}
\int_M \left| h_{\varphi}(\rho_i)-\rho_i h'_{\varphi}(\rho_i)-h_{\varphi}(\rho)-\rho h'_{\varphi}(\rho) \right| \,d\omega
 \le \frac{C_\varphi}{2-\theta_{\varphi}} \int_M |\rho_i^m -\rho^m| \,d\omega.
\end{equation}
We are done if the right hand side tends to zero as $i \to \infty$.

\begin{claim}\label{cl:gf2}
For $m=2-\theta_{\varphi} \in (0,2)$, we have
\[ \rho,\rho_i \in L^m(M,\omega), \qquad
 \lim_{i \to \infty} \|\rho_i -\rho\|_{L^m(M,\omega)}=0. \]
\end{claim}

\begin{proof}
The first assertion is clear when $m \le 1$.
For $m>1$, it is a consequence of $h_{\varphi}(\rho), h_{\varphi}(\rho_i) \in L^1(M,\omega)$
(guaranteed by $H_{\varphi}(\mu), H_{\varphi}(\mu_i)<\infty$).
Indeed, by Lemma~\ref{lem:mono} and \eqref{eq:ln},
we have on $\{x \in M \,|\, \rho(x) \ge C\}$ for any $C>0$
\[ u_{\varphi}(\rho)-u_{\varphi}(C)
 =\int_C^{\rho} \ln_\varphi(s) \,ds \ge \int_C^{\rho} \ell_m(s) \,ds
 =\frac{\rho^m-C^m-m(\rho-C)}{m(m-1)}, \]
which implies $\max\{\rho,C\} \in L^m(M,\omega)$ since $m-1>0$,
$u_{\varphi}(\rho)-u_{\varphi}(C) \ge 0$ and $u_{\varphi}(\rho) \in L^1(M,\omega)$.
Thus we obtain $\rho \in L^m(M,\omega)$ and $\rho_i \in L^m(M,\omega)$ similarly.
We remark that, as $\lim_{i \to \infty}H_{\varphi}(\mu_i)=H_{\varphi}(\mu)$ by assumption,
we have $\lim_{i \to \infty} \int_M u_{\varphi}(\rho_i) \,d\omega=\int_M u_{\varphi}(\rho) \,d\omega$
so that $\int_M \rho_i^m \,d\omega$ is uniformly bounded in $i$.

As for the second estimate, thanks to Claim~\ref{cl:gf1} and $m<2$,
it suffices to show that $\rho_i -\min\{\rho_i,C\}$ converges to
$\rho-\min\{\rho,C\}$ in $L^m(M,\omega)$ for some (arbitrarily fixed) $C>0$.
Note first that
\[ |(\rho_i -\min\{\rho_i,C\})-(\rho -\min\{\rho,C\})|
 =|\max\{\rho_i,C\}-\max\{\rho,C\}|. \]
We put $\rho_i^C:=\max\{\rho_i,C\}$ and $\rho^C:=\max\{\rho,C\}$ for brevity.
By the same argumentation as Claim~\ref{cl:gf1},
$\lim_{i \to \infty}H_{\varphi}(\mu_i)=H_{\varphi}(\mu)$ yields
\[ \lim_{i \to \infty} \int_M \frac{|\rho_i -\rho|^2}{\max\{\varphi(\rho_i),\varphi(\rho)\}} \,d\omega =0. \]
Since $\varphi$ is positive and non-decreasing, it holds
\[ \int_M \frac{|\rho_i -\rho|^2}{\max\{\varphi(\rho_i),\varphi(\rho)\}} \,d\omega
 \ge \int_M \frac{|\rho_i^C -\rho^C|^2}{\varphi(\rho_i^C)+\varphi(\rho^C)} \,d\omega. \]
It follows from the H\"older inequality that
\[ \|\rho_i^C -\rho^C\|_{L^m(M,\omega)}^m
 \le \bigg( \int_M \frac{|\rho_i^C -\rho^C|^2}{\varphi(\rho_i^C)+\varphi(\rho^C)} \,d\omega \bigg)^{m/2}
 \bigg( \int_M \{ \varphi(\rho_i^C)+\varphi(\rho^C) \}^{m/\theta_{\varphi}} \,d\omega \bigg)^{\theta_{\varphi}/2}. \]
Observe that
\[ \{ \varphi(\rho_i^C)+\varphi(\rho^C) \}^{m/\theta_{\varphi}}
 \le \begin{cases}
 \varphi(\rho_i^C)^{m/\theta_{\varphi}}+\varphi(\rho^C)^{m/\theta_{\varphi}} & \text{for}\ m \le 1, \\
 2^{m/\theta_{\varphi}-1}
 \{ \varphi(\rho_i^C)^{m/\theta_{\varphi}}+\varphi(\rho^C)^{m/\theta_{\varphi}} \} & \text{for}\ m>1.
 \end{cases} \]
We deduce from Lemma~\ref{lem:mono} that
\[ \varphi(\rho_i^C)^{m/\theta_{\varphi}}+\varphi(\rho^C)^{m/\theta_{\varphi}}
 \le \frac{\varphi(C)^{m/\theta_{\varphi}}}{C^m} \{ (\rho_i^C)^m +(\rho^C)^m \}. \]
Since $\int_M (\rho_i^C)^m \,d\omega$ is uniformly bounded in $i$, we find
\[ \limsup_{i \to \infty} \int_M \{ \varphi(\rho_i^C)+\varphi(\rho^C) \}^{m/\theta_{\varphi}} \,d\omega <\infty,  \]
and hence $\lim_{i \to \infty} \|\rho_i^C-\rho^C\|_{L^m(M,\omega)}=0$.
$\hfill \diamondsuit$
\end{proof}

Now we obtain, for $m \le 1$,
\[ \int_M |\rho_i^m -\rho^m| \,d\omega \le \int_M |\rho_i -\rho|^m \,d\omega
 \to 0 \quad (i \to \infty) \]
with the help of Claim~\ref{cl:gf2}.
Similarly, it holds for $m>1$ that
\begin{align*}
&\int_M |\rho_i^m -\rho^m| \,d\omega
 \le m\int_M |\rho_i -\rho| \max\{\rho_i,\rho\}^{m-1} \,d\omega \\
&\le m\bigg( \int_M |\rho_i -\rho|^m \,d\omega \bigg)^{1/m}
 \bigg(  \int_M (\rho_i +\rho)^m \,d\omega \bigg)^{(m-1)/m}
 \ \to 0 \quad (i \to \infty).
\end{align*}
$\qedd$
\end{proof}

We remark that, in Lemma~\ref{lm:gf} and hence in Theorem~\ref{th:gf},
the assumptions $\theta_{\varphi} \in (0,2)$ and $C_{\varphi}<\infty$ can be replaced with
\[ \delta_{\varphi} \in (0,2), \qquad
 D_{\varphi}:=\lim_{s \downarrow 0} \frac{s^{\delta_{\varphi}}}{\varphi(s)}<\infty, \qquad
 d_{\varphi}:=\lim_{s \uparrow \infty} \frac{s^{\delta_{\varphi}}}{\varphi(s)}>0. \]
Indeed, then we have
\[ \frac{s^{\delta_{\varphi}}}{D_{\varphi}} \le \varphi(s) \le \frac{s^{\delta_{\varphi}}}{d_{\varphi}} \]
for all $s>0$, and \eqref{eq:delta1} becomes
\[ \int_M \left| h_{\varphi}(\rho_i)-\rho_i h'_{\varphi}(\rho_i)-h_{\varphi}(\rho)-\rho h'_{\varphi}(\rho) \right| \,d\omega
 \le \frac{D_\varphi}{2-\delta_{\varphi}} \int_M |\rho_i^m -\rho^m| \,d\omega \]
for $m:=2-\delta_{\varphi}$.
With this $m \in (0,2)$, Claim~\ref{cl:gf2} follows from Proposition~\ref{lm:cased}
and $\varphi(s) \le s^{\delta_{\varphi}}/d_{\varphi}$
(Claim~\ref{cl:gf1} is unnecessary in this case
since we can treat $\rho$ and $\rho_i$ themselves instead of $\rho^C$ and $\rho_i^C$).

Note that $C_{\varphi}=D_{\varphi}=d_{\varphi}=1<\infty$ for $\varphi_m(s)=s^{2-m}$.
For
\[ \varphi(s):= \begin{cases}
 \sqrt{s} & \text{for}\ 0<s<1, \\
 s & \text{for}\ s \ge 1,
\end{cases} \]
we have $\theta_{\varphi}=1$, $\delta_{\varphi}=1/2$, $C_{\varphi}=D_{\varphi}=1$
and $d_{\varphi}=0$.
An example of $\varphi$ with $C_{\varphi}=\infty$ is
\[ \varphi(s):= \begin{cases}
 \sqrt{s} & \text{for}\ 0<s<1, \\
 s & \text{for}\ 1 \le s \le 2, \\
 \sqrt{2s} & \text{for}\ s>2,
\end{cases} \]
for which $\theta_{\varphi}=1$, $\delta_{\varphi}=1/2$, $D_{\varphi}=1$ and $d_{\varphi}=1/\sqrt{2}$.

%%%%%%%%%%%%%%%%%%%%%%%%%%%%%%%%%%%%%
\section{Gradient flow of $H_{\varphi}$: Noncompact case}\label{sc:gf+}%%%%%%%%

We continue the study of gradient flows in the Wasserstein space $(\cP^2(M),W_2)$.
For noncompact $M$, we can not follow the intrinsic argument in Subsection~\ref{ssc:gf1}
since Theorem~\ref{th:angle} is unavailable.
We can nevertheless introduce a Riemannian structure of $\cP^2(M)$
using the underlying Riemannian structure of $M$.
Then gradient flows in $\cP^2(M)$ are also formulated with the help of
the underlying Riemannian/differentiable structure of $M$.
In order to see that the analogue of Theorem~\ref{th:gf} holds true,
we follow the argumentation in \cite{AGS}, \cite{Er} and \cite[Chapter~23]{Vi2}.
We refer to \cite{AGS} for the further deep theory of gradient flows.

\subsection{Riemannian structure of $(\cP^2(M),W_2)$}%%%%%%%

Recall that minimal geodesics in $\cP^2(M)$ emanating from absolutely continuous
measures are described by the gradient vector fields of appropriate functions (Theorem~\ref{th:FG}).
This leads the following definitions due to Otto~\cite{Ot} of the tangent spaces
and the Riemannian structure.

\begin{definition}[Otto's Riemannian structure]\label{df:Otto}
We set
\[ \hat{T}\cP :=\{ \Phi=\nabla\phi \,|\, \phi \in C_c^{\infty}(M) \} \]
and define the \emph{tangent space} $(T_{\mu}\cP^2,\langle \cdot,\cdot \rangle_{\mu})$ of $\cP^2(M)$
at $\mu \in \cP^2(M)$ as the completion of $\hat{T}\cP$ with respect to the norm
$\|\cdot\|_{\mu}$ induced from the inner product
\[ \langle \Phi_1,\Phi_2 \rangle_{\mu} :=\int_M \langle \Phi_1,\Phi_2 \rangle \,d\mu,
 \qquad \Phi_1,\Phi_2 \in \hat{T}\cP. \]
\end{definition}

Note that $\langle \cdot,\cdot \rangle$ is extended to the whole space $T_{\mu}\cP^2$
as the limit, and $(T_{\mu}\cP^2,\langle \cdot,\cdot \rangle_{\mu})$ is a Hilbert space.
We next introduce the class of `differentiable curves' in a purely metric way
(cf.\ \cite[Section~1.1]{AGS}).

\begin{definition}[Absolutely continuous curves]\label{df:abc}
For $p \in [1,\infty]$, a curve $(\mu_t)_{t \in I} \subset \cP^2(M)$ on an open interval $I \subset \R$
is said to be \emph{$p$-absolutely continuous} if there is some $\eta \in L_{\loc}^p(I)$ such that
\begin{equation}\label{eq:abc}
W_2(\mu_s,\mu_t) \le \int_s^t \eta(r) \,dr
\end{equation}
holds for all $s,t \in I$ with $s<t$.
\end{definition}

Note that $p$-absolutely continuous curves are continuous.
We will consider only $2$-absolutely continuous curves, so that we simply call
them absolutely continuous curves.
For any absolutely continuous curve $(\mu_t)_{t \in I} \subset \cP^2(M)$,
the \emph{metric derivative}
\[ |\dot{\mu}_t| :=\lim_{s \to t}\frac{W_2(\mu_s,\mu_t)}{|t-s|} \]
exists for a.e.\ $t \in I$, and $\eta(t)=|\dot{\mu}_t|$ is a minimal function satisfying \eqref{eq:abc}
(cf.\ \cite[Theorem~1.1.2]{AGS}).
We can associate a one-parameter family of vector fields on $M$ with
an absolutely continuous curve in $\cP^2(M)$ via the continuity equation on $M$.

\begin{proposition}{\rm (\cite[Theorem~8.3]{AGS}, \cite[Proposition~2.5]{Er})}\label{pr:tvf}
Given an absolutely continuous curve $(\mu_t)_{t \in I} \subset \cP^2(M)$,
there exists a Borel vector field $\Phi:I \times M \lra TM$ $($with $\Phi_t(x):=\Phi(t,x) \in T_xM)$
satisfying $\Phi_t \in T_{\mu_t}\cP^2$ for a.e.\ $t \in I$ as well as the \emph{continuity equation}
\[ \frac{\del \mu_t}{\del t} +\div(\Phi_t \mu_t) =0 \]
in the weak sense that
\begin{equation}\label{eq:ceq}
\int_I \int_M \bigg\{ \frac{\del w_t}{\del t}+\langle \Phi_t,\nabla w_t \rangle \bigg\} \,d\mu_t dt=0
\end{equation}
holds for all $w \in C_c^{\infty}(I \times M)$.
Such a vector field $\Phi$ $($satisfying $\Phi_t \in T_{\mu_t}\cP^2$ and \eqref{eq:ceq}$)$
is uniquely determined up to a difference on a null measure set with respect to $d\mu_t dt$,
and we have $\|\Phi_t\|_{\mu_t}=|\dot{\mu}_t|$ for a.e.\ $t \in I$.

Conversely, if a curve $(\mu_t)_{t \in I} \subset \cP^2(M)$ admits
a Borel vector field $\Phi: I \times M \lra TM$ satisfying \eqref{eq:ceq} and
$\int_{t_0}^{t_1} \|\Phi_t\|_{\mu_t}^2 \,dt <\infty$ for all $t_0,t_1 \in I$ with $t_0<t_1$,
then $(\mu_t)_{t \in I}$ is absolutely continuous and
$|\dot{\mu}_t| \le \|\Phi_t\|_{\mu_t}$ at a.e.\ $t \in I$.
\end{proposition}

\begin{definition}[Tangent vector fields]\label{df:tvf}
We say that the vector field $\Phi$ as in Proposition~\ref{pr:tvf} is the \emph{tangent vector field}
of the absolutely continuous curve $(\mu_t)_{t \in I}$, and write $\dot{\mu}_t=\Phi_t$
(for a.e.\ $t \in I$).
\end{definition}

It is guaranteed by the following \emph{Benamou--Brenier formula} (\cite{BB})
that Otto's Riemannian structure is compatible with the $W_2$-structure,
\[ W_2(\mu_0,\mu_1)=\inf_{(\mu_t)_{t \in [0,1]}}
 \bigg( \int_0^1 \|\dot{\mu}_t\|_{\mu_t}^2 \,dt \bigg)^{1/2} \]
for any $\mu_0,\mu_1 \in \cP^2(M)$,
where the infimum is taken over all absolutely continuous curves
$(\mu_t)_{t \in [0,1]} \subset \cP^2(M)$ from $\mu_0$ to $\mu_1$.

%%%%%%%%%%%%%%%%%%%%%%%%%%%
\subsection{Gradient flow of $H_{\varphi}$}%%%%%%%%%%

Using the Riemannian structure of $\cP^2(M)$ in the previous subsection,
we can formulate gradient curves (trajectories of gradient flow)
in a way different from the previous section.
We first define gradient vectors.

\begin{definition}[Gradient vectors]\label{df:gf+}
Given a functional $H:\cP^2(M) \lra (-\infty,\infty]$ and $\mu \in \cP_{\ac}^2(M)$
with $H(\mu)<\infty$, we say that $H$ is {\it differentiable} at $\mu$
if there is $\Phi \in T_{\mu}\cP^2$ such that
\[ \limsup_{t \downarrow 0}\frac{H(\mu_t)-H(\mu)}{t}
 \le \int_M \langle \Phi,\nabla\phi \rangle \,d\mu \]
along all minimal geodesics $(\mu_t)_{t \in [0,1]} \subset \cP^2(M)$ with $\mu_0=\mu$,
where $\mu_t=(\cT_t)_{\sharp}\mu$ with $\cT_t(x)=\exp_x(t\nabla\phi(x))$,
and if equality holds for $\phi \in C_c^{\infty}(M)$
(with $\lim_{t \downarrow 0}$ in place of $\limsup_{t \downarrow 0}$).
Such $\Phi$ is unique if it exists, so that we will write $\nabla_W H(\mu)=\Phi$.
\end{definition}

Note that $|\grad (-H)|(\mu) \le \|\nabla_W H(\mu)\|_{\mu}$ holds by the Cauchy--Schwarz inequality.
A gradient curve of the $\varphi$-relative entropy $H_{\varphi}$ should be understood
as a solution to $\dot{\mu}_t=\nabla_W[-H_{\varphi}](\mu_t)$.
Compare the next proposition with Proposition~\ref{pr:DH}.

\begin{proposition}\label{pr:gf+}
Let $(M,\omega,\varphi,\Psi)$ be admissible, assume $\Ric_{N_{\varphi}} \ge 0$
and $\Hess\Psi \ge K$ on $M^{\Psi}_{\varphi}$ for some $K \in \R$
$(K>0$ if $M$ is noncompact and $\theta_{\varphi}<1)$.
Fix $\mu=\rho\omega \in \cP^2_{\ac}(M,\omega)$ with $\mu[M^{\Psi}_{\varphi}]=1$,
$H_{\varphi}(\mu)<\infty$ and with $|\nabla\Psi| \in L^2(M,\mu)$.
Then the following are equivalent$:$
\begin{enumerate}[{\rm (I)}]
\item $|\grad H_{\varphi}|(\mu) <\infty$,
\item $\rho \in H^1_{\loc}(M)$ and
\[ \frac{\nabla\rho}{\varphi(\rho)} +\nabla\Psi=-\Phi \]
holds $\mu$-a.e.\ for some $\Phi \in T_{\mu}\cP^2$.
\end{enumerate}
Moreover, then we have $\Phi=\nabla_W[-H_{\varphi}](\mu)$ and
$\|\Phi\|_{\mu}=|\grad H_{\varphi}|(\mu)$.
\end{proposition}

\begin{proof}
(I) $\Rightarrow$ (II):
Note that, by the calculation (before the integration by parts) in the proof of Proposition~\ref{pr:dHm},
\begin{align*}
&\bigg| \int_M \big[ \{ h'_{\varphi}(\rho)\rho-h_{\varphi}(\rho) \} \div_{\omega} V
 -\langle \rho\nabla\Psi,V \rangle \big] \,d\omega \bigg| \\
&= \lim_{t \downarrow 0} \bigg\{
 \frac{H_{\varphi}(\mu)-H_{\varphi}(\mu_t)}{W_2(\mu,\mu_t)} \frac{W_2(\mu,\mu_t)}{t} \bigg\}
 \le |\grad H_{\varphi}|(\mu) \|V\|_{\mu}
\end{align*}
for all $C^{\infty}$-vector fields $V$ of compact support,
where we put $\mu_t=(\cT_t)_{\sharp}\mu$ with $\cT_t(x)=\exp_x(tV(x))$.
Hence the hypothesis (I) together with $\Psi \in H^1_{\loc}(M)$ ensures that
the function $h'_{\varphi}(\rho)\rho-h_{\varphi}(\rho)$ is weakly differentiable.
Since the function $s \longmapsto h'_{\varphi}(s)s-h_{\varphi}(s)$ is differentiable
and increasing in $s>0$, this implies $\rho \in H^1_{\loc}(M)$, and we observe
\[ \nabla [h'_{\varphi}(\rho)\rho-h_{\varphi}(\rho)]
 =\frac{\rho}{\varphi(\rho)} \nabla\rho. \]
Moreover, the above estimate shows that the function
\begin{align*}
\hat{T}\cP \ni \nabla\phi\ \longmapsto\
 \int_M \langle \nabla [h'_{\varphi}(\rho)\rho-h_{\varphi}(\rho)]+\rho\nabla\Psi, \nabla\phi \rangle \,d\omega
 = \int_M \bigg\langle \frac{\nabla\rho}{\varphi(\rho)}+\nabla\Psi, \nabla\phi \bigg\rangle \,d\mu
\end{align*}
is extended to a bounded linear operator on the closure $T_{\mu}\cP^2$.
Therefore the Riesz representation theorem shows that there exists $\Phi \in T_{\mu}\cP^2$ with
\begin{equation}\label{eq:|Psi|}
\|\Phi\|_{\mu} \le |\grad H_{\varphi}|(\mu), \qquad
 \int_M \bigg\langle \frac{\nabla\rho}{\varphi(\rho)}+\nabla\Psi, \Xi \bigg\rangle \,d\mu
 =\int_M \langle -\Phi,\Xi \rangle \,d\mu
\end{equation}
for all $\Xi \in T_{\mu}\cP^2$.
Thus we have $\nabla\rho/\varphi(\rho)+\nabla\Psi=-\Phi$ $\mu$-a.e..

(II) $\Rightarrow$ (I):
We remark that the condition $K>0$ for $\theta_{\varphi}<1$
makes Proposition~\ref{pr:dHm} applicable.
Thus we obtain
\begin{equation}\label{eq:delH}
\limsup_{t \downarrow 0}\frac{H_{\varphi}(\mu)-H_{\varphi}(\mu_t)}{t}
 \le \int_M \langle \Phi,\nabla\phi \rangle \,d\mu
\end{equation}
along every minimal geodesic $(\mu_t)_{t \in [0,1]} \subset \cP^2(M)$ with $\mu_0=\mu$,
where $\mu_t=(\cT_t)_{\sharp}\mu$ and $\cT_t(x)=\exp_x(t\nabla\phi(x))$,
and equality holds if $\phi \in C_c^{\infty}(M)$.
Hence $|\grad H_{\varphi}|(\mu)<\infty$ follows from the hypothesis $\Phi \in T_{\mu}\cP^2$,
and we find $\Phi=\nabla_W[-H_{\varphi}](\mu)$ in the sense of Definition~\ref{df:gf+}.
We have $\|\Phi\|_{\mu} \le |\grad H_{\varphi}|(\mu)$ by \eqref{eq:|Psi|},
and $|\grad H_{\varphi}|(\mu) \le \|\Phi\|_{\mu}$ by \eqref{eq:delH},
so that $\|\Phi\|_{\mu} =|\grad H_{\varphi}|(\mu)$ holds.
$\qedd$
\end{proof}

Now, we are ready to show the main result of the section.
We remark that the roles of the conditions $\Ric_{N_{\varphi}} \ge 0$
and $\Hess\Psi \ge K$ are implicit at this stage,
whereas they were necessary for applying Proposition~\ref{pr:dHm}.

\begin{theorem}[Gradient flow of $H_{\varphi}$]\label{th:gf+}
Suppose that $(M,\omega,\varphi,\Psi)$ is admissible and satisfies $\Ric_{N_{\varphi}} \ge 0$
as well as $\Hess\Psi \ge K$ on $M^{\Psi}_{\varphi}$ for some $K \in \R$
$(K>0$ if $M$ is noncompact and $\theta_{\varphi}<1)$.
Let $(\mu_t)_{t \in [0,\infty)} \subset \cP^2_{\ac}(M,\omega)$ be a continuous curve
such that $\mu_t[M^{\Psi}_{\varphi}]=1$, $H_{\varphi}(\mu_t)<\infty$
and $|\nabla\Psi| \in L^2(M,\mu_t)$ for all $t>0$.
Then $(\mu_t)_{t \in (0,\infty)}$ is an absolutely continuous curve satisfying
\[ \dot{\mu}_t=\nabla_W[-H_{\varphi}](\mu_t) \in T_{\mu_t}\cP^2 \]
at a.e.\ $t \in (0,\infty)$ if and only if $(\rho_t)_{t \in [0,\infty)}$ is
a weak solution to the $\varphi$-heat equation \eqref{eq:pme}
with $\int_{t_0}^{t_1} |\nabla\rho_t/\varphi(\rho_t)|^2 \,d\mu_t dt<\infty$ for all $0<t_0<t_1<\infty$,
where $\mu_t=\rho_t \omega$.
\end{theorem}

\begin{proof}
Suppose $\dot{\mu}_t=\nabla_W[-H_{\varphi}](\mu_t)$ a.e.\ $t$.
Since $|\grad H_{\varphi}|(\mu_t) \le \|\nabla_W[-H_{\varphi}](\mu_t)\|_{\mu_t}<\infty$ by definition,
Proposition~\ref{pr:gf+} yields
\[ \dot{\mu}_t =-\bigg( \frac{\nabla\rho_t}{\varphi(\rho_t)}+\nabla\Psi \bigg) \in T_{\mu_t}\cP^2
 \qquad \text{a.e.}\ t. \]
Then it follows from the continuity equation $(\ref{eq:ceq})$ that
\[ \int_0^{\infty} \int_M \frac{\del w_t}{\del t} \,d\mu_t dt
 =\int_0^{\infty} \int_M \bigg\langle \frac{\nabla\rho_t}{\varphi(\rho_t)}+\nabla\Psi,\nabla w_t \bigg\rangle \,d\mu_t dt \]
for all $w \in C_c^{\infty}((0,\infty) \times M)$.
Therefore $\rho_t$ weakly solves $(\ref{eq:pme})$.

Conversely, if $\rho_t$ is a weak solution to $(\ref{eq:pme})$
with $\int_{t_0}^{t_1} |\nabla\rho_t/\varphi(\rho_t)|^2 \,d\mu_t dt<\infty$,
then the same calculation implies that
\[ \Phi_t=-\bigg( \frac{\nabla\rho_t}{\varphi(\rho_t)}+\nabla\Psi \bigg) \]
satisfies the continuity equation $(\ref{eq:ceq})$, and hence $(\mu_t)_{t \in (0,\infty)}$
is absolutely continuous by Proposition~\ref{pr:tvf}.
As Proposition~\ref{pr:dHm} guarantees
$|\grad H_{\varphi}|(\mu_t) \le \|\Phi_t\|_{\mu_t} <\infty$ a.e.\ $t$ (by \eqref{eq:delH}),
Proposition~\ref{pr:gf+} shows $\Phi_t=\nabla_W[-H_{\varphi}](\mu_t) \in T_{\mu_t}\cP^2$
and then the uniqueness of a solution to the continuity equation (Proposition~\ref{pr:tvf})
yields $\dot{\mu}_t=\Phi_t=\nabla_W[-H_{\varphi}](\mu_t)$ a.e.\ $t$.
$\qedd$
\end{proof}

%%%%%%%%%%%%%%%%%%%%%%%%%%
\subsection{Remarks on construction and contraction}%%%%%%%%%%%%%

We can construct the gradient flow of $H_{\varphi}$ along the line of \cite[Section~5]{Er},
provided that $M$ is compact.
Precisely, we need the compactness for applying Lemma~\ref{lm:lsc},
Claim~\ref{cl:gf}(i) and Lemma~\ref{lm:gf}.
As for the contractivity (see \eqref{eq:cont}), the usual technique starts from the first variation formula
for the distance $W_2(\mu_t^1,\mu_t^2)$ between two gradient curves
(see, e.g., \cite[Proposition~4.4]{Er}).
To follow this line, however, we need (at least) the $C^1$-regularity of the density functions $\rho_t^i$.
The authors do not know if such a regularity can be expected
for our (nonlinear, scale-variant) $\varphi$-heat equation \eqref{eq:pme}.

Another recipe (for construction as well as contraction)
would be to apply the general theory of Savar\'e~\cite{Sa}.
Under $\Hess H_{\varphi} \ge K$ and the additional \emph{semi-concavity condition}
of the squared distance function (which is always true for compact Riemannian manifolds),
one can construct a unique gradient flow of $H_{\varphi}$ enjoying the $K$-contractivity \eqref{eq:cont}.
However, we should take care about the point that his (metric) definition of gradient flows
is different from the one discussed in Theorem~\ref{th:gf+}.
Thus, in particular, the existence of a gradient flow in our sense
does not follow from Savar\'e's result.

We also mention another interesting contribution due to Gigli~\cite{Gi1},
he showed the unique existence of the gradient flow of the relative entropy
in a quite general situation without relying on the contractivity.
As mentioned at the end of \cite{Gi1}, his technique uses
some special properties of the generating function $u_{\varphi_1}(s)=s\log s-s$
and is not applicable to all $\varphi$'s in our consideration
(e.g., $\varphi_m$ for $m<1$ is excluded).

%%%%%%%%%%%%%%%%%%%%%%%%%%%%%%%%%%%%%
\section{Finsler case}\label{sc:Fins}%%%%%%%%%%%%%%%%%%

Most results in this article are extended to Finsler manifolds
according to the theory of Ricci curvature developed in \cite{Oint}, \cite{OS}
(see also a survey~\cite{Osur}).
A Finsler manifold is a differentiable manifold equipped with
a (Minkowski) norm on each tangent space.
Restricting these norms to those coming from inner products,
we have the family of  Riemannian manifolds as a subclass.
We refer to \cite{BCS} and \cite{Sh} for the basics of Finsler geometry.

\subsection{Finsler manifolds}%%%%%%%%%%%%%

Let $M$ be a connected, $n$-dimensional
$C^{\infty}$-manifold without boundary.
Given a local coordinate $(x^i)_{i=1}^n$ on an open set $U \subset M$,
we will always use the coordinate $(x^i,\bv^j)_{i,j=1}^n$ of $TU$ such that
\[ \bv=\sum_{j=1}^n \bv^j \frac{\del}{\del x^j}\Big|_x \in T_xM, \qquad x \in U. \]

\begin{definition}[Finsler structures]\label{df:Fstr}
We say that a nonnegative function $F:TM \lra [0,\infty)$ is
a \emph{$C^{\infty}$-Finsler structure} of $M$ if the following three conditions hold:
\begin{enumerate}[(1)]
\item(Regularity)
$F$ is $C^{\infty}$ on $TM \setminus 0$,
where $0 \subset TM$ stands for the zero section.

\item(\emph{Positive $1$-homogeneity})
It holds $F(c\bv)=cF(\bv)$ for all $\bv \in TM$ and $c>0$.

\item(\emph{Strong convexity})
The $n \times n$ symmetric matrix
\begin{equation}\label{eq:gij}
\big( g_{ij}(\bv) \big)_{i,j=1}^n :=
 \bigg( \frac{1}{2}\frac{\del^2 (F^2)}{\del \bv^i \del \bv^j}(\bv) \bigg)_{i,j=1}^n
\end{equation}
is positive-definite for all $\bv \in T_xM \setminus 0$.
\end{enumerate}
We call such a pair $(M,F)$ a \emph{$C^{\infty}$-Finsler manifold}.
\end{definition}

In other words, $F$ provides a $C^{\infty}$-Minkowski norm (see Example~\ref{ex:Fins}(a) below)
on each tangent space $T_xM$ which varies smoothly also in the horizontal direction.
For $x,y \in M$, we define the distance from $x$ to $y$ in a natural way by
$d_F(x,y):=\inf_{\gamma} \int_0^1 F\big( \dot{\gamma}(t) \big) \,dt$,
where the infimum is taken over all $C^1$-curves $\gamma:[0,1] \lra M$
such that $\gamma(0)=x$ and $\gamma(1)=y$.
Note that $d_F$ is not necessarily symmetric,
namely $d_F(y,x) \neq d_F(x,y)$ can happen, since $F$ is only positively homogeneous.
A $C^{\infty}$-curve $\gamma$ on $M$ is called a \emph{geodesic}
if it is locally distance minimizing and has a constant speed
(i.e., $F(\dot{\gamma})$ is constant).
We remark that $t \longmapsto \gamma(1-t)$ may not be a geodesic.
Given $\bv \in T_xM$, if there is a geodesic $\gamma:[0,1] \lra M$
with $\dot{\gamma}(0)=\bv$, then we define the \emph{exponential map}
by $\exp_x(\bv):=\gamma(1)$.
We say that $(M,F)$ is \emph{forward complete} if the exponential
map is defined on whole $TM$.
Then the Hopf--Rinow theorem ensures that any pair of points
is connected by a minimal geodesic (cf.\ \cite[Theorem~6.6.1]{BCS}).

We define the $K$-convexity of a function $\Psi:M \lra \R$ in the weak sense
similarly to the case of symmetric distances (Definition~\ref{df:K-conv}), i.e., for any $x,y \in M$
there is a minimal geodesic $\gamma:[0,1] \lra M$ from $x$ to $y$ such that
\[ \Psi\big( \gamma(t) \big) \le (1-t)\Psi(x)+t\Psi(y) -\frac{K}{2}(1-t)t d_F(x,y)^2  \]
for all $t \in [0,1]$.

For each $\bv \in T_xM \setminus 0$, the positive-definite matrix
$(g_{ij}(\bv))_{i,j=1}^n$ in \eqref{eq:gij} induces
the Riemannian structure $g_{\bv}$ of $T_xM$ via
\begin{equation}\label{eq:gv}
g_{\bv}\bigg( \sum_{i=1}^n a_i \frac{\del}{\del x^i}\Big|_x,
 \sum_{j=1}^n b_j \frac{\del}{\del x^j}\Big|_x \bigg)
 := \sum_{i,j=1}^n g_{ij}(\bv) a_i b_j.
\end{equation}
This is regarded as the best Riemannian approximation of $F|_{T_xM}$
in the direction $\bv$.
In fact, the unit sphere of $g_{\bv}$ is tangent to that of $F|_{T_xM}$
at $\bv/F(\bv)$ up to the second order.
In particular, we have $g_{\bv}(\bv,\bv)=F(\bv)^2$.

Let us denote by $\cL^*:T^*M \lra TM$ the \emph{Legendre transform}.
Precisely, $\cL^*$ is sending $\alpha \in T_x^*M$ to the unique element $\bv \in T_xM$
such that $\alpha(\bv)=F^*(\alpha)^2$ and $F(\bv)=F^*(\alpha)$,
where $F^*$ stands for the dual norm of $F$.
Note that $\cL^*|_{T^*_xM}$ is a linear operator only when $F|_{T_xM}$
comes from an inner product.
For a differentiable function $\rho:M \lra \R$, the \emph{gradient vector}
of $\rho$ at $x$ is defined as the Legendre transform of the derivative of $\rho$,
\[ \nabla \rho(x):=\cL^*\big( D\rho(x) \big) \in T_xM. \]
If $D\rho(x)=0$, then clearly $\nabla \rho(x)=0$.
If $D\rho(x) \neq 0$, then we can write in coordinates
\[ \nabla \rho
 =\sum_{i,j=1}^n g^{ij}(\nabla \rho) \frac{\del \rho}{\del x^j} \frac{\del}{\del x^i}, \]
where $(g^{ij})$ stands for the inverse matrix of $(g_{ij})$.
We must be careful when $D\rho(x)=0$,
because $g_{ij}(\nabla \rho(x))$ is not defined as well as
the Legendre transform $\cL^*$ being only continuous at the zero section.
We also remark that the gradient $\nabla$ is a nonlinear operator
(i.e., $\nabla(\rho_1+\rho_2)(x) \neq \nabla \rho_1(x)+\nabla \rho_2(x)$
and $\nabla(-\rho)(x) \neq -\nabla \rho(x)$ in general),
since the Legendre transform is nonlinear unless $F$ happens to be Riemannian.

We mention some of basic examples of non-Riemannian Finsler manifolds.

\begin{example}\label{ex:Fins}
(a) (Minkowski spaces)
A \emph{Minkowski norm} $|\cdot|$ on $\R^n$ is a nonnegative function on $\R^n$
satisfying the conditions in Definition~\ref{df:Fstr}.
Note that the unit ball of $|\cdot|$ is a strictly convex
(but not necessarily symmetric to the origin) domain containing the origin in its interior.
A Minkowski norm induces a Finsler structure
in a natural way through the identification between $T_x\R^n$ and $\R^n$.
Then $(\R^n,|\cdot|)$ has the flat flag curvature
(the \emph{flag curvature} is a generalization of the sectional curvature).

(b) (Randers spaces)
A \emph{Randers space} $(M,F)$ is a special kind of Finsler manifold given by
$F(\bv)=\sqrt{g(\bv,\bv)} +\beta(\bv)$
for some Riemannian metric $g$ and a one-form $\beta$,
where we suppose $|\beta(\bv)|^2<g(\bv,\bv)$ unless $\bv=0$,
for $F$ being positive on $TM \setminus 0$.
Randers spaces are important in applications and reasonable for concrete calculations.
Sometimes $\beta$ is regarded as the effect of wind blowing on the Riemannian manifold $(M,g)$.

(c) (Hilbert geometry)
Let $D \subset \R^n$ be a bounded open set with smooth boundary
such that its closure $\overline{D}$ is strictly convex.
Then the associated \emph{Hilbert distance function} is defined by
\[ d_H(x_1,x_2):=\log \bigg(
 \frac{|x_1-x'_2| \cdot |x_2-x'_1|}{|x_1-x'_1| \cdot |x_2-x'_2|} \bigg) \]
for distinct $x_1,x_2 \in D$, where $|\cdot|$ is the standard Euclidean norm
and $x'_1,x'_2$ are intersections of $\del D$ and the line passing through
$x_1,x_2$ such that $x'_i$ is on the side of $x_i$.
Hilbert geometry is known to be realized by a Finsler structure
with constant negative flag curvature, and gives the Klein model of hyperbolic space
if $D$ is an ellipsoid.

(d) (Teichm\"uller space)
Teichm\"uller metric on Teichm\"uller space is arguably one of the most famous
Finsler structures in differential geometry.
It is known to be complete, while, e.g., the Weil--Petersson metric is incomplete
and Riemannian.
\end{example}

%%%%%%%%%%%%%%%%%%%%%%%%%%%%%%%%
\subsection{Weighted Ricci curvature and nonlinear Laplacian}%%%%%%%

Different from the Riemannian situation, one can not choose a unique canonical measure
on a Finsler manifold.
There are several constructive measures, such as the \emph{Busemann--Hausdorff measure} and
the \emph{Holmes--Thompson measure}, which are canonical in their own ways (see, e.g., \cite{AT}).
Thus we will fix an arbitrary positive $C^{\infty}$-measure $\omega$ on $M$
as our base measure, like the theory of weighted Riemannian manifolds.

The \emph{Ricci curvature} (as the trace of the flag curvature) on a Finsler manifold
is defined by using the Chern connection (there are other connections
but the flag and Ricci curvatures are in fact independent of the choice of connection).
Instead of giving a precise definition in coordinates,
here we explain a useful interpretation due to Shen \cite[\S 6.2]{Sh}.
Given a unit vector $\bv \in T_xM \cap F^{-1}(1)$, we extend it to a $C^{\infty}$-vector field $V$
on a neighborhood of $x$ in such a way that every integral curve of $V$ is geodesic,
and consider the Riemannian structure $g_V$ induced from \eqref{eq:gv}.
Then the Ricci curvature $\Ric(\bv)$ of $\bv$ with respect to $F$ coincides with
the Ricci curvature of $\bv$ with respect to $g_V$
(in particular, it is independent of the choice of $V$).

Inspired by the above interpretation of the Ricci curvature as well as the theory of
weighted Riemannian manifolds, the \emph{weighted Ricci curvature} for $(M,F,\omega)$
was introduced in \cite{Oint} as follows.

\begin{definition}[Weighted Ricci curvature]\label{df:wRic}
Given a unit vector $\bv \in T_xM$, let $\gamma:(-\ve,\ve) \lra M$ be the geodesic
such that $\dot{\gamma}(0)=\bv$.
We decompose $\omega$ as $\omega=e^{-f}\vol_{\dot{\gamma}}$ along $\gamma$,
where $\vol_{\dot{\gamma}}$ is the volume form of $g_{\dot{\gamma}}$.
Define
\begin{enumerate}[(1)]
\item $\Ric_n(\bv):=\displaystyle \left\{
 \begin{array}{ll} \Ric(\bv)+(f \circ \gamma)''(0) & {\rm if}\ (f \circ \gamma)'(0)=0, \\
 -\infty & {\rm otherwise}, \end{array} \right.$

\item $\Ric_N(\bv):=\Ric(\bv) +(f \circ \gamma)''(0) -\displaystyle\frac{(f \circ \gamma)'(0)^2}{N-n}\ $
for $N \in (-\infty,0) \cup (n,\infty)$,

\item $\Ric_{\infty}(\bv):=\Ric(\bv) +(f \circ \gamma)''(0)$.
\end{enumerate}
For $c \ge 0$, we set $\Ric_N(c\bv):=c^2 \Ric_N(\bv)$.
\end{definition}

It is established in \cite[Theorem~1.2]{Oint} that, for $K \in \R$ and $N \in [n,\infty]$,
the bound $\Ric_N(\bv) \ge KF(\bv)^2$ is equivalent to the curvature-dimension condition $\CD(K,N)$
(note that $(M,d_F)$ is non-branching and thus Sturm's and Lott--Villani's conditions are equivalent).
This extends the corresponding result on weighted Riemannian manifolds
(Theorems~\ref{th:CD}, \ref{th:LVCD}).
There are further applications of $\Ric_N$ beyond the curvature-dimension condition,
e.g., a Bochner-type formula and gradient estimates (\cite{OS3}).

\begin{remark}\label{rm:Ric}
For a Riemannian manifold $(M,g,\vol_g)$ endowed with the Riemannian volume measure,
clearly we have $f \equiv 0$ and hence $\Ric_N =\Ric$ for all $N$.
It is also known that, for Finsler manifolds of \emph{Berwald type},
the Busemann--Hausdorff measure satisfies $(f \circ \gamma)' \equiv 0$
(in other words, Shen's \emph{$\mathbf{S}$-curvature} vanishes, see \cite[\S 7.3]{Sh}).
In general, however, there may not exist any measure $\omega$ of vanishing
$\mathbf{S}$-curvature, see \cite{ORand} for such an example.
This means that, on a general Finsler manifold, there is no measure
as good as the Riemannian volume measure.
This is a reason why we began with an arbitrary measure $\omega$.
\end{remark}

Define the \emph{divergence} of a differentiable vector field $V$ on $M$
with respect to the base measure $\omega$ by
\[ \div_{\omega} V:=\sum_{i=1}^n
 \bigg( \frac{\del V_i}{\del x^i} +V_i \frac{\del \eta}{\del x^i} \bigg), \]
where we decompose $\omega$ in coordinates as $d\omega=e^{\eta} \,dx^1 dx^2 \cdots dx^n$.
Similarly to the Riemannian case,
this can be rewritten (and extended to weakly differentiable vector fields) in the weak form as
\[ \int_M w \div_{\omega} V \,d\omega =-\int_M Dw(V) \,d\omega \]
for all $w \in C_c^{\infty}(M)$.
Then we define the corresponding \emph{Laplacian} of $\rho \in H^1_{\loc}(M)$ by
$\Delta^{\omega} \rho:=\div_{\omega}(\nabla \rho)$ in the distributional sense that
\[ \int_M w\Delta^{\omega} \rho \,d\omega:=-\int_M Dw(\nabla \rho) \,d\omega \]
for $w \in C_c^{\infty}(M)$.
We remark that $H^1_{\loc}(M)$ is defined solely in terms of the differentiable structure of $M$.
It is established in \cite{OS} and \cite{OS3} that this nonlinear Laplacian
works quite well with the weighted Ricci curvature.

For later convenience, we introduce the following notations.

\begin{definition}[Reverse Finsler structure]\label{df:rev}
Define the \emph{reverse Finsler structure} $\overleftarrow{F}$ of $F$ by
$\overleftarrow{F}(\bv):=F(-\bv)$.
We will put arrows $\leftarrow$ on those quantities associated with $\overleftarrow{F}$,
for example, $\overleftarrow{d}\!_F(x,y)=d_F(y,x)$, $\overleftarrow{\nabla}\rho=-\nabla(-\rho)$
and $\overleftarrow{\Ric}_N(\bv)=\Ric_N(-\bv)$.
\end{definition}

%%%%%%%%%%%%%%%%%%%%%%%%%%%%%%%%%%%
\subsection{Displacement convexity of $H_{\varphi}$ and applications}%%%%%%%

From now on, we consider only compact Finsler manifolds for simplicity.
We remark that all compact Finsler manifolds are forward complete.

Let us consider an admissible space $(M,\omega,\varphi,\Psi)$ in the sense of Definition~\ref{df:adm}
similarly to the Riemannian case.
Then the analogue of Theorem~\ref{th:mCD} is demonstrated along the same line as the Riemannian case
(see \cite{Oint} for details).

We can show the functional inequalities in Theorem~\ref{thm:inequ} also in the same way
by using the directional derivative of $H_{\varphi}$ (see \eqref{eq:dHm}) modified into
\[ \liminf_{t \downarrow 0} \frac{H_{\varphi}(\mu_t)-H_{\varphi}(\mu)}{t}
 \ge \int_M \left( \frac{D\rho}{\varphi(\rho)}+D\Psi \right) (\nabla\phi) \,d\mu. \]
Precisely, the $\varphi$-relative Fisher information of $\mu=\rho\omega \in \cP_{\ac}(M,\omega)$
is defined by
\[ I_{\varphi}(\mu):=\int_M F(\nabla[\ln_{\varphi}(\rho)-\ln_{\varphi}(\sigma)])^2 \,d\mu
 =\int_M F^*\left( \frac{D\rho}{\varphi(\rho)}+D\Psi \right)^2 \,d\mu, \]
and the $\varphi$-global Poincar\'e inequality means
\[ \int_{M^{\Psi}_{\varphi}} \frac{w^2\sigma}{ \varphi(\sigma)} \,d\nu
 \le \frac{1}{K}\int_{M^{\Psi}_{\varphi}} F\bigg( \nabla \bigg( \frac{w\sigma}{\varphi(\sigma)} \bigg) \bigg)^2 \,d\nu. \]
We also remark that $W_2(\mu,\nu)$ in (i) and (ii) of Theorem~\ref{thm:inequ} can be replaced with $W_2(\nu,\mu)$
since the curvature bound $\Ric_N \ge K$ for $F$ is equivalent to that
for its reverse $\overleftarrow{F}$.
The above $\varphi$-Talagrand inequality shows the concentration of measures as in Section~\ref{sc:conc},
where the open ball $B(A,r)$ in the definition of the concentration function
$\alpha(r)$ is replaced with
\[ B^+(A,r):=\Big\{ y \in M \,\Big|\, \inf_{x \in A} d_F(x,y)<r \Big\}\ \, \text{or}\ \,
 B^-(A,r):=\Big\{ y \in M \,\Big|\, \inf_{x \in A} d_F(y,x)<r \Big\}. \]

%%%%%%%%%%%%%%%%%%%%%%%%%%%%%%%
\subsection{Gradient flow of $H_{\varphi}$}%%%%%%%%

As for the gradient flow of $H_{\varphi}$, due to the lack of the analogue of
Theorem~\ref{th:angle}, the argument in Section~\ref{sc:gf} is unavailable.
Nonetheless, one can apply the discussion in Section~\ref{sc:gf+} using a (formal)
Finsler structure of the Wasserstein space, and obtain a result corresponding
to Theorem~\ref{th:gf+}.
We remark that, however, the $K$-contraction property \eqref{eq:cont}
essentially depends on the Riemannian structure and can not be expected
in the Finsler setting (see \cite{OS2} for details).

Let $(M,F)$ be compact again.
We introduce a Finsler structure of $(\cP(M),W_2)$ similarly to Section~\ref{sc:gf+}.
Given $\mu \in \cP(M)$, define the {\it tangent space} $(T_{\mu}\cP, F_{\mu})$ at $\mu$ by
\begin{align*}
F_{\mu}(\nabla\phi) :=\bigg( \int_M F(\nabla\phi)^2 \,d\mu \bigg)^{1/2}
 \ {\rm for}\ \phi \in C^{\infty}(M), \quad
T_{\mu}\cP :=\overline{\{ \nabla\phi \,|\, \phi \in C^{\infty}(M) \}},
\end{align*}
where the closure was taken with respect to the (Minkowski) norm $F_{\mu}$.
Then we can follow the line of Section~\ref{sc:gf+} up to some computational differences.
We denote by $\cL:=(\cL^*)^{-1}:TM \lra T^*M$ the Legendre transform
in the reverse direction.

\begin{definition}[Gradient vectors]\label{df:Fgf}
Given a functional $H:\cP(M) \lra (-\infty,\infty]$ and $\mu \in \cP(M)$
with $H(\mu)<\infty$, we say that $H$ is {\it differentiable} at $\mu$
if there is $\Phi \in T_{\mu}\cP$ such that
\[ \limsup_{t \downarrow 0}\frac{H(\mu_t)-H(\mu)}{t}
 \le \int_M \cL(\Phi)(\nabla\phi) \,d\mu \]
along all minimal geodesics $(\mu_t)_{t \in [0,1]} \subset \cP(M)$ with $\mu_0=\mu$,
where $\mu_t=(\cT_t)_{\sharp}\mu$ and $\cT_t(x):=\exp_x(t\nabla\phi(x))$,
and if equality holds for $\phi \in C^{\infty}(M)$
(with $\lim_{t \downarrow 0}$ in place of $\limsup_{t \downarrow 0}$).
Such $\Phi$ is unique if it exists, and then we write $\nabla_W H(\mu)=\Phi$.
\end{definition}

\begin{proposition}\label{pr:Fgf}
Let $(M,\omega,\varphi,\Psi)$ be compact, admissible and satisfy
$\Ric_{N_{\varphi}} \ge 0$ and $\Hess\Psi \ge K$ on $M^{\Psi}_{\varphi}$
for some $K \in \R$, and fix $\mu=\rho\omega \in \cP_{\ac}(M,\omega)$
with $\mu[M^{\Psi}_{\varphi}]=1$ and $H_{\varphi}(\mu)<\infty$.
Then the following are equivalent$:$
\begin{enumerate}[{\rm (I)}]
\item $|\grad H_{\varphi}|(\mu) <\infty$,
\item $\rho \in H^1(M)$ and
\[ \Phi= \cL^* \bigg( -\frac{D\rho}{\varphi(\rho)} -D\Psi \bigg) \qquad \mu\text{-a.e.} \]
for some $\Phi \in T_{\mu}\cP$.
\end{enumerate}
Moreover, then we have $\Phi=\nabla_W[-H_{\varphi}](\mu)$ and
$F_{\mu}(\Phi)=|\grad H_{\varphi}|(\mu)$.
\end{proposition}

Note that
\[ \Phi= \cL^* \bigg( -\frac{D\rho}{\varphi(\rho)} -D\Psi \bigg)
 =\cL^* \big( D[-\ln_{\varphi}(\rho)-\Psi] \big) =\nabla[-\ln_{\varphi}(\rho)-\Psi]. \]

\begin{theorem}[Gradient flow of $H_{\varphi}$]\label{th:Fgf}
Let us suppose that $(M,\omega,\varphi,\Psi)$ is compact, admissible and satisfies $\Ric_{N_{\varphi}} \ge 0$
as well as $\Hess\Psi \ge K$ on $M^{\Psi}_{\varphi}$ for some $K \in \R$,
and let $(\mu_t)_{t \in [0,\infty)} \subset \cP_{\ac}(M,\omega)$ be a continuous curve
such that $\mu_t[M^{\Psi}_{\varphi}]=1$ and $H_{\varphi}(\mu_t)<\infty$ for all $t>0$.
Then $(\mu_t)_{t \in (0,\infty)}$ is an absolutely continuous curve satisfying
\[ \dot{\mu}_t=\nabla_W[-H_{\varphi}](\mu_t) \in T_{\mu_t}\cP \]
at a.e.\ $t \in (0,\infty)$ if and only if $(\rho_t)_{t \in [0,\infty)}$ is
a weak solution to the reverse $\varphi$-heat equation of the form
\begin{equation}\label{eq:rpme}
\frac{\del \rho}{\del t} =-\div_{\omega}\big( \rho\nabla [-\ln_{\varphi}(\rho)-\Psi] \big)
\end{equation}
with $\int_{t_0}^{t_1} F(\nabla[-\ln_{\varphi}(\rho_t)])^2 \,d\mu_t dt<\infty$
for all $0<t_0<t_1<\infty$, where $\mu_t=\rho_t \omega$.
\end{theorem}

\begin{proof}
If $\dot{\mu}_t=\nabla_W[-H_{\varphi}](\mu_t)$ a.e.\ $t$, then
Proposition~\ref{pr:Fgf} yields
$\dot{\mu}_t=\nabla [-\ln_{\varphi}(\rho_t)-\Psi] \in T_{\mu_t}\cP$ a.e.\ $t$.
Thus it follows from the continuity equation $(\ref{eq:ceq})$ that
\[ \int_0^{\infty} \int_M \frac{\del w_t}{\del t} \,d\mu_t dt
 =-\int_0^{\infty} \int_M Dw_t \big( \nabla [-\ln_{\varphi}(\rho)-\Psi] \big) \,d\mu_t dt \]
for all $w \in C_c^{\infty}((0,\infty) \times M)$,
and hence $\rho_t$ weakly solves $(\ref{eq:rpme})$.
Conversely, if $\rho_t$ is a weak solution to $(\ref{eq:rpme})$,
then the same calculation implies that $\Phi_t=\nabla [-\ln_{\varphi}(\rho_t)-\Psi]$
satisfies the continuity equation $(\ref{eq:ceq})$,
and $(\mu_t)_{t \in (0,\infty)}$ is absolutely continuous.
Therefore Proposition~\ref{pr:Fgf} shows
$\dot{\mu}_t=\Phi_t=\nabla_W[-H_{\varphi}](\mu_t)$ a.e.\ $t$.
$\qedd$
\end{proof}

We meant by the \emph{reverse $\varphi$-heat equation} the equation
with respect to the reverse Finsler structure $\overleftarrow{F}(\bv)=F(-\bv)$.
Since the gradient vector for $\overleftarrow{F}$ is written as
$\overleftarrow{\nabla}\rho=-\nabla(-\rho)$, $(\ref{eq:rpme})$ is indeed rewritten as
\[ \frac{\del \rho}{\del t} =\div_{\omega}\big( \rho\overleftarrow{\nabla}
 [\ln_{\varphi}(\rho)+\Psi] \big). \]

%%%%%%%%%%%%%%%%%%%%%%%%%%%%%%%%%%%%%%
\renewcommand{\thesection}{\Alph{section}}
\setcounter{section}{0}
\section{Appendix: Measure concentration via $u_{\varphi}$-entropy inequality}\label{sc:appen}%%%%%%%%

Let us go back to the Riemannian situation.
In Section~\ref{sc:func}, we introduced the $\varphi$-logarithmic Sobolev inequality \eqref{eq:LS}
by generalizing the relative entropy to the $\varphi$-relative entropy associated with the Bregman divergence.
Precisely, the classical logarithmic Sobolev inequality (corresponding to $\varphi_1(s)=s$) of the form
\[ \Ent_{\nu}(\mu)-\Ent_{\nu}(\nu)
 \le \frac{2}{K} \int_M \bigg| \nabla\bigg( \sqrt{\frac{\rho}{\sigma}} \bigg) \bigg|^2 \,d\nu
 =\frac{1}{2K} \int_M \bigg| \frac{\nabla\rho}{\rho}-\frac{\nabla\sigma}{\sigma} \bigg|^2 \,d\mu \]
is generalized to
\[ H_{\varphi}(\mu)-H_{\varphi}(\nu)
 \le \frac{1}{2K}\int_M \big| \nabla[\ln_{\varphi}(\rho)-\ln_{\varphi}(\sigma)] \big|^2 \,d\mu, \]
where $\mu=\rho\omega$, $\nu=\sigma\omega$ and $K$ is a positive constant.

The logarithmic Sobolev inequality has the alternative form
\[ \int_M w\ln(w) \,d\nu -\left(\int_M w  \,d\nu \right) \ln \left(\int_M w \,d\nu \right)
 \le \frac{1}{2K} \int_M \frac{|\nabla w|^2}{w} \,d\nu \]
for nonnegative measurable functions $w:M \lra [0,\infty)$.
Then the inequality
\[ \int_M u_\varphi(w) \,d\nu -u_\varphi \left(\int_M w \,d\nu \right)
 \le \frac{1}{2K} \int_M u_\varphi''(w) |\nabla w|^2 \,d\nu
 =\frac{1}{2K} \int_M \frac{|\nabla w|^2}{\varphi(w)} \,d\nu \]
obtained by replacing the function $r \longmapsto r\ln r$ (generating the relative entropy)
with $u_\varphi$ is called the \emph{$u_\varphi$-entropy inequality},
which provides a generalization of the logarithmic Sobolev inequality
different from our $\varphi$-logarithmic Sobolev inequality.
The function $\varphi$ is usually imposed to be concave,
that is equivalent to the convexity of the function
\[ (s,t)\ \longmapsto\ d_{\varphi}(s+t,t) :=u_\varphi(s+t)-u_\varphi(t)-\ln_\varphi(t)s. \]
Note that $d_\varphi$ coincides with the density function of  the  Bregman divergence $D_\varphi$.
We refer to \cite{C1} and \cite{C2} for details, where instead of $u_\varphi$ it is treated
$C^2$-strictly convex functions $\Phi$ such that $1/\Phi''$ is concave.

We demonstrated in Section~\ref{sc:conc} that the $\varphi$-Talagrand inequality leads
the $m(\varphi)$-normal concentration of measures.
In the classical case of $\varphi_1(s)=s$,
it is known that the normal concentration also follows from the logarithmic Sobolev inequality
by the Herbst argument (see, e.g., \cite[Chapter~5]{Le}).
In the same spirit, we can deduce from the $u_\varphi$-entropy inequality
the corresponding $\varphi$-normal concentration of measures.
We first recall a kind of Chebyshev's inequality for later use.

\begin{lemma}[Chebyshev's inequality]\label{lm:Cheby}
Let $w$ be a measurable function on a measure space $(X,\mu)$.
Then for any nonnegative, non-decreasing, measurable function $v$ on $\R$,
\[ \mu \left[ \left\{ x \in X \,|\, w(x) \ge t \right\} \right] \le \frac{1}{v(t)} \int_X v(w) \,d\mu \]
holds for any $t>0$ with $v(t)>0$.
\end{lemma}

We next show an auxiliary lemma.
We will normalize $\varphi$ as $\varphi(1)=1$ for simplicity,
recall that such a normalization does not change the value of $\theta_{\varphi}$
(Remark~\ref{rm:sim}).

\begin{lemma}\label{monono}
Let $\varphi:(0,\infty) \lra (0,\infty)$ be a positive concave function with $\varphi(1)=1$.
Then we have  $\theta_\varphi \le 1$ and
\[ u_\varphi(s)+a_\varphi s \ge a_\varphi \varphi(s) \ln_\varphi(s) \]
for any $s>0$, where we set $a_\varphi:=-u_\varphi(1)>0$.
\end{lemma}

\begin{proof}
It follows from the concavity of $\varphi$ that
\[ \frac{\varphi(s+t)-\varphi(s)}{t} \le \frac{\varphi(s)-\varphi(\ve)}{s-\ve} <\frac{\varphi(s)}{s-\ve} \]
for any $0<\ve<s<s+t$.
Letting $\ve \downarrow 0$ and then $t \downarrow 0$, we find
\[ \frac{s}{\varphi(s)} \cdot \limsup_{t \downarrow 0} \frac{\varphi(s+t)-\varphi(s)}{t} \le 1. \]
Since $s>0$ is arbitrary, we obtain $\theta_\varphi \le 1$.

Set  $A(s):=u_\varphi(s)+a_\varphi s- a_\varphi \varphi(s) \ln_\varphi(s)$
and observe $A(1)=0$ by the choice of $a_{\varphi}$.
Proposition~\ref{lm:cased} implies
\[ 0 \ge \lim_{s \downarrow 0}\varphi(s)\ln_{\varphi}(s)
 \ge \lim_{s \downarrow 0} s^{\delta_{\varphi}} \ell_{2-\delta_{\varphi}}(s)=0, \]
so that $\lim_{s \downarrow 0}\varphi(s) \ln_\varphi(s)=0$ and we can put $A(0):=0$.
Since the concavity of $\varphi$ ensures that the right derivative
\[ \varphi'_+(s):=\lim_{\ve \downarrow 0} \frac{ \varphi(s+\ve)-\varphi(s)}{\ve}
 \in \left[0,\frac{\varphi(s)}{s} \theta_\varphi \right] \]
is well-defined and non-increasing on $(0,\infty)$, a direct computation yields
\[ A'_+(s) :=\lim_{\ve \downarrow 0} \frac{ A(s+\ve)-A(s)}{\ve}
 =\ln_\varphi(s)\left\{ 1 -a_\varphi \varphi_+'(s) \right\}. \]
Note that \eqref{eq:ln} shows
\[ 1 -a_\varphi  \varphi'_+(1) \ge 1-a_\varphi \theta_\varphi
 =1+\theta_\varphi \int_0^1 \ln_{\varphi}(t) \,dt
 \ge 1+\theta_{\varphi} \int_0^1 \ell_{2-\theta_{\varphi}}(t) \,dt
 =1-\frac{\theta_\varphi}{2-\theta_\varphi} \ge 0. \]

For $s \ge 1$, we deduce from $\ln_\varphi(s) \ge \ln_\varphi(1)=0$ and
$\varphi_+'(s) \le \varphi_+'(1)$ that $A'_+(s) \ge 0$.
Hence we have $A(s) \ge A(1)=0$.
On $(0,1)$, since $\ln_{\varphi}<0$, $A(0)=A(1)=0$ and $\varphi'_+$ is non-increasing,
$A$ is identically zero or there is some $s_0 \in (0,1)$ such that
$A'_+ \ge 0$ on $(0,s_0)$ and that $A'_+ \le 0$ on $(s_0,1)$.
Therefore we conclude that $A \ge 0$ on $(0,1)$.
$\qedd$
\end{proof}

\begin{remark}
The condition $\theta_\varphi \le 1$ does not imply the concavity of $\varphi$.
For instance, let
\[ \varphi(s):= \left\{ \begin{array}{cl}
 \sqrt{s} & {\rm for}\ 0<s<1, \\
 s & {\rm for}\ s \ge 1.
\end{array} \right. \]
Then we have $\theta_\varphi =1$, whereas $\varphi$ is clearly not concave.
\end{remark}

Now we prove that the $u_{\varphi}$-entropy inequality implies the $\varphi$-normal concentration
for $\varphi$ as in Lemma~\ref{monono}.

\begin{theorem}[$\varphi$-normal concentration from $u_{\varphi}$-entropy inequality]\label{th:u_phi}
Take a positive concave function $\varphi:(0,\infty) \lra (0,\infty)$ such that $\varphi(1)=1$.
For a Riemannian manifold $(M,g)$ and $\nu \in \cP(M)$,
assume that there is a positive constant $K$ such that the $u_{\varphi}$-entropy inequality
\begin{equation}\label{51}
\int_M u_\varphi(w) \,d\nu -u_\varphi \left(\int_M w \,d\nu \right)
 \le \frac{1}{2K} \int_M u_\varphi''(w) |\nabla w|^2 \,d\nu
\end{equation}
holds for every nonnegative measurable function $w \in L^1(M,\nu)$ satisfying
$u_\varphi''(w)|\nabla w|^2 \in L^1(M, \nu)$.
Then for any $r>0$ we have
\[ \alpha (r)^{-1} \ge \exp_\varphi \left( -\frac{u_\varphi(1)K}{8}r^2 \right), \]
where $\alpha$ stands for the concentration function of $(M,\nu)$.
\end{theorem}

\begin{proof}
Fix arbitrary $A \subset M$ with $\nu[A] \ge 1/2$ and $r>0$.
Putting $B:=M \setminus B(A,r)$, we also assume $\nu[B]>0$ since we have $\alpha(r)=0$ if $\nu[B]=0$ for all such $A$.
Set $F_r(x):=\min\{d_g(x,A), r\}$ for $x \in M$, and observe that $F_r$ is $1$-Lipschitz.
Note also that the function
\[ G_r(x):=F_r(x) -\int_M F_r \,d\nu \]
satisfies $G_r(x) \ge r/2 $ for any $x \in B$
since $\int_M F_r \,d\nu \le r \cdot \nu[M \setminus A] \le r/2$.
Applying Chebyshev's inequality (Lemma~\ref{lm:Cheby}) to the nonnegative, non-decreasing function
\[ v_s(t) :=\exp_\varphi\left( s t-\frac{ s^2}{2 a_\varphi K} \right) \]
with $s>0$ and  $a_\varphi:=-u_\varphi(1) >0$, we have
\begin{equation}\label{eq:Pent}
\nu[B] \le \nu\left[ \left\{ x \in M \Bigm| G_r(x) \ge \frac{r}{2} \right\} \right]
 \le \frac{1}{v_s (r/2)} \int_M v_s(G_r) \,d\nu.
\end{equation}
We shall show that $I(s):=\int_M v_s(G_r) \,d\nu \ge \int_B v_s(r/2) \,d\nu >0$
is bounded above by $1$.

Set
\begin{align*}
w_s(x) &:=v_s \big( G_r(x) \big) =\exp_\varphi \left(s G_r(x)-\frac{s^2}{2a_\varphi K} \right), \\
X_s &:= \left\{ x \in M \biggm| sG_r(x) -\frac{s^2}{2a_\varphi K}  >l_\varphi \right\}.
\end{align*}
For $s\in (0,a_\varphi Kr)$ and any $x \in B$, we have
\[ sG_r(x) -\frac{s^2}{2a_\varphi K}
 \ge \frac{rs}2-\frac{s^2}{2a_\varphi K}
=-\frac{1}{2a_\varphi K} \left( s-\frac{a_\varphi K r}{2} \right)^2 +\frac{a_\varphi K r^2}{8}
 \ge 0 >l_\varphi, \]
proving $B \subset X_s$.
Let us introduce the strictly convex function $\Phi_\varphi(t):=u_\varphi(t) +a_\varphi t$
on $[0,\infty)$, and observe that
$\Phi_{\varphi} \le 0$ on $[0,1]$ and $\Phi_{\varphi}>0$ on $(1,\infty)$.
Then the inequality~\eqref{51} applied to $w=w_s$ can be rewritten as
\[ \int_{X_s} \Phi_\varphi(w_s) \,d\nu -\frac{1}{2K} \int_{X_s} \Phi_\varphi''(w_s) |\nabla w_s|^2 \,d\nu
 \le \Phi_\varphi \left( \int_{X_s} w_s \,d\nu \right). \]
Note that $w_s$ is bounded since $G_r$ is bounded by definition, and hence $w_s \in L^1(M,\nu)$.
Moreover, $u''_{\varphi}(w_s)|\nabla w_s|^2 \in L^1(M,\nu)$ is seen by
\[ \int_{X_s} \Phi''_\varphi(w_s) |\nabla w_s|^2 \,d\nu
 =\int_{X_s} s^2 \varphi(w_s) |\nabla G_r|^2 \,d\nu
 < s^2 \int_{X_s} \varphi(w_s) \,d\nu \]
for $s\in (0,a_\varphi Kr)$, where we used the fact that $|\nabla G_r| \le 1$ on whole $M$
and $|\nabla G_r| \equiv 0$ on $B$.
It follows from Lemma~\ref{monono} that
\begin{align*}
\int_{X_s} \left( \Phi_\varphi(w_s)- \frac{s^2}{2K} \varphi(w_s) \right) \,d\nu
&\ge \int_{X_s} \left( a_\varphi \varphi(w_s) \ln_\varphi(w_s) -\frac{s^2}{2K} \varphi(w_s) \right) \,d\nu \\
&= \int_{X_s} \varphi(w_s) \left( sa_\varphi G_r-\frac{s^2}{K} \right) \,d\nu 
= sa_\varphi \frac{d}{ds} \left( \int_{X_s} w_s \,d\nu \right).
\end{align*}
These together imply, as $\int_{X_s} w_s \,d\nu=\int_M w_s \,d\nu=I(s)$,
\begin{equation}\label{bibun}
sa_\varphi I'(s) < \Phi_\varphi \big( I(s) \big) \qquad \text{for}\ s \in (0,a_\varphi Kr).
\end{equation}
For $s_0 \in (0,a_{\varphi}Kr)$ chosen later, set
\[ P(s):=\exp \left( \frac1{a_\varphi} \int_{s_0}^s \frac{\Phi_\varphi'(I(t))}{t} \,dt \right),
 \qquad Q(s):=\frac{\Phi_\varphi(I(s))}{P(s)} \]
for $s \in (0,s_0]$, and observe
\[ Q'(s) =\frac{\Phi'_{\varphi}(I(s))}{P(s)}
 \bigg\{ I'(s)-\frac{\Phi_{\varphi}(I(s))}{sa_{\varphi}} \bigg\}. \]
Then we deduce from \eqref{bibun} that $Q'(s)=0$ if and only if $\Phi'_\varphi(I(s))=0$.

Assume that $\sup_{s \in (0,a_{\varphi}Kr)}I(s)>I(0)=1$ and choose $s_0 \in (0,a_{\varphi}Kr)$
such that $I(s_0)>1$ and
\[ c:=\sup_{s \in (0,s_0]} \Phi_\varphi' \big( I(s) \big) \in (a_\varphi,2a_{\varphi}) \]
(note that $\Phi'_{\varphi}(I(0))=a_{\varphi}$).
Then we have
\begin{equation}\label{eq:P(s)}
P(s) \ge \exp \left( \frac{c}{a_\varphi} \int_{s_0}^s \frac{1}{t} \,dt \right)
 =\left( \frac{s}{s_0} \right)^{c/a_{\varphi}},\qquad s \in (0,s_0].
\end{equation}
Moreover, since the convexity of $\Phi_\varphi$ and $\Phi_\varphi(0)=0$ imply
$I(s) \Phi_\varphi' (I(s)) \ge \Phi_\varphi (I(s))$,
we find $\Phi'_{\varphi}(I(s_0))>0$ and hence $Q'(s_0)<0$.
Note that there does not exist $s \in(0,s_0)$ such that $Q' \le 0$ on $(s,s_0)$ as well as $Q'(s)=0$,
since then $I(s)\Phi'_{\varphi}(I(s))=0 \ge \Phi_{\varphi}(I(s))$ and $Q(s_0) \le Q(s) \le 0$
which contradicts $I(s_0)>1$.
Thus $Q'<0$ on $(0,s_0)$, and by \eqref{eq:P(s)}
\[ Q(s_0) \le \limsup_{s \downarrow 0} Q(s)
 \le s_0^{c/a_{\varphi}} \limsup_{s \downarrow 0} \frac{\Phi_\varphi(I(s))}{s^{c/a_\varphi}}. \]
Now, since
\begin{align*}
I'(0) = \int_M G_r \,d\nu=0, \qquad
\Phi_{\varphi}\big( I(s) \big) = \Phi_{\varphi}(1) +s\Phi'_{\varphi}(1)I'(0) +O(s^2) =O(s^2)
\end{align*}
and $c<2a_{\varphi}$, it holds $\lim_{s \downarrow 0}s^{-c/a_{\varphi}}\Phi_{\varphi}(I(s))=0$.
This means $Q(s_0) \le 0$ and hence $I(s_0) \le 1$, which is a contradiction.
We therefore obtain $I(s) \le I(0)=1$ for any $s \in (0,a_\varphi Kr)$ as desired.

Hence we deduce from \eqref{eq:Pent} that
$\nu[B] \le v_s(r/2)^{-1}$
for any $s \in (0,a_\varphi Kr)$.
Choosing $s=a_\varphi Kr/2$ and taking the supremum in $A$, we conclude that
\[ \alpha(r) \le \frac{1}{\exp_\varphi(a_\varphi Kr^2/8)}. \]
$\qedd$
\end{proof}

\begin{remark}\label{rm:BG}
Bolley and Gentil~\cite{BG} showed that if a probability measure on $\R^n$ satisfies $\CD(K,\infty)$ with $K>0$,
then it satisfies the $u_{\varphi}$-entropy inequality~\eqref{51} with the same constant $K$.
We remark that the condition $\CD(K,\infty)$ leads the normal concentration which is stronger than
the $\varphi$-normal concentration for $\theta_\varphi<1$
(since $\exp_{\varphi}(r)^{-1} \ge e_{2-\theta_{\varphi}}(r)^{-1} \ge e^{-r}$ by \eqref{eq:exp}),
whereas there exists a probability measure which satisfies \eqref{51} and does not satisfy $\CD(K,\infty)$.
See~\cite[Theorem~2]{LO} for details, where they proved that the probability measure on $\R^n$ of the form
\[ d\mu_a(x):=\left( \frac{a}{2\Gamma(1/a)} \right)^n \exp(-|x|^a) \,d\cL^n(x) \]
with $a \in [1,2)$ satisfies the $u_{\varphi_m}$-entropy inequality for $m\in(1,2]$,
while the concentration function $\alpha(r)$ of $\mu_a$ is dominated by
$\exp(-r^2/3)$ (resp.\ $\exp(-r^a/3)$) for $r<1$ (resp.\ $r>1$).
\end{remark}

{\small%%%%%%%%%%%%%%%%%%%%%%%%%%%%%

}
\end{document}